\documentclass[11pt]{amsart}
\usepackage[a4paper,margin=1.5in]{geometry}
\usepackage{amsmath,amssymb,amsthm,mathtools}
\usepackage{lmodern}
\usepackage{enumitem}
\usepackage{xcolor}
\usepackage{tikz-cd}
\usepackage{mathrsfs}
\usepackage[expansion=false]{microtype}
\usepackage[colorlinks=true,linkcolor=red!70!black,citecolor=green!50!black,urlcolor=magenta]{hyperref}
\RequirePackage{etex}
\usepackage{centernot}
\usepackage{mathrsfs}

\usepackage{forest}
\usepackage{graphicx}
\usepackage{verbatim}
 \usepackage{tkz-euclide}
 \usepackage{adjustbox}
\usepackage[english]{babel}
\usepackage{relsize}

\usepackage{multicol}

\usepackage{tikz}
\usepackage{tikz-cd}

\usepackage{hyperref}
\usepackage{xcolor}

\calclayout

\theoremstyle{plain}
\newtheorem{Th}{Theorem}[section]
\newtheorem{Lem}[Th]{Lemma}
\newtheorem{Conj}[Th]{Conjecture}
\newtheorem*{Lem*}{Lemma}

\newtheorem{Cor}[Th]{Corollary}
\newtheorem{Prop}[Th]{Proposition}

\usepackage{epsfig}
\usepackage{stmaryrd}
\newtheorem*{theorem*}{Theorem}

\usepackage{breqn}
 \theoremstyle{definition}
 \newtheorem{Not}[Th]{Notation}
 
 \newtheorem*{Prob*}{Problem}
 \newtheorem*{Ex*}{Example}
 \newtheorem*{Def*}{Definition}
 \newtheorem{Rem}[Th]{Remark}

\newtheorem{Def}[Th]{Definition}

\RequirePackage{etex}

\newcommand{\At}{\mathrm{At}}

\newcommand{\ch}{\mathrm{ch}}
\newcommand{\End}{\mathrm{End}}

\newcommand{\Ext}{\mathrm{Ext}}

\newcommand{\GL}{\mathrm{GL}}

\newcommand{\HH}{\mathrm{HH}}
\newcommand{\HKR}{\mathrm{HKR}}
\newcommand{\Hom}{\mathrm{Hom}}
\newcommand{\HT}{\mathrm{HT}}
\newcommand{\Id}{\mathrm{Id}}

\newcommand{\Kum}{\mathrm{Kum}}
\newcommand{\NS}{\mathrm{NS}}
\newcommand{\ob}{\mathrm{ob}}
\newcommand{\PGL}{\mathrm{PGL}}

\newcommand{\RHom}{\mathrm{RHom}}
\newcommand{\rk}{\mathrm{rk}}
\newcommand{\SL}{\mathrm{SL}}

\newcommand{\Spec}{\mathrm{Spec}}

\newcommand{\td}{\mathrm{td}}
\newcommand{\Tr}{\mathrm{Tr}}
\newcommand{\Vol}{\mathrm{Vol}}

\newcommand{\bC}{\mathbb{C}}

\newcommand{\bR}{\mathbb{R}}

\newcommand{\bP}{\mathbb{P}}
\newcommand{\bQ}{\mathbb{Q}}

\newcommand{\bZ}{\mathbb{Z}}

\newcommand{\sS}{\mathscr{S}}

\newcommand{\cC}{\mathcal{C}}

\newcommand{\cE}{\mathcal{E}}
\newcommand{\cF}{\mathcal{F}}
\newcommand{\cG}{\mathcal{G}}
\newcommand{\cH}{\mathcal{H}}

\newcommand{\cK}{\mathcal{K}}
\newcommand{\cL}{\mathcal{L}}

\newcommand{\cO}{\mathcal{O}}
\newcommand{\cP}{\mathcal{P}}

\newcommand{\cT}{\mathcal{T}}

\newcommand{\cX}{\mathcal{X}}
\newcommand{\cY}{\mathcal{Y}}
\newcommand{\cZ}{\mathcal{Z}}

\title[The LSC for $4d$-dimensional Kummer-type hyper-K\"ahlers]{The Lefschetz standard conjectures for $4d$-dimensional Kummer-type hyper-K\"{a}hler varieties}

\date{}

\author{Josiah Foster}

\begin{document}
\begin{abstract}
    We prove the Lefschetz standard conjectures for $4d$-dimensional projective varieties of generalized Kummer deformation type. Along the way, we prove a generalization of a theorem of Buchweitz and Flenner in the analytic setting. 
\end{abstract}
\maketitle
\tableofcontents

\section{Introduction}
\subsection{The standard conjectures}
  Let $X$ be a smooth projective variety of dimension $n$ over $\mathbb{C}$, and let the cohomological operator $L_{\alpha}$ be defined by taking the cup product with a polarization $\alpha\in H^{1, 1}(X, \mathbb{Q})$. By the Hard Lefschetz theorem, there exists an isomorphism
\[
L_{\alpha}^{n-k}: H^k(X, \mathbb{Q})\xrightarrow{\sim} H^{2n-k}(X, \mathbb{Q})
\]
for all $k\leq n$. Over $\bC$, the following statement implies all of Grothendieck's standard conjectures. In the context of this paper, we refer to it as \textit{the} Lefschetz standard conjecture (LSC):

\begin{Conj}[{\cite[p. 196]{Gr}}]\label{BX}
For each $k\leq n$, there exists an algebraic self-correspondence \[[\mathcal{Z}]\in H^{2k}(X\times X, \mathbb{Q}),\] arising from a codimension-$k$ cycle $\mathcal{Z}\in \mathrm{CH}^k(X\times X)_{\mathbb{Q}}$, such that the isomorphism 
\begin{equation}\label{inverse hard lefschetz isomorphisms}
[\mathcal{Z}]^*: H^{2n-k}(X, \mathbb{Q})\xrightarrow{\sim} H^{k}(X, \mathbb{Q})
\end{equation}
is the inverse of the isomorphism $L^{n-k}_{\alpha}$.
\end{Conj}
\begin{Rem}
    Conjecture \ref{BX} is independent of the choice of polarization on $X$ and is in fact equivalent to the algebraicity of the adjoint Lefschetz operator $\Lambda_{\alpha}$ of \S \ref{standard conjecture subsection}, by \cite[Lemma 6]{Charles}. Over $\mathbb{C}$, Conjecture \ref{BX} implies that homological and numerical equivalence coincide and that the K\"{u}nneth components of the diagonal $\Delta_X\subset X\times X$ are algebraic (see for example \cite{Kleiman68}). It has important applications in the theory of motives (see \cite[Chapter 5]{Andre}), and in degree $2$ it implies a weakened formulation of the period-index conjecture (\cite{dJP}). In addition, the LSC implies the variational Hodge conjecture by \cite[Theorem 7]{VoisinHodge}.
\end{Rem}

There are few classes of projective varieties for which the LSC is known to hold. One obvious class consists of varieties $X$ for which the cycle class map $\mathrm{CH}(X)_{\mathbb{Q}}\rightarrow H^*(X, \mathbb{Q})$ is an isomorphism (e.g. Grassmanians, flag varieties, etc.). The conjecture also holds for abelian varieties by \cite{Lieberman}. Charles and Markman in \cite{CharlesMarkman} prove the LSC for irreducible holomorphic symplectic varieties deformation equivalent to Hilbert schemes of points on a $K3$ surface. An alternative proof of their result is known (\cite[Theorem 1.5, Lemma 1.6]{Markman22}, cf. Corollary \ref{Derived equivalent LSC}). In \cite{Arapura}, Arapura proves the LSC for moduli spaces of stable vector bundles over a smooth projective curve and for the Hilbert scheme of points on any smooth projective surface, as well as for uniruled threefolds and unirational fourfolds. The results of \cite{Foster} generalize Arapura's result by showing that the LSC holds for projective varieties deformation equivalent to a moduli space of Gieseker-stable sheaves on an abelian surface with primitive, positive, and effective Mukai vector. As a corollary, the LSC holds for projective varieties deformation equivalent to a generalized Kummer variety in certain cohomological degrees. For important results on the standard conjectures for Lagrangian fibered IHSMs, see \cite{ACLS}.

\subsection{Kummer-type irreducible holomorphic symplectic manifolds}

\begin{Def}
    An \textit{irreducible holomorphic symplectic manifold (IHSM)} (or compact hyper-K\"{a}hler manifold) is a simply connected compact K\"{a}hler manifold such that $H^{2, 0}(X)$ is spanned by a unique, everywhere-nondegenerate holomorphic $2$-form.
\end{Def}

To date, there are four known deformation types of IHSMs. Two were constructed in \cite{Beauville} and arise in each even complex dimension. They are varieties deformation equivalent to Hilbert schemes of points on $K3$ surfaces and deformation equivalent to generalized Kummer varieties. Two additional examples in dimensions $10$ and $6$ were constructed in \cite{OG10} and \cite{OG6} respectively. They arise as desingularizations of certain moduli spaces of sheaves on symplectic surfaces. 

For $X$ an IHSM, the second integral cohomology $H^2(X, \mathbb{Z})$ carries an integral, symmetric, primitive, non-degenerate bilinear pairing $q_X$ of signature $(3, b_2-3)$, known as the Beauville-Bogomolov-Fujiki (BBF) pairing \cite[Theorem 5]{Beauville}, which is defined by the property that 
for all $\alpha\in H^2(X, \mathbb{C})$, there is the equation
\begin{equation}\label{BBF form}
 \int_X \alpha^{2n} = c_Xq_X(\alpha)^n
\end{equation}
where $c_X\in \mathbb{Q}_{> 0}$ is the Fujiki constant. This is a generalization of the pairing on the second cohomology of a $K3$ surface.

In comparison to $K3$ surfaces, general IHSMs exhibit a weakened version of the global Torelli theorem (see \cite{Huybrechts1} and \cite{VerbitskyTorelli}).

\begin{Def}
    Let $A$ be an abelian surface, and $A^{[n+1]}$ the Hilbert scheme of zero-dimensional length-$(n+1)$ subschemes of $A$. For $h: A^{[n+1]}\rightarrow A$ the composition of the Hilbert-Chow morphism and summation in the group $A$, the fiber $\Kum_n(A): = h^{-1}(0)$ is $2n$-dimensional IHSM known as a \textit{generalized Kummer variety}. An IHSM deformation-equivalent to $\Kum_n(A)$ is said to be of \textit{generalized Kummer deformation type}.
\end{Def}
Note that the abelian surface $A$ has a four-dimensional deformation space, while $\Kum_n(A)$ deforms in its five-dimensional period domain (as $b_2(\Kum_n(A)) = 7$). Therefore, a generic variety of generalized Kummer deformation type is not the Kummer variety arising from an abelian surface. 
Let us describe the connection between generalized Kummer varieties and moduli spaces of sheaves on an abelian surface $A$. 

\begin{Def}
    The \textit{Mukai lattice} of $A$ consists of the following data: 
    \begin{enumerate}
        \item The $8$-dimensional lattice $(H^{ev}(A, \mathbb{Z}), \langle\cdot , \cdot \rangle)$, where $H^{ev}(A, \mathbb{Z}): = H^0(A, \mathbb{Z})\oplus H^2(A, \mathbb{Z})\oplus H^4(A, \mathbb{Z})$ is the ring of even cohomology, and for $v = (v_0, v_2, v_4)$ and $w = (w_0, w_2, w_4)$ in $H^{ev}(A, \mathbb{Z})$,
        \[
        \langle v, w\rangle: = \int_A v_2w_2-v_0w_4-v_4w_0
        \]
        is the \textit{Mukai pairing}.

        \item A weight-two Hodge structure on $H^{ev}(A, \mathbb{C})$ defined by
        \\
        \begin{itemize}
            \item[] $H^{0, 2}(H^{ev}(A, \mathbb{C})): = H^{0, 2}(A),$
            \item[] $H^{1, 1}(H^{ev}(A, \mathbb{C})): = H^{0,0}(A)\oplus H^{1,1}(A)\oplus H^{2,2}(A),$
            \item[]  $H^{2,0}(H^{ev}(A, \mathbb{C})): = H^{2, 0}(A)$.
             \end{itemize} \end{enumerate}
\end{Def}

\begin{Th}[{\cite[Theorem 0.2.]{Yoshiokaabelian}}]
    Let $A$ be an abelian surface and let $v\in H^{ev}(A, \mathbb{Z})$ be a primitive Mukai vector that represents a sheaf. Let $H$ be a $v$-generic polarization and assume $v^2\geq 6$. Then a fiber over $a\in A\times \widehat{A}$ of the Albanese morphism 
    \[
    \mathrm{alb}_v: \mathcal{M}_H(v)\xrightarrow[]{\widehat{\det}\times \det} A\times \widehat{A}
    \]
    is a variety $\Kum_a(v)$ of generalized Kummer deformation type.
\end{Th}
\subsection{The main results}

\begin{Th}[{Theorem \ref{maintheorem}}]\label{maintheoremintroduction}
Conjecture \ref{BX} holds for all projective $4d$-complex dimensional IHSMs of generalized Kummer deformation type. 
\end{Th}

Theorem \ref{maintheoremintroduction} is the result of the following statement, proven by the end of the paper, for general IHSMs:

\begin{Th}[{Theorem \ref{Th: summing up paper results}}]\label{summing up paper results introduction}
    Let $X$ and $Y$ be projective $2n$-dimensional deformation-equivalent IHSMs. Assume that if $n$ is even, then $b_2$ is odd.\footnote{This condition holds for all known IHSMs.} Assume that there exists a Fourier-Mukai equivalence $\Phi_\cF: D^b(X)\rightarrow D^b(Y)$ such that $\cF$ satisfies the following:
    \begin{enumerate}
        \item $\cF$ is locally free;
        \item $\cF$ is $\mu_{\omega}$-stable for some K\"ahler class $\omega$.
    \end{enumerate}
    Then the LSC holds for any projective variety deformation equivalent to $X$.
\end{Th}

The correspondences satisfying the assumptions of Conjecture \ref{BX} arise from Fourier-Mukai kernels of derived equivalences between Kummer-type IHSMs. Deformation equivalent IHSMs $X$ and $Y$ that are derived equivalent with a Fourier-Mukai kernel of nonzero rank satisfy the LSC (Corollary \ref{Derived equivalent LSC}). This fact follows from the main result in \cite{Taelman}, which shows that the Hodge isometry associated to such a derived equivalence conjugates the LLV Lie algebra $\mathfrak{g}(X)$ to $\mathfrak{g}(Y)$. The main result of \cite{Magni} provides a derived equivalence between $4d$-dimensional Kummer-type IHSMs $\Kum_n(A)$ and $\Kum_n(\widehat{A})$, with the assumption that there exists a polarization $A\rightarrow \widehat{A}$ satisfying certain numerical criteria. The Fourier-Mukai kernels of these equivalences are images of an inflated semihomogeneous vector bundle under the Bridgeland-King-Reid (BKR) correspondence. Our main objective is then to deform these Fourier-Mukai equivalences.

The deformation theory of sheaves on IHSMs has been well-studied. The beautiful work in \cite{OGModularK32}, \cite{OGKum}, and \cite{Markman2} provides a description of special sheaves on IHSMs known as \textit{modular sheaves}. A modular sheaf $(\mathcal{F}, H)$ is a slope-stable reflexive sheaf $\mathcal{F}$ with respect to an ample class $H$ such that there exists an open subcone of the ample cone containing $H$ such that $\mathcal{F}$ remains slope-stable with respect to every class $H^{\prime}$ in this subcone. This is equivalent to the orthogonal projection of the characteristic class $\kappa_2(\mathcal{F})$ introduced by Markman in \cite{Markman20} (see Definition \ref{kappa class definition}) being a rational multiple of the class of the BBF form in the Verbitsky component of $H^4(X, \mathbb{Q})$ (cf. \cite[Definition 1.1]{OGModularK32}). Such sheaves are related to hyperholomorphic sheaves studied by Verbitsky in \cite{VerbitskyHH}; Markman in \cite{Markman1} provides conditions for the existence of modular sheaves on IHSMs and further a proof of the fact that that their sheaf of Azumaya algebras $\mathcal{E}nd(F)$ deforms with every K\"{a}hler deformation. 

Our situation is different than that considered in these papers, as we wish to deform a sheaf over a \textit{product} of IHSMs using the calculus of twistor lines developed in \cite{VerbitskyHH}. The concept of modularity does not readily transfer over, as the definition is given in terms of the Beauville-Bogomolov-Fujiki form, and it is also not a priori clear what it means to be projectively hyperholomorphic over a product. Slope-stability also presents an issue, as we cannot expect the sheaf to remain slope stable with respect to arbitrary K\"{a}hler deformations of the factors. To this end, in section \ref{deforming derived equivalences section}, we deform our sheaf over a ``moduli space of rational Hodge isometries," i.e. a moduli space of products of IHSMs whose cohomology rings are rationally Hodge-isometric, following Markman in \cite{Markman22}. In lieu of modularity, the natural conditions to consider are the (twisted-)semiregularity theorems of Buchweitz-Flenner, Perry, and Pridham.
\subsection{Strategy of proof}
Let $\cF^K$ denote the Fourier-Mukai kernel on a product of Kummer varieties $K\times \widehat{K}$, obtained from Magni's construction in \cite{Magni}. It is locally free (Lemma \ref{local freeness}). The proof of Theorem \ref{maintheoremintroduction} involves four steps:
\begin{enumerate}
    \item[(i)] For $X$ a $K$-trivial variety, the kernel of any fully faithful functor $D^b(X)\rightarrow D^b(Y)$ is semiregular by \cite[Lemma 8.2, Corollary 8.5]{Rienks}. In section \ref{section 4 semiregularity bundles} we include a proof of a more general statement that recovers this result. Perry and Pridham's twisted Buchweitz-Flenner theorem \cite[Theorem 1.1]{Perry} (cf. \cite[Remark 2.26]{Pridham}), in the analytic form of Theorem \ref{analytic Buchweitz-Flenner conjecture}, implies that $\cF^{K}$ deforms as a twisted vector bundle along every deformation of $K\times \widehat{K}$ for which $\kappa(\cF^K)$ stays of Hodge-type. 
    \item[(ii)] The characteristic class $\kappa(\cF^K)$ remains of Hodge-type over a locus $U\subseteq \mathfrak{M}_{\psi}^0$ in a connected component of the moduli space of marked Kummer-type IHSMs on which the isometry \[\psi: = \tilde{H}_0\left[\kappa(\cF^K)\sqrt{\td_{K\times \widehat{K}}}\right]_*: \tilde{H}(K, \bQ)\rightarrow \tilde{H}(\widehat{K}, \bQ)\] defined in \eqref{H tilde map} remains a Hodge-isometry of rational cohomology (Proposition \ref{hodge isometry psi of hodge type}). We demonstrate that twisted semiregularity implies that $\cF^K$ deforms over $U$.
    \item[(iii)] We demonstrate that $\cF^K$ is slope-stable with respect to ample classes near the pullback of a Chow class, following Stapleton in \cite{Stapleton}. Slope-stability is an open condition for a family of compact K\"ahler manifolds (\cite[Theorem 5.1.1]{LubkeTeleman}). The same result holds for twisted sheaves (\cite{WrayThesis}). At points in $U\subseteq \mathfrak{M}_{\psi}^0$ for which both factors have trivial Picard group, the deformed kernel $\tilde{\cF^K}$ has no subsheaves of intermediate rank, and hence is stable with respect to every K\"ahler class. In addition, $\psi$ maps a subcone of the K\"ahler cone of the first factor to a subcone of the K\"ahler cone of the second factor at that point (Corollary \ref{missing-link}).

    \item[(iv)] From such a point, results of Markman \cite{Markman22} and Verbitsky \cite{VerbitskyHH} transport the sheaf, and hence the induced Hodge-isometry, to every point of $\mathfrak{M}^0_{\psi}$ by way of diagonal twistor lines. At each projective point, the $\kappa$-class is algebraic, and therefore the LSC holds for both factors by the LLV argument of Corollary \ref{Derived equivalent LSC}.
\end{enumerate} 

\subsection{Outline of the paper}
In section \ref{semiregularity section} we introduce semiregularity, the Buchweitz-Flenner theorems, and their generalizations in the form in which we use them. In section \ref{semihomogeneous section}, we introduce the construction of the vector bundles on products of Kummer-type IHSMs that are of interest to us. In section \ref{section 4 semiregularity bundles}, we study the equivariance properties of these bundles and their semiregularity. In section \ref{section 6 LSC and derived equivalences} we make concrete the relationship between nontrivial derived equivalences and the Lefschetz standard conjecture. In section \ref{section 6 slope-stability} we prove the slope-stability of $\cF^K$ and its deformations following Stapleton in \cite{Stapleton}, and in section \ref{deforming derived equivalences section} we deform our bundles over the moduli space of marked Hodge-isometric generalized Kummer varieties. 
\section*{Acknowledgments}
I sincerely thank Eyal Markman for suggesting this problem, for all of his support, and for numerous discussions. I graciously acknowledge helpful conversations with Nicolas Addington, Yuxuan Yang, and Ben Tighe, as well as correspondence with Jonathan Pridham. I thank Nicolas Addington for pointing out the reference \cite{WrayThesis}, Richard Thomas for pointing out the reference \cite{Rienks}, and Alex Perry for sharing with me an early draft of \cite{Perry}. This work was partially supported by NSF grant DMS-2039316.
\begin{Rem}
    Throughout, we work over the ground field $\mathbb{C}$.
\end{Rem}
\section{Semiregularity}\label{semiregularity section}
We recall the property of semiregularity from \cite{Buchweitz-Flenner}. Let $X$ be a smooth projective variety.
Throughout this section, let $\mathcal{I}_{\Delta}\subset \mathcal{O}_{X\times X}$ denote the ideal sheaf of the diagonal of $X\times X$. We let $\mathcal{F}\rightarrow X$ be a coherent sheaf. In the sequel we will specify when we assume that $\mathcal{F}$ is locally free. There is a short exact sequence of coherent sheaves
\begin{equation}\label{differential one forms short exact sequence}
0\rightarrow \Omega_X^1\rightarrow \mathcal{O}_{X\times X}/\mathcal{I}^2_{\Delta}\rightarrow \mathcal{O}_X\rightarrow 0,\end{equation} where we note that $\mathcal{O}_{X\times X}/\mathcal{I}^2_{\Delta}$ defines the structure sheaf of the first order infinitesimal neighborhood of $\Delta$.
Let $p_i: X\times X\rightarrow X$ denote the projection onto the $i$th factor. Twisting \eqref{differential one forms short exact sequence} by $p_1^*\mathcal{F}$  yields an exact sequence
$$
0\rightarrow \Omega_X^1 \otimes p_1^*\mathcal{F}\rightarrow \mathcal{O}_{X\times X}/\mathcal{I}^2_{\Delta}\otimes_{\mathcal{O}_{X\times X}}p_1^*\mathcal{F}\rightarrow p_1^*\mathcal{F}\rightarrow 0,
$$
and thus defines an extension class \begin{equation}
\At(\mathcal{F})\in \mathrm{Ext}^1_X(\mathcal{F}, \mathcal{F}\otimes \Omega_X^1),
\end{equation} known as the \textit{Atiyah class}. We observe that the Atiyah class obstructs the existence of global holomorphic connections $\mathcal{F}\rightarrow \mathcal{F}\otimes \Omega^1_X$ on $\mathcal{F}$. Taking the $i$th cup product of the Atiyah class with itself, one defines the $i$th power of the Atiyah class
\begin{equation} \At^i(\mathcal{F})\in \mathrm{Ext}^i_X(\mathcal{F}, \mathcal{F}\otimes \Omega^i_X).
\end{equation}

\noindent Note that $\At^0(\mathcal{F}) = \mathrm{Id}$. We denote the exponential Atiyah class by
\begin{equation}
    \mathrm{exp}\,  \At(\mathcal{F}) = \mathrm{Id}+ \At_{\mathcal{F}}+ \frac{\At_{\mathcal{F}}^2}{2}+\cdots.
\end{equation}
Define the map $\sigma_q$ as the composition
\begin{equation}\label{qth semiregularity map}
\sigma_q: \mathrm{Ext}^2_X(\mathcal{F}, \mathcal{F})\xrightarrow{\cdot(-\At(\mathcal{F})^q/q!)} \mathrm{Ext}_X^{q+2}(\mathcal{F}, \mathcal{F}\otimes \Omega_X^q)\xrightarrow{H^{q+2}(\Tr)} H^{q+2}(X, \Omega_X^q),
\end{equation}
and the \textit{semiregularity map}
\begin{equation}\label{semiregularity map}
    \sigma: = (\sigma_0, \dots, \sigma_{n-2}): \mathrm{Ext}^2(\mathcal{F}, \mathcal{F})\longrightarrow \prod_{q=0}^{n-2}H^{q+2}(X, \Omega^q_X),
\end{equation}
where $n$ is the dimension on $X$. Here, $\Tr$ denotes the trace map 
\begin{equation}
    \Tr(=\Tr^0): \cH om_{\cO_X}(\mathcal{F}, \mathcal{F})\longrightarrow \mathcal{O}_X,
\end{equation}
which in turn induces maps on hypercohomology and therefore determines trace maps 
\begin{equation}
    \Tr^q: \mathrm{Ext}^q(\mathcal{F}, \mathcal{F})\longrightarrow H^q(X, \mathcal{O}_X);
\end{equation}
see \cite[Proposition 1]{Artamkin}. Assuming that $\rk(\mathcal{F})>0$ and $\cF$ is locally free, we note that the existence of the trace map provides a section for the exact sequence 
\[0\rightarrow \mathcal{O}_X\rightarrow \mathcal{E}nd(\mathcal{F})\rightarrow \mathcal{E}nd(\mathcal{F})/\mathcal{O}_X\rightarrow 0\] and therefore $\mathcal{E}nd(\mathcal{F})$ splits as a direct sum $\mathcal{O}_X\oplus \mathcal{E}nd(\mathcal{F})/\mathcal{O}_X$. As a result, we have a decomposition on the cohomology \begin{equation}H^q(\mathcal{E}nd(\mathcal{F}))\cong H^q(\mathcal{O}_X)\oplus H^q(\mathcal{E}nd(\mathcal{F})/\mathcal{O}_X)).\end{equation}
\begin{Def}\label{Definition of semiregular}
    A sheaf $\mathcal{F}$ is said to be \textit{semiregular} if the semiregularity map (\ref{semiregularity map}) is injective.
\end{Def}

\subsection{Buchweitz-Flenner theorems} Semiregularity is a sufficient condition for a sheaf to deform sideways with its Chern character in a sublocus of the Hodge locus. That is, for $X$ a projective variety and a semiregular sheaf $E\rightarrow X$, it follows from \cite[Section 5]{Buchweitz-Flenner} that the pair $(X, E)$ deforms in a sublocus of the Hodge locus of $\ch(E)$. 
Let $\mathcal{T}_X$ be the tangent sheaf of $X$. By \cite[Corollary 4.3]{Buchweitz-Flenner}, there exists a commutative diagram for each $q$:
\begin{equation}\label{regular semiregular diagram}
\begin{tikzcd}
	{H^1(X, \mathcal{T}_X)} & {} & {\mathrm{Ext}^2_X(E, E)} \\
	\\
	& {H^{q+2}(X, \Omega^q_X)}
	\arrow["{\langle *,\, -\At(E)\rangle}", from=1-1, to=1-3]
	\arrow["{\langle*,\, \ch_{q+1}(E)\rangle}"', from=1-1, to=3-2]
	\arrow["{\sigma_q}", from=1-3, to=3-2]
\end{tikzcd}
\end{equation}
\begin{Rem}\label{Griffiths transversality}
    Note that by Griffiths transversality (see \cite[Theorem 5.6]{VoisinHT} and the discussion in \cite[Section 5.4]{Buchweitz-Flenner} for the precise interpretation) the horizontal class extending $\ch_{q+1}(E)$ remains of Hodge-type if and only if for a deformation $\alpha\in H^1(X, \mathcal{T}_X)$, we have $\langle \alpha, \ch_{q+1}(E)\rangle = 0$. 
\end{Rem}
Taking the above commutative diagram and Remark \ref{Griffiths transversality} into account, Buchweitz and Flenner prove the following theorem:

\begin{Th}[{\cite[Section 5]{Buchweitz-Flenner}}]\label{Buchweitz-Flenner theorem} For $B$ an analytic disk, let $\cY\xrightarrow{\pi} B$ be a smooth and proper family such that a coherent sheaf $\cF_0$ over $Y_0: = \pi^{-1}(0)$ is semiregular. Assume that $\ch(\cF_0)$ remains of Hodge-type over $B$, in other words, the parallel transport of $\ch_p(F_0)$ lies in $F^pR^{2p}\pi_*\Omega^{\bullet}_{\cY/B}$ for each point of $B$. Then there exists a neighborhood $U\subset B$ containing $0$ such that $\cF_0$ extends over $\pi^{-1}(U)$ to a coherent sheaf $\cE$ which is flat over $U$.
    
\end{Th}
\subsection{Twisted sheaves and semiregularity}

 For a locally free sheaf $\mathcal{F}$, the projective bundle $\mathbb{P}(\mathcal{F})$ also deforms over a family over a locus in the base manifold, provided that we relax the condition that $c_1(\mathcal{F})$ remains of Hodge-type $(1, 1)$. Let $\mathcal{U}: = \{U_i\}_{i\in I}$ be an open cover of $X$. We refer to a coherent sheaf $\mathcal{B}$ of nonzero rank $r$ on $X$ which is twisted by the $2$-cocycle $\theta\in \mathcal{Z}(X, \mu_r)$, where $\mu_r$ is the local system of $r$th roots of unity, as a $\theta_r$\textit{-twisted sheaf} (we refer to \cite{Caldararu2} and \cite[Section 2.2]{Markman22} for a detailed discussion). Every such projective bundle lifts to a locally free twisted sheaf. We  wish to extend the formulation of Theorem \ref{Buchweitz-Flenner theorem} to account for twisted sheaves. As is studied in \cite{TodaFM}, first order deformations of $D^b(X)$ are parametrized by the second Hochschild cohomology $\HH^2(X): = \mathrm{Hom}_{D^b(X\times X)}(\mathcal{O}_{\Delta_X}, \mathcal{O}_{\Delta_X}[2])$. As $\HH^i(X)$ parametrizes natural transformations from the identity endofunctor of $D^b(X)$ to the shift of the identity endofunctor by $i$, evaluation at an element $\mathcal{F}\in D^b(X)$ yields a graded homomorphism 
\begin{equation}\label{obstruction map}
    \ob_{\mathcal{F}}: \HH^*(X)\longrightarrow \mathrm{Ext}^*(\mathcal{F}, \mathcal{F}),
\end{equation}
the kernel of which parametrizes generalized deformations of $\mathcal{F}$. For 
\begin{equation}
    \HT^k(X): = \bigoplus_{i+j = k} H^i(X, \bigwedge^j\mathcal{T}_X),
\end{equation}
we have the \textit{Hochschild-Kostant-Rosenberg (HKR)} isomorphism 
\begin{equation}
    \HKR: \HH^*(X)\xrightarrow{\sim}\HT^*(X),
\end{equation}
(see \cite[Theorem 2.5]{TodaFM}). Note that an element $\xi\in \HT^2(X)$ decomposes into $\xi = \alpha+\beta+\gamma$, where $\beta\in H^1(X, \mathcal{T}_X)$ corresponds to a deformation of the complex structure on $X$, $\alpha\in H^2(X, \mathcal{O}_X)$ corresponds to a gerby deformation of $X$, and $\gamma\in H^0(X, \bigwedge^2\mathcal{T}_X)$ corresponds to a noncommutative deformation of $X$. For an element $\xi\in \HT^2(X)$, there exists the derived category $D^b(X, \xi)$ of the  $\mathbb{C}[\epsilon]/(\epsilon^2)$-linear abelian category $\mathrm{Coh}(X, \xi)$ (\cite[Definition 4.4]{TodaFM}) and the natural embedding $i_*: D^b(X)\hookrightarrow D^b(X, \xi)$. Toda proves that a necessary and sufficient condition for an element $\mathcal{F}$ in $D^b(X)$ to deform in the direction $\xi$ is that $\xi$ is in the kernel of the map given by contraction with the exponential Atiyah class: 
\begin{Prop}[{\cite[Proposition 6.1]{TodaFM}}]\label{TodaDeformations}
For $E\in D^b(X)$ and $\xi\in \HT^2(X)$, assume that $\langle \xi, \exp\,\At(E)\rangle = 0$ in $\Ext^2_X(E, E)$. Then there exists an object $E^{\xi}\in D^b(X, \xi)$ such that $Li^*E^{\xi} \cong E$. 
\end{Prop}
In fact, the kernel of the contraction map $\langle -, \, \exp\At(E)\rangle$ is equal to the kernel of the obstruction map (\ref{obstruction map}) by the following theorem:
\begin{Th}[{\cite[Theorem A, Theorem B]{Huang}}]\label{HuangTheorems}
    There is a commutative diagram 
    \[\begin{tikzcd}
	{\HH^*(X)} && {\mathrm{Ext}^*(\mathcal{F}, \mathcal{F})} \\
	{\HT^*(X)}
	\arrow["ob_{\mathcal{F}}", from=1-1, to=1-3]
	\arrow["\HKR", from=2-1, to=1-1]
	\arrow["{\langle-, \,\mathrm{exp}(\At_{\mathcal{F}})\rangle}"', from=2-1, to=1-3]
\end{tikzcd}\]
Furthermore, for $\tilde{\alpha}\in \HT^*(X)$ a polyvector, if $\langle \tilde{\alpha}, \, \mathrm{exp}(\At_{\mathcal{F}})\rangle = 0$, then $\langle D(\tilde{\alpha}), \, v(\mathcal{F})\rangle = 0$, where $D$ is the Duflo operator acting as an automorphism on $\HT^*(X)$ via contraction by the Todd class $\td_X^{1/2}$.
\end{Th}

% \begin{Rem}
% We note that in the case of abelian varieties, $\td_X$ is trivial and therefore the kernel of the obstruction map (\ref{obstruction map}) is contained in the kernel of contraction by the Chern character of $\mathcal{F}$.
% \end{Rem}

Let us denote the obstruction map given by contraction with the exponential Atiyah class as  
\begin{equation}\label{HTObstruction map}
    \mathrm{ob}^{\HT}_{\cF}: \HT^*(X)\longrightarrow \mathrm{Ext}^*(\mathcal{F}, \mathcal{F}).
\end{equation}

\begin{Def}\label{kappa class definition}
    For $\tilde{E}$ a $\theta$-twisted coherent sheaf of rank $r>0$, define the characteristic class
\begin{equation}\label{twisted kappa}
\kappa(\tilde{E}): = \sqrt[r]{\ch(\tilde{E}^{\otimes r}\otimes \mathrm{det}(\tilde{E})^{-1})},
\end{equation}
where $\sqrt[r]{(-)}$ denotes the unique $r$th root in $H^{ev}(X, \bQ)$ whose degree-zero component is equal to $r$. The sheaf $\tilde{E}^{\otimes r}\otimes \det(\tilde{E})^{-1}$ is untwisted, since the cocycle twisting $\tilde{E}^{\otimes r}$ is inverse to the one twisting $\det(\tilde{E})^{-1}$. Note that if $E$ is an untwisted sheaf, (\ref{twisted kappa}) reduces to
\begin{equation}
    \kappa(E) = \ch(E)\exp\left(\frac{-c_1(E)}{r}\right).
\end{equation}
\end{Def}

\begin{Def}\label{chern character of twisted sheaves}
    Let $\tilde{E}$ be a $\theta$-twisted coherent sheaf of rank $r$. Assume $\det(\tilde{E})\cong \cO$. Then the Chern character of $\tilde{E}$ is defined as 
    \begin{equation}
        \ch(\tilde{E}): = \sqrt[r]{\ch(\tilde{E}^{\otimes r})}.
    \end{equation}
\end{Def}
\begin{Rem}
    For a deformation, the Chern character $\ch(E)$ remains of Hodge-type if and only if both $\kappa(E)$ and $c_1(E)$ do, therefore the Hodge locus of $\kappa(E)$ contains the Hodge locus of $\ch(E)$.  
\end{Rem}
\begin{Lem}\label{replacing kappa with chern}
    Let $\mathcal{G}\rightarrow X$ be a rank $r$ possibly $\theta$-twisted sheaf. There exists a twisted sheaf $\mathcal{G}^{\prime}$ with trivial determinant such that $\kappa(\mathcal{G})= \kappa(\mathcal{G}^{\prime}) = \ch(\mathcal{G}^{\prime})$. As a result, for a deformation of $X$ over a base $B$, the $\kappa(\mathcal{G})$ remains of Hodge-type on a locus $U\subset B$ if and only if $\ch(\mathcal{G}^{\prime})$ remains of Hodge-type on that locus. 
\end{Lem}
\begin{proof}
    Given such a sheaf $\mathcal{G} $, the existence of $\mathcal{G}^{\prime}$ follows from \cite[Construction 7.3.3]{Markman2n}. Let us provide the details: 
    Let $\{U_i\}$ be an analytic open cover of $X$, fine enough so that there exists a trivialization $\psi_i: \wedge^r \mathcal{G}_i: = \wedge^r \mathcal{G}|_{U_i}\xrightarrow{\sim} \mathcal{O}_{U_i}$
    of the determinant line bundle. Let $\eta_{ij}$ denote the \v{C}ech cocycle representing $\wedge^r \mathcal{G}$. Let $\tilde{\eta}_{ij}$ denote a choice of the $r$-th root of $\eta_{ij}$, and let $\phi_{ij}: \mathcal{G}_{j}|_{U_{ij}}\xrightarrow{\sim} \mathcal{G}_i|_{U_{ij}}$ be multiplication by $\tilde{\eta}_{ij}^{-1}$. Then the sheaf $\mathcal{G}^{\prime}: = (\{\mathcal{G}_i\}, \{\phi_{ij}\})$ is a $\tilde{\theta}$-twisted sheaf where $\tilde{\theta}_{ijk}: = (\tilde{\eta}_{ij}\tilde{\eta}_{jk}\tilde{\eta}_{ki})^{-1}$. The determinant line bundle $\wedge^r \mathcal{G}^{\prime}$ is trivial, as it is represented by the \v{C}ech cocycle $\tilde{\eta}_{ij}^{-1}\eta_{ij} = 1$. Note that if $\mathcal{G} = (\{\mathcal{G}_i\}, \{\tilde{\phi}_{ij}\})$ were already a $\theta$-twisted sheaf, then $\mathcal{G}^{\prime}$ is a $\tilde{\theta}\theta$-twisted sheaf. The identity $\kappa(\cG^{\prime}) = \kappa(\cG)$ follows from the identity $\cG^{\prime \otimes r}\cong \cG^{\otimes r}\otimes \det(\cG)^{-1}$. 

Finally, since $\det\cG^{\prime}\cong \cO_X$, Definitions \ref{kappa class definition} and \ref{chern character of twisted sheaves} give $\kappa(\cG^{\prime}) =\sqrt[r]{\ch(\cG^{\prime \otimes r})}= \ch(\cG)$.

    % Additionally, for a sheaf $\mathcal{G}^{\prime}$ with trivial determinant and twisted by a class $\tilde{\theta}$ with coefficients in $\mu_r$, we have $\kappa(\mathcal{G}^{\prime}) = \ch(\mathcal{G}^{\prime}),
    % $
    % as $\kappa(\mathcal{G}^{\prime}) = \ch(\mathcal{G}^{\prime})\mathrm{exp}\left(\frac{-c_1(\mathcal{G}^{\prime})}{r}\right)$ and the determinant line bundle $\wedge^r\mathcal{G}^{\prime}$ is trivial.
\end{proof}

\begin{Rem}
    A $\theta$-twisted coherent sheaf on $X$ is the same as a coherent sheaf on a $\mu_r$-gerbe $q: \tilde{X}\rightarrow X$ on which $\mu_r$ acts via the identity character. We will use the two concepts interchangeably in the sequel. 
\end{Rem}
\begin{Lem}\label{gerbe cohomology isomorphism}
    Let $Y$ be a proper smooth scheme over a local Artinian $\bC$-algebra $A$ and $q: \tilde{Y}\rightarrow Y$ a  $\mu_r$ gerbe. Then for every coherent sheaf $\cF$ on $Y$, 
    \[
    q^*: H^i(Y, \cF)\rightarrow H^i(\tilde{Y}, q^*\cF)
    \]
    is an isomorphism.
\end{Lem}
\begin{proof}
By \cite[Proposition 2.2.1.6]{Lie07}, every quasi-coherent sheaf $\cF$ on $\tilde{Y}$ is the direct sum of its $\mu_r$-eigensheaves $\cF_\chi$, $\chi\in \bZ/r$, and $q_*\cF=q_*(\cF_0)$. The weight-zero summand is a quasi-coherent sheaf on $Y$ pulled back to $\tilde{Y}$. Hence, $q_*$ is exact on quasi-coherent sheaves, $q_*\cO_{\tilde{Y} }=\cO_Y$, and $R^iq_*=0$ for $i>0$. By the projection formula $Rq_*q^*\cF=\cF$, and the degeneration of the Leray spectral sequence gives the claim.
\end{proof}

\subsection{A generalization of the Buchweitz-Flenner theorem}

The main result of \cite{Perry} is the following:
\begin{Th}[{\cite[Theorem 1.1]{Perry}}]\label{Perry semiregularity}

Let $f:\cX\rightarrow S$ be a smooth proper family of complex varieties, over a complex variety $S$. Let $0\in S(\bC)$, let $X_0$ be the central fiber, and let $E_0\in D_{\mathrm{perf}}(X_0)$ be a semiregular perfect complex with $\Ext^{<0}(E_0, E_0) = 0$. Assume that $B_0\in H^2(X_0, \bQ(1))$ is an algebraic class such that 
\[
w_0: = \ch(E_0)\exp(B_0)\in \bigoplus_{k\geq 0}H^{2k}(X_0, \bQ(k))
\]
remains of Hodge-type along $S$, i.e. lifts to a global section $w$ of the local system $\oplus_{k\geq 0}R^{2k}f_*\bQ(k)$ which is of Hodge-type at every point. Then: 
\begin{enumerate}
    \item $E_0$ deforms as a twisted perfect complex over an \'etale neighborhood of $0$.
    \item $w_s$ is algebraic for every $s\in S(\bC)$.
\end{enumerate}
    
\end{Th}

\begin{Rem}
    In the context of Theorem \ref{Perry semiregularity}, $E_0$ is semiregular if the semiregularity map $\sigma: \Ext^2(E_0, E_0)\rightarrow \HH_{-2}(X_0)$ is injective (\cite[Definition 2.4]{Perry}). Applying the $\HKR$ isomorphism gives the map \eqref{semiregularity map} for $E_0$ a sheaf and \cite[Definition 4.1]{Buchweitz-Flenner} for a perfect complex. For $E_0$ a rank $r$ vector bundle, we take $B_0 = -c_1(E_0)/r$, so that $w_0 = \kappa(E_0)$.
\end{Rem} 
\begin{Prop}[{\cite[Lemma 8.2, Corollary 8.5]{Rienks}}]
    Let $X$ and $Y$ be smooth projective varieties, and let $X$ be $K$-trivial, i.e. $\omega_X\cong \cO_X$. Then for $\cE\in D^b(X\times Y)$ the kernel of a fully faithful functor $D^b(X)\rightarrow D^b(Y)$, $\cE$ is semiregular. 
\end{Prop}

Consequently, by Theorem \ref{Perry semiregularity} with $B_0 = -c_1(\cE)/r$, on every smooth proper algebraic family deforming $X\times Y$ along which $\kappa(\cE)$ remains of Hodge-type, $\cE$ deforms as a twisted perfect complex over an \'etale neighborhood of the base point.

\subsection{The analytic case}\label{the analytic case}
In this section, let $f:\cX\rightarrow (S, 0)$ be a smooth proper family over a smooth germ with a smooth projective central fiber $X_0$. Let $E_0$ be a locally free sheaf on $X_0$. Assume that $E_0$ is semiregular as in Definition \ref{Definition of semiregular}.

Let us briefly outline our strategy for this section: Our families of products of Kummer-type IHSMs in Section \ref{section 6 LSC and derived equivalences} pass through non-projective fibers and therefore the algebraic semiregularity theorems of Perry and Pridham do not apply directly. On the other hand, the main theorem of \cite{Buchweitz-Flenner}, though stated in the analytic setting, does not cover twisted sheaves. We combine the two as follows: Replacing $E_0$ with its normalization $G_0$ (a twisted sheaf with trivial determinant whose Chern character is equal to $\kappa(E_0)$), we aim to extend $G_0$ over the family $f: \cX\rightarrow (S, 0)$ whenever $\kappa(E_0)$ stays of Hodge-type. The question is local and we answer it with deformation theory in two steps. First, we extend $G_0$ over each infinitesimal thickening $X_k$ of $X_0$ in $\cX$. Each $X_k$ is defined over the Artinian ring $A_k: =\cO_{S, 0}/\mathfrak{m}^{k+1}$. Then \cite[Corollary 2.25, Remarks 2.26 and 2.27]{Pridham} state that his Lefschetz map (which in this context coincides with the semiregularity map of \cite{Buchweitz-Flenner}) annihilates the obstruction to extending $G_k$ to $X_{k+1}$, provided that $\kappa(E_0)$ remains of Hodge-type to order $k$. Since $E_0$ is semiregular, this implies that the obstruction itself vanishes. The result is a formal deformation of $G_0$. Second, we show that this formal deformation converges. Since each $G_k$ is a twisted locally free sheaf, the projectivizations $\bP(G_k)$ are pullbacks via a convergent versal deformation along a formal section of the base (Lemma \ref{semiuniversal deformation in the twisted case}). The Artin approximation theorem replaces this formal section with a convergent one, and so the resulting projective bundle over $\cX$ lifts to the desired twisted sheaf.  

\begin{Not}\label{section 2 notation}
Let $\mathfrak{m}\subset\mathcal{O}_{S,0}$ be the maximal ideal, and for
$k\ge 0$ define
\[
  A_k:=\cO_{S,0}/\mathfrak{m}^{k+1},\qquad
  S_k:=\Spec A_k,\qquad
  X_k:=\cX\times_S S_k,\qquad
  I_k:=\mathfrak{m}^{k+1}/\mathfrak{m}^{k+2},
\]
so that $A_{k+1}\to A_k$ is a small extension with kernel $I_k$, i.e.\
$\mathfrak{m}_{A_{k+1}}I_k=0$. All $X_k$ have the underlying topological space
$|X_0|$, so a locally constant $\mu_r$-cocycle $\theta_r$ on $|X_0|$ defines
compatible $\mu_r$-gerbes $\tilde{X}_k\to X_k$, with
$\tilde{X}_{k+1}\times_{S_{k+1}}S_k\cong\tilde{X}_k$.
\end{Not}

Let $p:\tilde{X}_0\rightarrow X_0$ be a $\mu_r$-gerbe of $r$th roots of $\det(E_0)^{-1}$, and let $M$ be the tautological line bundle on $\tilde{X}_0$ so that $M^{\otimes r} = p^*\det(E_0)^{-1}$. We write $G_0: = p^*E_0\otimes M$. $G_0$ is a twisted locally free sheaf of rank $r$ with trivial determinant and $\kappa(E_0) = \ch(G_0)$; We have $G^{\otimes r}_0 = p^*(E_0^{\otimes r})\otimes M^{\otimes r} = p^*(E^{\otimes r}_0\otimes \det(E_0)^{-1})$ is the pullback of the untwisted sheaf $E^{\otimes r}_0\otimes \det(E_0)^{-1}$ on $X_0$, so Definitions \ref{kappa class definition} and \ref{chern character of twisted sheaves} give $\ch(G_0) = \sqrt[r]{\ch(E^{\otimes r}_0\otimes \det(E_0)^{-1}} =\kappa(E_0)$. Likewise, $\det(G_0) = p^*\det(E_0)\otimes M^{\otimes r}\cong \cO$.

\begin{Def}\label{gerby normalization def}
    We call $G_0: = p^*E_0\otimes M$ the \textit{gerby normalization} of $E_0$.
\end{Def}

\begin{Lem}\label{equality of ext groups for gerbes}
    Let $G_0$ be the gerby normalization of $E_0$ on $p: \tilde{X}_0\rightarrow X_0$. Then for all $i$, 
    \[
    \Ext^i_{\tilde{X}_0}(G_0, G_0) = H^i(X_0, \cE nd(E_0)) = \Ext^i_{X_0}(E_0, E_0).
    \]
\end{Lem}
\begin{proof}
    Since $M$ is a line bundle, $\cH om_{\tilde{X}_0}(G_0, G_0) = \cH om(p^* E_0, p^* E_0) = p^*\cE nd(E_0)$. Both $G_0$ and $E_0$ are locally free, therefore
    \[
    \Ext^i_{\tilde{X}_0}(G_0, G_0)\cong H^i(\tilde{X}_0, p^*\cE nd(E_0)) = H^i(X_0, \cE nd(E_0)),
    \]
    by Lemma \ref{gerbe cohomology isomorphism}.
\end{proof}

\begin{Th}[Analytic analogue of Theorem \ref{Perry semiregularity}]\label{analytic Buchweitz-Flenner conjecture}
    Let $f: \cX\rightarrow (S, 0)$ be the deformation of a projective variety $X_0$ over a smooth germ, and let $E_0$ be a rank $r>0$ semiregular locally free coherent sheaf on $X_0$ such that for each $p$, the flat section $w_p$ of $R^{2p}f_*\bC$ with $w_p(0)= \kappa_p(E_0)$ lies in $F^pR^{2p}f_*\Omega^{\bullet}_{\cX/S}$ at every point of $S$. Then, after possibly shrinking $S$ so that $\cX$ deformation retracts to $X_0$, $G_0$ extends to a $\theta_r$-twisted locally free coherent sheaf $\cE$ flat over $S$ such that $\det \cE\cong \cO_{\cX}$ and $\cE_{X_0}\cong G_0$. Moreover, $\kappa(\cE_s) = \ch(\cE_s)$ is the parallel transport of $\kappa(E_0)$.
\end{Th}
The proof of this theorem occupies the rest of this section. Throughout the remainder of the section, we rely on numerous results of Pridham in \cite{Pridham}. We reformulate Pridham's results in classical language. 
\begin{Def}[{\cite[Definition 1.16]{Pridham}}]\label{Pridham lefschetz map}
    Let $Y$ be a smooth proper scheme over an Artinian algebra $A$ and let $F$ be a perfect complex on $Y$, or on a $\mu_r$-gerbe $\tilde{Y}\rightarrow Y$. Pridham's \textit{Lefschetz map} is the Hochschild trace map 
    \[
    \cL_F: \RHom_{Y}(F, F)\rightarrow \HH_*(Y/A)
    \]
    given by the composition 
    \begin{equation}\label{Pridham map composition}
    \RHom_Y(F, F)\cong R\Gamma_{Y\times Y}((F\boxtimes F^{\vee})\otimes^L \Delta_* \cO_{Y/A})\xrightarrow{ev} R\Gamma_{Y\times Y}(\Delta_*\cO_{Y/A}\otimes \Delta_*\cO_{Y/A}) = \HH_*(Y/A)
    \end{equation}
    (see the discussion following \cite[Definition 1.16 and \S 2.2  ]{Pridham}). The map $ev$ is the composition of maps $\Delta^*$ and $\Tr$. On cohomology, this gives a morphism $\Ext^i(F, F)\rightarrow \HH_{-i}(Y/A)$.
    In degree $2$, and after applying the HKR isomorphism, it gives 
    \begin{equation}\label{Pridham semiregularity map}
    L_F = (L_{F, p-1})_{p\geq 1}: \Ext^2(F, F)\rightarrow \prod_{p\geq 1} H^{p+1}(Y, \Omega^{p-1}_{Y/A}),
    \end{equation}
    by \cite[Definition 2.19, Proposition 2.20]{Pridham}.  With coefficients in an $A$-module $I$, it gives a morphism $\Ext^2(F, F\otimes_A I)\rightarrow \prod_{p\geq 1} H^{p+1}(Y, \Omega^{p-1}_{Y/A}\otimes_A I)$. In particular, if $Y$ is a scheme over $A = \bC$, then $L_F$ is the Buchweitz-Flenner map $\sigma_F(\nu) = \Tr(\nu \circ \exp(-\At_F))$ of \eqref{semiregularity map}, by \cite[Remark 2.21]{Pridham}; it follows that $F$ is semiregular, in the sense of Definition \ref{Definition of semiregular}, if and only if $L_F$ is injective. 
\end{Def}

\begin{Rem}\label{Naturality of Pridham's map}
    For any morphism $h: Z\rightarrow Y$ of derived Artin stacks and for any perfect complex $F$, each term in the composition \eqref{Pridham map composition} is compatible with $Lh^*$ \cite[\S 2.2, Definition 2.18-2.19, Proposition 2.20]{Pridham}. Therefore, 
    \begin{equation}
        L_{h^*F}(h^*\nu) = h^*L_F(v),
    \end{equation}
    for $\nu\in \Ext^2(F, F)$. On a smooth projective variety, Pridham's Chern character (\cite[Definition 2.18]{Pridham}) agrees with the classical Chern character (\cite[Remark 2.21]{Pridham}). Therefore, $\ch(h^*(F)) = h^*\ch(F)$.
\end{Rem}

\begin{Def}\label{Hdr condition}
    For a smooth proper scheme over a local Artinian $\bC$-algebra $A$ with closed fiber $X_0$, we say that $X/A$ satisfies \textbf{HdR} if its Hodge-to-de Rham spectral sequence degenerates at the $E_1$ page and each $H^q(X, \Omega^p_{X/A})$ is a free $A$-module (cf. \cite[Remark 2.27]{Pridham}).
\end{Def}

\begin{Rem}\label{HdR remark}
    Assume that $X^{\prime}/A^{\prime}$ satisfies \textbf{HdR} and let $0\rightarrow I\rightarrow A^{\prime}\rightarrow A\rightarrow 0$ be a small extension, i.e. $\mathfrak{m}_{A^{\prime}}I = 0$. Then for each $j$, the groups $H^q(X, \Omega^i_{X/A})$ commute with base change, i.e.\[
    H^q(X, \Omega^j_{X/A}\otimes_A I)\cong H^q(X_0, \Omega^j_{X_0})\otimes_{\bC}I.
    \]
    Furthermore by the condition \textbf{HdR}, the truncated de Rham cohomology $\mathbb{H}^n(X, \Omega^{<p}_{X^{\prime}/A^{\prime}})$ is a free $A^{\prime}$-module and commutes with base change. So in the long exact sequence associated to 
    \[
    0\rightarrow \Omega^{<p}_{X^'/A^'}\otimes I\rightarrow \Omega^{<p}_{X^'/A^'}\rightarrow \Omega^{<p}_{X/A}\rightarrow 0, 
    \]
    the restriction maps $\mathbb{H}^q(X, \Omega^{<p}_{X^{\prime}/A^{\prime}})\rightarrow \mathbb{H}^q(X, \Omega^{<p}_{X/A})$ are surjective (cf. \cite[Remark 2.27]{Pridham}) and $\mathbb{H}^{2p}(X, \Omega^{<p}_{X^{\prime}/A^{\prime}}\otimes I)\hookrightarrow \mathbb{H}^{2p}(X^{\prime}, \Omega^{<p}_{X^{\prime}/A^{\prime}})$ is injective. Since $\Omega^{<p}_{X^{\prime}/A^{\prime}}\otimes I = \Omega_{X_0}^{<p}\otimes_{\bC} I$, and the Hodge-to-de Rham spectral sequence degenerates for $X_0$, $H^{p+1}(X_0, \Omega_{X_0}^{p-1})\otimes I$ is a direct summand of $\mathbb{H}^{2p}(X^{\prime}, \Omega^{<p}_{X^{\prime}/A^{\prime}}\otimes I)$. Therefore, 
    \begin{equation}\label{injection into hypercohomology}
    H^{p+1}(X, \Omega^{p-1}_{X/A}\otimes_A I)\cong H^{p+1}(X_0, \Omega^{p-1}_{X_0})\otimes_{\bC}I\hookrightarrow \mathbb{H}^{2p}(X^{\prime}, \Omega^{<p}_{X^{\prime}/A^{\prime}})
    \end{equation}
    is injective. In particular, for $G$ a locally free twisted sheaf on $\tilde{X}$, flat over $A$ and with closed fiber $G_0$, then $\Ext^2_{\tilde{X}}(G, G\otimes_{A} I)\cong \Ext^2_{\tilde{X}_0}(G_0, G_0)\otimes I$ and under this identification, Pridham's Lefschetz map with coefficients in $I$ is given by 
    \[
    L_{G_0}\otimes \Id_I: \Ext^2(G, G\otimes_AI)\cong \Ext^2(G_0, G_0)\otimes_{\bC} I\rightarrow \bigoplus_{p\geq 0}H^{p+1}(X_0, \Omega^{p-1}_{X_0}) 
    \]
    (see the proof of \cite[Theorem 4.1]{Perry} and \cite[Proposition 2.20]{Pridham}).

\end{Rem}

We next check that Pridham's semiregularity map is stable under gerby normalization:
\begin{Lem}\label{semiregularity is stable under gerby normalization}
    Assume that $E_0$ is locally free and semiregular in the sense of Definition \ref{Definition of semiregular}. Then Pridham's map 
    \[
    L_{G_0}: \Ext^2(G_0, G_0) = \Ext^2(E_0, E_0)\rightarrow \bigoplus_{p\geq 1}H^{p+1}(X_0, \Omega_{X_0}^{p-1}).
    \]
    \eqref{Pridham semiregularity map} is injective, where $G_0$ is the gerby normalization of $E_0$.

\end{Lem}
\begin{proof}

Pridham's Lefschetz map $\cL_F$ is the derivative of the Abel-Jacobi map of \cite[Definition 2.18]{Pridham}. It is defined as a composition of natural maps and so it commutes with pullbacks by Remark \ref{Naturality of Pridham's map}.

Let $p:\tilde{X}_0\rightarrow X_0$ be a $\mu_r$-gerbe of $r$th roots of $\det(E_0)^{-1}$.
   There exists a smooth projective variety $Z$, a finite flat surjective morphism of varieties $g: Z\rightarrow X_0$ and a line bundle $N$ on $Z$ such that $N^{\otimes r} = g^*\det(E_0)^{-1}$ (\cite[Theorem 4.1.10]{LazarsfeldI}). There is then a lift $\tilde{g}: Z\rightarrow \tilde{X}_0$ such that $\tilde{g}^*M = N$, and hence $\tilde{g}^*G_0 = g^*E_0\otimes N$ is a locally free sheaf on $Z$.

Let $\nu\in \Ext^2(G_0, G_0)$ be such that $L_{G_0}(\nu) = 0$. By Remark \ref{Naturality of Pridham's map} we have
\[
L_{g^*E_0\otimes N}(g^*\nu\otimes \Id_N) = \tilde{g}^*L_{G_0}(\nu) = 0
\]
and for $Z$ a smooth projective variety, $L$ is the semiregularity map $\sigma$ as in \cite{Buchweitz-Flenner}. By the additivity of the Atiyah class, 
\[
\At(g^*E_0\otimes N) = \At(g^*E_0)\otimes 1+ 1\otimes \At(N). 
\]
Thus, by the equality $c_1(N) = \Tr(-\At(N))$, 
\[
\sigma_{g^*E_0\otimes N}(g^*\nu\otimes \Id_N) = \Tr((g^*\nu\otimes \Id_N)\exp(-\At(g^*E_0)))\cup \exp(c_1(N)) = \sigma_{g^*E_0}(g^*\nu)\cup \exp(c_1(N)).
\]
Furthermore, $\sigma_{g^*E_0}(g^*\nu) = g^*\sigma_{E_0}(\nu)$ by the naturality of the semiregularity map (\cite[Remark 4.7(1)]{Buchweitz-Flenner}). It follows that $g^*\sigma_{E_0}(\nu) = 0$, since the cup product with $\exp(c_1(N))$ is an automorphism of $\oplus_p H^{q+2}(Z, \Omega^q_Z)$.

    Since $g$ is a finite surjective morphism of smooth, equidimensional varieties, we have $g_*g^* = \deg(g)\cdot \Id$, so $g^*$ is injective on cohomology. Since it is a morphism of Hodge structures, it is injective on $H^{q+2}(X_0,\Omega^q)$. Therefore, $\sigma_{E_0}(\nu) = 0$ and $\nu = 0$, by the semiregularity of $E_0$.
    \end{proof}

Pridham's semiregularity theorem (\cite[Corollary 2.25, Remarks 2.26 and 2.27]{Pridham}) and the Hochschild-theoretic Lefschetz map through which it is formulated take as inputs smooth proper schemes over Artinian rings. On the other hand, the thickenings $X_k$ are complex analytic spaces and the family $\cX\rightarrow S$ cannot be algebraized, since its fibers are in general not projective. We proceed by algebraizing not the family, but the infinitesimal neighborhood of $X_0$.

\begin{Lem}\label{analytification lemma}
Let $f$, $X_0$, $A_k$, and $X_k$ be as in the setup of \S \ref{the analytic case} and  Notation \ref{section 2 notation}. Let $p\colon\tilde{X}_0\to X_0$ be the $\mu_r$-gerbe of $r$-th roots of
$\det(E_0)^{-1}$; it is algebraic, since $\det(E_0)$ is algebraic, by GAGA.
\begin{enumerate}
\item There are smooth proper $A_k$-schemes $X^{\mathrm{alg}}_k$ and
isomorphisms $\iota_k\colon(X^{\mathrm{alg}}_k)^{\mathrm{an}}\xrightarrow{\sim}X_k$
of complex spaces over $A_k$, such that
$X^{\mathrm{alg}}_{k+1}\otimes_{A_{k+1}}A_k\cong X^{\mathrm{alg}}_k$ and
$\iota_{k+1}$ restricts to $\iota_k$.
\item There are algebraic $\mu_r$-gerbes
$\tilde{X}^{\mathrm{alg}}_k\to X^{\mathrm{alg}}_k$, compatible under base
change, with $\tilde{X}^{\mathrm{alg}}_0=\tilde{X}_0$.
\item Each $X^{\mathrm{alg}}_k/A_k$ satisfies \textbf{HdR}, and there exist $A_k$-linear isomorphisms
\[\Theta_k: H^{2p}(X_0,\mathbb{C})\otimes A_k\xrightarrow{\sim} H^{2p}_{\mathrm{dR}}(X^{\mathrm{alg}}_k/A_k)\]
with the following property: If $w$ is a section of $R^{2p}f_*\bC$ over a neighborhood of $0$ whose image in $R^{2p}f_*\Omega^{\bullet}_{\cX/S}$ is a section of $F^pR^{2p}f_*\Omega^{\bullet}_{\cX/S}$, then $\Theta_k(w(0)\otimes 1)\in F^pH^{2p}_{dR}(X^{alg}_k/A_k)$.
\end{enumerate}
\end{Lem}

\begin{proof}
(1): The compact complex space $X_k$ has the single irreducible component $X_0$, which is projective, and therefore it is a Moishezon space in the sense of \cite[Definition 7.1]{Artin}. By \cite[Theorem 7.3]{Artin} a Moishezon space is algebraizable, i.e. there exists a unique proper algebraic space $Y_k$ over $\bC$, together with an isomorphism $\iota_k: Y^{an}_k\xrightarrow{\sim} X_k$. Since $Y_k$ is the nilpotent thickening of the projective scheme $X_0$, it is a scheme \cite[0BPW]{Stacks}. Set $X_k^{alg}: = Y_k$.

By GAGA, $\Gamma(X^{alg}_k, \cO)= \Gamma(X_k, \cO)$ \cite[Exp. XII, Corollary 4.3]{SGA1}. The $A_k$-algebra structure of $\cO_{X_k}$ is given by a homomorphism $A_k \rightarrow \Gamma(X_k, \cO)$ and hence defines a morphism $X_k^{alg}\rightarrow \Spec A_k$, the analytification of which is the structure morphism of $X_k$. Therefore, $\iota_k$ is an isomorphism over $A_k$. Since $X_k\rightarrow A_k$ is smooth, so is $X_k^{alg}\rightarrow \Spec A_k$ (\cite[Exp. XII, Proposition 3.1]{SGA1}).

By \cite[Exp. XII, Corollary 4.5]{SGA1}, the closed immersion $X_k\hookrightarrow X_{k+1}$ is the analytification of a unique morphism $\phi_{k+1}: X^{alg}_k\rightarrow X^{alg}_{k+1}$ which is furthermore a closed immersion by \cite[Exp. Proposition 3.2]{SGA1}. The induced morphism
\[
X_k^{alg}\rightarrow X^{alg}_{k+1}\otimes_{A_{k+1}}A_k
\]
is then an isomorphism. By construction, $\iota_{k+1}$ rsetricts to $\iota_k$.

(2): Let $i_k: X_0\hookrightarrow X^{alg}_k$ be the closed immersion of part~(1); it is
defined by the nilpotent ideal $\mathfrak{m}_{A_k}\cO_{X^{alg}_k}$. By
\cite[Theorem 18.1.2]{EGAIV4}, the functor $U\mapsto U\times_{X^{alg}_k}X_0$ is an
equivalence of \'etale sites $(X^{alg}_k)_{\mathrm{et}}\simeq(X_0)_{\mathrm{et}}$, and it
carries the \'etale group scheme $\mu_r$ to $\mu_r$. Furthermore, these equivalences are compatible with the closed immersions $\phi_{k+1}$ of part (1). 

The gerbe $\tilde{X}_0$ corresponds to a $\mu_r$-gerbe $\tilde{X}_k^{alg}$ on $X_k^{alg}$ with $i^*_k\tilde{X}^{alg_k} =\tilde{X}_0$. Base change of the gerbe along $\phi_{k+1}$ is the restriction of the corresponding stack along $(X_k^{alg})_{et}\cong (X_{k+1}^{alg})_{et}$, which is the identity under the above equivalence of \'etale sites. As a result, 
\[
\tilde{X}_{k+1}^{alg}\times_{X_{k+1}^{alg}} X_k^{alg} = \tilde{X}_k^{alg}.
\]

Finally, a $\mu_r$-gerbe over a scheme is a Deligne-Mumford stack, \'etale locally of the form $B\mu_r\times U$, where $B\mu_r$ is the classifying stack of $\mu_r$ (\cite[Lemma 2.2.1.1, Remark 2.2.1.2]{Lie07}). Therefore $\tilde{X}^{alg}_k$ is smooth and proper over $\Spec A_k$.

(3): Condition \textbf{HdR}
for the smooth and proper morphism $X^{\mathrm{alg}}_k\to\Spec A_k$ is given by  \cite[Theorem 5.5(i),(ii)]{Deligne}. By (1), there are isomorphisms $\iota_k: (X^{alg}_k)^{an}\xrightarrow{\sim} X_k$ of complex spaces over $A_k$, and by \cite[Exp. XII, Theorem 1.1]{SGA1}, there exist canonical morphisms of ringed spaces 
\[
\delta: (X_k)^{an}\rightarrow X_k^{alg}.
\]
Let $\epsilon: = \delta\circ \iota_k^{-1}: X_k\rightarrow X_k^{alg}$. We write 
\[
\Omega^{\bullet, an}: = (\Omega^{\bullet}_{X^{alg}/A_k})^{an} = \epsilon^*\Omega^{\bullet}_{X_k^{alg}/A_k},
\]
which is a complex of sheaves of $A_k$-modules on $|X_0|$. Let $\underline{A_k}$ denote the constant sheaf on $|X_0|$. By \cite[Lemma 5.5.3]{Deligne}, there is a quasi-isomorphism
\[
c_k: \underline{A_k}\xrightarrow{\sim} \Omega^{\bullet, an},
\]
thus, there is an induced $A_k$-linear isomorphism
\[
c_k^H: H^{2p}(|X_0|, A_k)\xrightarrow{\sim} \mathbb{H}^{2p}(X_k, \Omega^{\bullet, an}).
\]
Since $A_k$ is a finite-dimensional $\bC$-vector space, the fact that cohomology commutes with finite direct sums implies $H^{2p}(|X_0|, A_k) = H^{2p}(X_0, \bC)\otimes_{\bC} A_k$.

 Since $\epsilon^{\#}: \cO_{X_k^{alg}}\rightarrow \epsilon_*\cO_{X_k}$ is a morphism of sheaves of $A_k$-algebras, the map $\gamma$ is $A_k$-linear. 
The unit of the adjunction $\Id\mapsto \epsilon_*\epsilon^*$ evaluated at $\Omega^{\bullet}_{X_k^{alg}/A_k}$ induces the map
\[
\gamma: \mathbb{H}^{2p}(X_k^{alg}, \Omega^{\bullet}_{X_k^{alg}/A_k})\rightarrow \mathbb{H}^{2p}(X_k, \Omega^{\bullet, an}).
\]

For $p\geq 0$, let $\Omega^{\geq p}$ denote the truncation of the de Rham complex $\Omega^{\bullet}$. Let 
\[
j_p:  \mathbb{H}^{2p}(X_k^{alg}, \Omega^{\geq p}_{X_k^{alg}/A_k})\rightarrow \mathbb{H}^{2p}(X_k^{alg}, \Omega^{\bullet}_{X^{alg}_k/A_k})
\]
be the morphism induced by inclusion. The Hodge filtration is then $F^pH^{2p}_{dR}(X_k^{alg}/A_k)=: \mathrm{Im}(j_p)$. By the \textbf{HdR} property, $j_p$ is injective. Let $\gamma_{\geq p}$ be the associated maps on the truncated complexes. By naturality of the unit $\Id\mapsto \epsilon_*\epsilon^*$ with respect to the inclusion $\Omega^{\leq p}\subset \Omega^{\bullet}$, 
\begin{equation}\label{j commutes with gamma}
    \gamma \circ j_p = j^{an}_p\circ \gamma_{\geq p},
\end{equation}
where $j^{an}$ is the natural induced map.
Applying \cite[Exp. XII, Corollary 4.3]{SGA1} to the coherent sheaves $\Omega^i_{X_k^{alg}/A_k}$ on the proper $\bC$-scheme $X_k^{alg}$ yields that the canonical maps 
\[
H^q(X_k^{alg}, \Omega^i_{X_k^{alg}/A_k})\rightarrow H^q(X_k, \Omega^{i, an})
\]
are isomorphisms for each $q$ and $i$. The spectral sequence associated to the truncation filtration then shows that the canonical maps $\gamma$ and $\gamma_{\geq p}$ are isomorphisms (see the discussion immediately preceding \cite[Lemma(5.5.4)]{Deligne}). In particular, by \eqref{j commutes with gamma},
\[
\gamma(F^pH^{2p}_{dR}(X^{alg}_k/A_k)) = \mathrm{Im}(j^{an}_p).
\]
Therefore, we set
\begin{equation}
    \Theta_k: = \gamma^{-1}\circ c_k^H: H^{2p}(X_0, \bC)\otimes_\bC A_k\xrightarrow{\sim} H^{2p}_{dR}(X^{alg}_k/A_k),
\end{equation}
which is $A_k$-linear for each $k$. Furthermore, $\Theta_k(x)\in F^pH^{2p_{dR}}(X^{alg}_k/A_k)$ if and only if $c_k(x)\in \mathrm{Im}(j_p^{an})$.

Finally, Let $c: \underline{\bC}\rightarrow \Omega^{\bullet}_{\cX/S}$ be the inclusion, let $V\ni 0$ be an open set, and let $w$ be as in the statement. After possibly shrinking $V$, the section $w$ is represented by a class $\tilde{w}\in H^{2p}(f^{-1}(V), \bC)$, and $c(\tilde{w})$ lies in the image of $j$. Furthermore, $w(0) = \tilde{w}|_{X_0}\in H^{2p}(X_0, \bC)$, so restricting to $X_k$ we have 
\[
c_k(w(0)\otimes 1)= c(\tilde{w})|_{X_k}\in \mathrm{Im}(j^{an}_p),
\]
where $w(0)\otimes 1$ is the image of $w(0)$ under the map induced by $\bC\rightarrow A_k$. Thus, $\Theta_k(w(0)\otimes 1)\in F^pH^{2p}_{dR}(X_k^{alg}/A_k)$.
\end{proof}

% \[
% H^{2p}_{\mathrm{dR}}(X^{\mathrm{alg}}_k/A_k)
% \cong H^{2p}(X_0,\mathbb{C})\otimes A_k
% = R^{2p}f_*\Omega^\bullet_{\mathcal{X}/S}\otimes_{\mathcal{O}_S}A_k.
% \]
% GAGA on the graded pieces
% (cf.\ \cite[Example 2.14]{Pridham}), identifies
% $F^pH^{2p}_{\mathrm{dR}}(X^{\mathrm{alg}}_k/A_k)$ with
% $F^pH^{2p}_{\mathrm{dR}}(X_k/A_k)$, the image of
% $\mathbb{H}^{2p}(X_k,\Omega^{\ge p}_{X_k/A_k})$. After shrinking $S$, the
% fibres of $f$ are compact K\"ahler by Kodaira--Spencer stability, so
% $F^pR^{2p}f_*\Omega^\bullet_{\mathcal{X}/S}$ is a subbundle of
% $R^{2p}f_*\Omega^\bullet_{\mathcal{X}/S}$, namely the image of
% $R^{2p}f_*\Omega^{\ge p}_{\mathcal{X}/S}$ \cite[Ch.~10]{VoisinHT1}. Its
% restriction to $S_k$ maps under the base-change map
% $R^{2p}f_*\Omega^{\ge p}_{\mathcal{X}/S}\otimes A_k\to
% \mathbb{H}^{2p}(X_k,\Omega^{\ge p}_{X_k/A_k})$ into
% $F^pH^{2p}_{\mathrm{dR}}(X_k/A_k)$, which proves the inclusion.

Let us now formulate what Pridham proves in \cite{Pridham} in the way in which we use it:

\begin{Th}[{\cite[Corollary 2.25, Remarks 2.26, 2.27]{Pridham}}]\label{Pridham map kills obstructions}
    Let $0\rightarrow I\rightarrow A^{\prime}\rightarrow A\rightarrow 0$ be an extension of local Artinian $\bC$-algebras such that $\mathfrak{m}_{A^{\prime}}I = 0$. Let $X^{\prime}$ be a smooth proper $A^{\prime}$-scheme with closed projective fiber $X_0$, $\tilde{X}^{\prime}\rightarrow X^{\prime}$ a $\mu_r$-gerbe, and $G$ a twisted locally free coherent sheaf on $\tilde{X}:= \tilde{X}^{\prime}\times_{A^{\prime}}A$ with closed fiber $G_0$ and flat over $A$. Let \[o(G)\in \Ext^2(G_0, G_0)\otimes I\]
    represent the obstruction to extending $G$ over $\tilde{X}^{\prime}$. Then $o(G)$ has the following two properties:

    \begin{enumerate}
        \item If $o(G) = 0$, then $G$ extends to a locally free twisted sheaf on $\tilde{X}^{\prime}$, flat over $A^{\prime}$.
        
    \item If for each $p$, the extension of $\ch_p(G_0)\otimes 1$ remains of Hodge-type over $X^{\prime}/A^{\prime}$, then Pridham's map $(L_{G_0}\otimes \Id_I)(o(G)) = 0$ in $\oplus_{p\geq 1}H^{p+1}(X_0, \Omega^{p-1}_{X_0})\otimes I$. 
    \end{enumerate}
\end{Th}
\begin{Rem}
    Note that statement (2) is the slight strengthening of \cite[Corollary 2.25]{Pridham}, which asserts that $(L_{G_0}\otimes \Id_I)(o(G)) = 0$ in the hypercohomology group $\mathbb{H}^{2p}(X, \Omega_{X/A}^{<p})$. This strengthening is provided by \cite[Remark 2.27]{Pridham}.
\end{Rem}

\begin{proof}

(1):  
Since $G$ is flat over $A$ and $\mathfrak m_AI=0$, the complex $G\otimes^L_AI$ is the sheaf $G_0\otimes_{\bC}I$ supported on $\widetilde X_0$, so $\Ext^2_{\tilde X}(G,G\otimes^L_AI)=\Ext^2(G_0,G_0)\otimes I$ by adjunction, and $o(G)$ is the obstruction of \cite[Lemma 1.8]{Pridham} to extending $G$ to a perfect complex on $\tilde X'$. As $o(G)=0$, there is a perfect complex $G'$ on $\tilde X'$ whose restriction to $\tilde X$ is $G$.

It remains to see that $G'$ is a locally free sheaf. Its restriction to $\widetilde X_0$ is that of $G$, namely $G_0$, a locally free sheaf; in particular $G'\otimes^L N$ is a sheaf for every $\cO_{\tilde X_0}$-module $N$. An arbitrary $\cO_{\tilde X'}$-module $N$ has the finite filtration by the submodules $\mathfrak m_{A'}^jN$, whose graded pieces are $\cO_{\tilde X_0}$-modules, so $G'\otimes^LN$ is a sheaf for every $N$. For $N=\cO_{\tilde X'}$ this says that $G'$ is a sheaf, and for general $N$ that it is flat; being perfect, it is locally free. It is flat over $A'$ because $\cO_{\tilde X'}$ is, and $G'|_{\tilde X}\cong G$.

(2):
By Lemma \ref{gerbe cohomology isomorphism}, $\tilde{X}'$ has the same Hodge, de Rham (\cite[Remark 2.26]{Pridham}) and
rational cohomology as $X'$, so \cite[Corollary 2.25]{Pridham} applies to
$\tilde{X}'$ and $G$ \cite[Remark 2.26]{Pridham}. Its hypothesis is that the
constant extension of $\operatorname{ch}_p(G_0)$ to $H^{2p}_{\mathrm{dR}}(X'/A')$
lies in $F^p$, which is equivalent to the assumption in the statement. The conclusion of Pridham's remark is that, for every $p\ge 1$, the
image of $\cL_{p-1}(o(G))$ in
$\mathbb{H}^{2p}(X',\Omega^{<p}_{X'/A'})$ vanishes. By
Remark \ref{HdR remark}, which uses that $X^{\prime}/A^{\prime}$ satisfies the condition \textbf{HdR}, $\cL_{p-1}(o(G))$ is the
$(p-1)$-st component of
$(L_{G_0}\otimes\mathrm{Id}_I)(o(G))
\in\bigoplus_{p\ge 1}H^{p+1}(X_0,\Omega^{p-1}_{X_0})\otimes I$,
and this group injects into $\mathbb{H}^{2p}(X',\Omega^{<p}_{X'/A'})$. Hence
$(L_{G_0}\otimes\Id_I)(o(G))=0$.
\end{proof}

% implies $\sigma_{g^*E_0\otimes N}(g^*\nu\otimes \Id_N) = 0$; by \cite[Remark 2.26]{Pridham}, $H^*(X, \Omega_X^q) \cong H^*(Y, \Omega_Y^q)$ for all $\mu_r$-gerbes $X$ over $Y$, so $L_{G_0}$ takes values in $H^*(X_0, \bC)$.  By the additivity of the Atiyah class, \[\sigma_{g^*E_0\otimes N}(g^*\nu\otimes \Id_N) = \sigma_{g^*E_0}(g^*\nu)\cup \exp{c_1(N)},\] where $c_1(N) = -g^*c_1(E_0)/r$ and furthermore, $\sigma_{g^*E_0}(g^*\nu) = g^*\sigma_{E_0}(\nu)$, since the Atiyah class and trace commute with pullback. Since $Rg_*\cO_{Z} = \cO_{X_0}$, $g^*$ is injective on cohomology. Therefore, $\sigma_{E_0}(\nu) = 0$ implies $\nu = 0$.

\begin{Lem}\label{semiuniversal deformation in the twisted case}
    Let $G_0$ be the gerby normalization of $E_0$. Suppose that for $k\geq 0$, there are $\bP^{r-1}$-bundles $\pi_k: \bP_k\rightarrow X_k$, flat and proper over $S_k$ such that $\bP_{k+1}\times_{S_{k+1}}S_k\cong \bP_k$ over $X_k$ and $\bP_0 = \bP(G_0) = \bP(E_0)$. Then, after possibly shrinking $S$, there is a $\theta_r$-twisted locally free sheaf $\cE$ on $\cX$, flat over $S$ such that $\det\cE\cong \cO_{\cX}$ and $\cE|_{X_0}\cong G_0$.
\end{Lem}

\begin{proof}
    The proof follows arguements of \cite[Sections 5, 6, and 7]{Buchweitz-Flenner}. The collection $(\pi_k: \bP_k\rightarrow X_k)_k$ forms a deformation theory over $(S, 0)$ relative to the family $f: \cX\rightarrow S$, as detailed in \cite[\S 7.18]{Buchweitz-Flenner}. As a result, \cite[\S 7.18]{Buchweitz-Flenner} states that this deformation theory admits a convergent versal deformation $\cP\rightarrow \cX\times_S T$ over a germ $(T, 0)\rightarrow (S, 0)$. Then the isomorphisms $G_{k+1}|_{X_k}\cong G_k$ make $(\pi_k:\bP_k\rightarrow X_k)_k$ a compatible system \[(\pi_k: \bP_k\rightarrow X_k)\hookrightarrow (\pi_{k+1}: \bP_{k+1}\rightarrow X_{k+1})\] over each embedding $S_k\hookrightarrow S_{k+1}$. The lifting property of a versal deformation given in \cite[\S 6.5]{Buchweitz-Flenner} yields morphisms $\phi_k: S_k\rightarrow T$ such that $\phi_{k+1}|_{S_k} = \phi_k$ and $(1\times \phi_k)^*\cP\cong \bP_k$. Therefore, there exists a formal section $\widehat{\phi}: \varinjlim \phi_k: \widehat{S}\rightarrow T$ of $T\rightarrow S$. By Artin's approximation theorem, there exists a convergent section $\phi: (S, 0)\rightarrow (T, 0)$ (cf. \cite[Proof of Proposition 5.1]{Buchweitz-Flenner}).

    Define
    \begin{equation}
        \cP_S: = (1\times \phi)^*\cP\rightarrow \cX,
    \end{equation}
    then $\cP_S$ is flat and proper over $S$ and restricts to $\pi_0$ over $0\in S$. Since $\bP_0\rightarrow X_0$ is a $\bP^{r-1}$-bundle and $H^1(\bP^{r-1}, \cT_{\bP^{r-1}}) = 0$, the fibers of $\cP_S\rightarrow \cX$ over a neighborhood of $X_0$ are isomorphic to $\bP^{r-1}$ and a flat proper family of $\bP^{r-1}$ bundles is locally trivial in the analytic topology. Thus, after possibly shrinking $S$, $\cP_S\rightarrow \cX$ is a $\bP^{r-1}$-bundle.

    Let $[\cP_S]\in H^1(\cX, \PGL_r(\cO_{\cX}))$ be the class of the $\bP^{r-1}$-bundle $\cP_S$. We have the short exact sequence 
    \[
    1\rightarrow \mu_r\rightarrow \SL_r\rightarrow \PGL_r\rightarrow 1,
    \]
    with connecting homomorphism $c: H^1(\cX, \PGL_r(\cO_{\cX}))\rightarrow H^2(\cX, \mu_r)$. Then $c[\cP_S]|_{X_0} = c[\bP_0]\in H^2(X_0, \mu_r)$, as $G_0$ is a $\theta_r$-twisted sheaf such that $\det(G_0) = \cO_{X_0}$. Since after possibly shrinking $S$,
    \[
    H^2(\cX, \mu_r)\xrightarrow{\sim} H^2(X_0, \mu_r),
    \]
    $c[\cP_S]|_{X_0} = c[\bP_0] = [\theta_r]$.

    By \cite[Example 7.3.1]{Markman2n}, $\cP_S$ lifts to a $\theta_r$-twisted locally free sheaf $\cE^{\prime}$ on $\cX$ such that $\det(\cE^{\prime}) = \cO_{\cX}$ and $\bP(\cE^{\prime})\cong \cP_S$. On $X_0$, $\cE^{\prime}$ and $G_0$ are both $\theta_r$-twisted locally free sheaves such that $\bP(G_0)\cong \bP(\cE^{\prime}|_{X_0})$. Therefore, there exists a line bundle $L$ on $X_0$ such that $\cE^{\prime}|_{X_0}\cong G_0\otimes L$. Since the determinants of the sheaves are trivial, $L^{\otimes r}\cong \cO_{X_0}$. As a result, the class $[L]\in H^1(X_0, \mu_r)$ lifts to a class $[\cL]\in H^1(\cX, \mu_r)$ and $L$ extends to a line bundle $\cL\rightarrow \cX$, for which $\cL^{\otimes r}\cong \cO_{\cX}$. Then define $\cE: = \cE^{\prime}\otimes \cL^{-1}$. It follows that $\det(\cE)\cong \cO_{\cX}$ and $\cE$ is flat over $S$ with $\cE|_{X_0}\cong G_0$. 
\end{proof}

We can now prove the main theorem of this section:

\begin{proof}[Proof of Theorem \ref{analytic Buchweitz-Flenner conjecture}]
  Let $A_k=\mathcal{O}_{S,0}/\mathfrak{m}^{k+1}$,
$S_k=\operatorname{Spec}A_k$, $X_k=\mathcal{X}\times_S S_k$,
$I_k=\mathfrak{m}^{k+1}/\mathfrak{m}^{k+2}$, so that $A_{k+1}\to A_k$ is a
small extension with kernel $I_k$. Let $\theta_r$ be the $\mu_r$-cocycle of the
gerbe of $r$-th roots of $\det(E_0)^{-1}$. By Lemma \ref{analytification lemma}(1) there
is a compatible system of smooth proper $A_k$-schemes $X^{\mathrm{alg}}_k$ with
$(X^{\mathrm{alg}}_k)^{\mathrm{an}}\cong X_k$, and by
Lemma \ref{analytification lemma}(2) there are compatible algebraic $\mu_r$-gerbes
$\tilde{X}^{\mathrm{alg}}_k\to X^{\mathrm{alg}}_k$ associated to the Brauer class $[\theta_r]$.
The sheaf $E_0$ is algebraic on the projective variety $X_0$ by GAGA, hence so
is $G_0$, as it is a $\theta_r$-twisted locally free sheaf on
$X^{\mathrm{alg}}_0=X_0$.

By assumption, for each $p$, the flat section $w_p$ of $R^{2p}f_*\mathbb{C}$
through $\kappa_p(E_0)$ lies in the Hodge subbundle
$F^pR^{2p}f_*\Omega^\bullet_{\cX/S}\subset
R^{2p}f_*\bC\otimes\cO_S$ at every point of $S$. By
Lemma \ref{analytification lemma}(3), $X^{\mathrm{alg}}_{k+1}/A_{k+1}$ satisfies \textbf{HdR} and
\[
  \Theta_{k+1}(\kappa_p(E_0)\otimes 1)\in
  F^pH^{2p}_{\mathrm{dR}}\big(X^{alg}_{k+1}/A_{k+1}\big)
\]
for all $k\geq 0$ and $p\geq 1$. Furthermore, $\kappa_p(E_0)=\ch_p(G_0)$.

We use induction to construct $\theta_r$-twisted locally free sheaves $G_k$ on
$X^{alg}_k$, flat over $A_k$, with
$G_{k+1}|_{X^{alg}_k}\cong G_k$. Since $X_0$ is projective, we have $G_0$ by construction. Suppose $G_k$ has been constructed. We apply Theorem \ref{Pridham map kills obstructions} to the small extension $A_{k+1}\to A_k$, to $X':=X^{\rm alg}_{k+1}$ with its gerbe $\tilde X':=\tilde X^{\rm alg}_{k+1}$, and to $G:=G_k$ on $\tilde X'\otimes_{A_{k+1}}A_k=\tilde X^{ alg}_k$ (Lemma \ref{analytification lemma}(2)). The hypothesis of Theorem \ref{Pridham map kills obstructions}(2) is the conclusion of the previous paragraph, since $\ch_p(G_0)=\kappa_p(E_0)$ and the map $\Theta_{X'}$ of Theorem \ref{Pridham map kills obstructions} is the map $\Theta_{k+1}$ of Lemma \ref{analytification lemma}(3). Hence $(L_{G_0}\otimes\Id_{I_k})(o(G_k))=0$. As $E_0$ is locally free and semiregular, $L_{G_0}$ is injective by Lemma 2.23, hence so is $L_{G_0}\otimes\Id_{I_k}$, and $o(G_k)=0$. Theorem \ref{Pridham map kills obstructions}(1) now gives a $\theta_r$-twisted locally free sheaf $G_{k+1}$ on $\tilde X^{ alg}_{k+1}$, flat over $A_{k+1}$, with $G_{k+1}|_{\tilde X^{alg}_k}\cong G_k$.

For each $k$ the projectivization $\mathbb{P}(G_k)\to X^{\mathrm{alg}}_k$ is a
$\mathbb{P}^{r-1}$-bundle, and the isomorphisms
$G_{k+1}|_{X^{\mathrm{alg}}_k}\cong G_k$ induce the product
$\mathbb{P}(G_{k+1})\times_{A_{k+1}}A_k\cong\mathbb{P}(G_k)$. Analytifying
and using $(X^{\mathrm{alg}}_k)^{\mathrm{an}}=X_k$ from Lemma \ref{analytification lemma}(1), we obtain a compatible
system of analytic $\mathbb{P}^{r-1}$-bundles
$\bP_k:=\mathbb{P}(G_k)^{\mathrm{an}}\to X_k$ with
$\bP_0=\mathbb{P}(G_0)=\mathbb{P}(E_0)$ and with Brauer class $[\theta_r]$. Thus, we get the input of Lemma \ref{semiuniversal deformation in the twisted case}, which yields, after possibly shrinking $S$,
a $\theta_r$-twisted locally free sheaf $\cE$ on $\cX$, flat
over $S$, with $\det\cE\cong\cO_{\cX}$ and
$\cE|_{X_0}\cong G_0$.

The sheaf $\cE^{\otimes r}$ is an untwisted locally free sheaf on
$\cX$, as the cocycle $\theta_r^{\,r}$ is trivial, so
$\ch(\cE_s^{\otimes r})$ is the parallel transport of
$\ch(G_0^{\otimes r})$ along $S$. Taking the unique $r$-th root
with degree-zero term $r$ (Definition \ref{kappa class definition}), and using
$\det\cE_s\cong\cO$, we obtain that
$\ch(\cE_s)=\kappa(\cE_s)$ is the parallel
transport of $\ch(G_0)=\kappa(E_0)$.
\end{proof}

\section{Simple semihomogeneous vector bundles on products of abelian varieties}\label{semihomogeneous section}
We now describe a construction of derived equivalences of Kummer-type IHSMs following Magni in \cite{Magni}. The construction begins with semihomogeneous vector bundles on abelian varieties.

Semihomogeneous vector bundles are those which are compatible with two natural classes of deformations over an abelian variety: The pullback by the translation map and tensorization by line bundles. They are important in the study of vector bundles on abelian varieties and provide an intriguing interpretation as the analogue on abelian varieties to modular sheaves on IHSMs (see \cite[Remark 3.5(2)]{Markman1}).

Let $\Lambda$ be a lattice of rank $2n$ and let $V: = \Lambda\otimes_{\mathbb{Z}}\mathbb{R}$. We have the complex torus $A = (V/\Lambda, J)$ with complex structure $J$ and $\widehat{A}: = (\widehat{V}/\widehat{\Lambda}, -J^t)$ the dual complex torus. An \textit{abelian variety} is a pair $(A, \lambda)$, where $\lambda: = c_1(L)$, for $L\rightarrow A$ a positive line bundle. We refer to $\lambda$ as a polarization and simply write $A$ to denote $(A, \lambda)$. For $\lambda: A\rightarrow \widehat{A}$ a polarization of $A$, the exponent $\mathrm{exp}(\lambda)$ is the smallest natural number such that $\mathrm{ker}(\lambda)\subset \mathrm{ker}(\cdot\, n)$. Recall $\widehat{A}\cong \mathrm{Pic}^0(A)$, the moduli space of degree $0$ line bundles on $A$. Recall also that there exists a universal line bundle $\mathcal{P}\rightarrow A\times \widehat{A}$ known as the \textit{Poincar\'{e} bundle} that parametrizes degree zero line bundles on $A$. The Poincar\'{e} bundle is characterized by the properties that $\mathcal{P}_{\{0\}\times \widehat{A}}\cong \mathcal{O}_{\widehat{A}}$ and $\mathcal{P}_{A\times {\{\alpha}\}}$ is isomorphic to the line bundle on $A$ corresponding to the point $\alpha\in \widehat{A}$.
\begin{Def}
    Let $\mathcal{F}$ be a vector bundle on an abelian variety $A$. The \textit{slope} $\mu(\mathcal{F})$ of $\mathcal{F}$ is defined as 
    $$
    \mu(\mathcal{F}): = \frac{\mathrm{det}(\mathcal{F})}{\mathrm{rank}(\mathcal{F})}.
    $$
    We may consider the slope as an equivalence class $[\mathcal{L}]\otimes \frac{1}{\ell} \in NS(A)\otimes \mathbb{Q}$, where $\ell$ is a positive integer. 
\end{Def}

\begin{Def}
    A vector bundle $\mathcal{F}$ on $A$ is said to be \textit{semihomogeneous} if for each closed point $a\in A$, there exists a line bundle $\mathcal{L}_a\in \widehat{A}: = \mathrm{Pic}^0(A)$ such that $t^*_a\mathcal{F} \cong \mathcal{F}\otimes \mathcal{L}_a$. 
\end{Def}

\begin{Rem}
    In general, semihomogeneous vector bundles arise as follows: Let $f: A\rightarrow B$ be an isogeny of abelian varieties of degree $m$, and let $\mathcal{L}$ be a line bundle on $A$. Then $\mathcal{F}: = f_*\mathcal{L}^{\otimes m}$ is a semihomogenous vector bundle on $B$.
\end{Rem}

Let $\mu = [\mathcal{L}]\otimes \frac{1}{\ell}\in NS(A)_{\mathbb{Q}}$, and let 
\begin{equation}
    \phi_{\mathcal{L}}: A\longrightarrow \widehat{A}
\end{equation}
be the homomorphism given by $\phi_{\mathcal{L}}(a) = t^*_a\mathcal{L}\otimes \mathcal{L}^{-1}$.
Define
\begin{equation}
\Phi_{\mu}: = \mathrm{im}((\ell, \phi_{\mathcal{L}}): A\rightarrow A\times \hat{A}),
\end{equation}
and denote by $\mathrm{pr}_1: \Phi_{\mu}\rightarrow A$ the projection onto the first factor. Let 
\begin{equation}\label{semihomogeneous kernel of projection map}
    \Sigma_{\mu}: = \mathrm{ker}(\mathrm{pr}_1).
\end{equation}
\begin{Prop}[{\cite[Lemma 4.9]{Orlov}}]
    Let $\mathcal{F}$ be a simple semihomogeneous vector bundle. Then $(a, \alpha)\in \Phi_{\mu}$ if and only if $t^*_a\mathcal{F}\cong \mathcal{F}\otimes \mathcal{P}_{\alpha}$, where $\mathcal{P}_{\alpha}: = \mathcal{P}_{A\times\{\alpha\}}$ is the fiber of the Poincar\'{e} bundle $\mathcal{P}\rightarrow X\times \widehat{X}$ representing $\alpha$.
\end{Prop}

We use the following stability property in Section \ref{section 6 slope-stability}: 

\begin{Prop}[{\cite[Section 6]{Mukai1}}]\label{simple semihomogeneous bundles are mu stable}
    A simple semihomogeneous vector bundle on an abelian variety is $\mu$-stable with respect to every ample class.
\end{Prop}

In \cite{Magni}, Magni constructs a simple semihomogeneous vector bundle on the product of abelian varieties that will be of interest to us. Recently, the results of \cite{YuxuanPaper} provide sufficient criteria for derived equivalences of abelian varieties to be lifted to derived equivalences of generalized Kummer varieties and results in the construction of many such derived equivalences; see \cite[Theorem 2.9]{YuxuanPaper}, in particular.

Let $K_{n+1}: = \mathrm{ker}(\Sigma: \mathbb{Z}^{n+1}\rightarrow \mathbb{Z})$ be the kernel of the summation map and let
\begin{equation}\label{definition of W}
    W: = A\otimes K_{n+1}
\end{equation}
be the $2n$-dimensional abelian variety that is the kernel of the summation map $A^{n+1}\rightarrow A$. We abuse notation slightly and denote by $\widehat{W}: = \widehat{A}\otimes K_{n+1}$ (not $\widehat{A\otimes K_{n+1}}$).
\begin{comment}
\begin{Lem}
    There is a natural $\mathscr{S}_{n+1}$-equivariant isomorphism 
    \[
\widehat{(A\otimes_{\mathbb{Z}}K_{n+1})}\cong \widehat{A}\otimes_{\mathbb{Z}} \widehat{K}_{n+1}. 
    \]
\end{Lem}
\begin{proof}
    There is a natural identification of $\mathbb{Z}^n$ with $K_{n+1}$  Therefore, we may identify $A\otimes K_{n+1}$ with $A^n$.
\end{proof}
\end{comment}
For abelian varieties $A$ and $B$ and $g: A\rightarrow B$ a homomorphism, let $\widehat{g}: \widehat{B}\rightarrow \widehat{A}$ denote the dual homomorphism. Let $f: A\times \widehat{A}\rightarrow B\times \widehat{B}$ be an isomorphism of abelian varieties and write 
\[
f: = \begin{pmatrix}
f_1 & f_2\\ f_3 & f_4
\end{pmatrix}, \hspace{10mm}
\tilde{f}: = \begin{pmatrix}
    \widehat{f}_4 & -\widehat{f}_2\\
    -\widehat{f}_3 & \widehat{f}_1.
\end{pmatrix}
\]
Here, we have identified $A\cong \widehat{\widehat{A}}$ and $B\cong \widehat{\widehat{B}}$ (see \cite[Remark 9.12]{Huybrechts3} for an important nuance with respect to this identification).
\begin{Def}
    Define the \textit{Mukai-Polishchuk} group as the group of symplectomorphisms
    \[
    \mathrm{Sp}(A, B): = \{f: A\times \widehat{A}\xrightarrow{} B\times \widehat{B}: \tilde{f} = f^{-1}\}.
    \]
    When $A = B$, we write $\mathrm{Sp}(A): = \mathrm{Sp}(A, A)$. For $G$ a finite group acting on $A$, we denote by $\mathrm{Sp}(A)^G$ the Mukai-Polishchuk group of $G$-invariant symplectomorphisms, i.e. $f$ such that $\sigma^*f = f$ for all $\sigma\in G$.
\end{Def}

Fix $n\geq 2$. Let us assume that $A$ is an abelian variety with $\End(A)\cong \bZ$ such that there exists a polarization $\lambda: A\rightarrow \widehat{A}$ of type $(d_1, \dots, d_g)$ whose exponent is coprime to $n+1$ and let $\lambda^{\prime}$ be the polarization $\lambda^{\prime}: \widehat{A}\rightarrow \widehat{\widehat{A}}\cong A$ such that $\lambda \circ \lambda^{\prime} = [\mathrm{exp}(\lambda)]$ and let $\lambda^{\delta}: = d_1\lambda^{\prime}$ be the dual polarization of $\lambda$. Let 
\begin{equation}
    \phi_0: K_{n+1}\rightarrow \widehat{K}_{n+1}
\end{equation}
be the canonical map and let $\widehat{\phi}_0$ denote its dual. The equalities $\lambda\circ \lambda^{\delta} = \lambda^{\delta}\circ \lambda = \cdot[\mathrm{exp}(\lambda)]$ and $\phi_{0}\circ \widehat{\phi}_0 = \widehat{\phi}_0\circ \phi_0 = \cdot [n+1]$ follow from \cite[Proposition 3.6]{Magni} and \cite[Section 1.9]{Magni}.

\begin{Prop}\cite[Proposition 4.7(ii)]{Magni}
   Let $n\geq 2$, then there is an injection 
   \[
   \Gamma_0((n+1)\mathrm{exp}(\lambda))\hookrightarrow \mathrm{Sp}(W)^{\mathscr{S}_{n+1}}, 
   \]
   where $\Gamma_0((n+1)\exp(\lambda))\subset \SL_2(\mathbb{Z})$ is the Hecke congruence subgroup of level $(n+1)\exp(\lambda)$.
\end{Prop}

\begin{Prop}
    The group $\mathrm{Sp}(W)^{\mathscr{S}_{n+1}}$ has the form
    \[
    \mathrm{Sp}(W)^{\mathscr{S}_{n+1}} = \left\{f = \begin{pmatrix} a_1\cdot (\Id\otimes \Id) & a_2\cdot (\lambda^{\delta}\otimes \widehat{\phi}_0)\\ a_3\cdot (\lambda\otimes \phi_0) & a_4\cdot (\Id\otimes \Id) \end{pmatrix} \bigg | \ a_i\in \mathbb{Z}, \ f^{-1} = \tilde{f}\right \}.
    \]
    Furthermore, the condition that $f^{-1} = \tilde{f}$ is equivalent to the condition that $\mathrm{det}(f)=1$.
\end{Prop}
\begin{proof}
    Let $W^{\vee}: = \widehat{A}\otimes K_{n+1}^{\vee}$. We have the following isomorphisms 
    \begin{align*}
    \Hom(W, W)^{\mathscr{S}_{n+1}}\cong (\bZ\cdot \mathrm{Id})\otimes_{\mathbb{Z}} (\mathbb{Z}\cdot \mathrm{Id}),\\
    \mathrm{Hom}(W, W^{\vee})^{\mathscr{S}_{n+1}}\cong (\mathbb{Z}\cdot \lambda)\otimes_{\mathbb{Z}} (\mathbb{Z}\cdot \phi_0),\\
        \mathrm{Hom}(W^{\vee}, W)^{\mathscr{S}_{n+1}}\cong (\mathbb{Z}\cdot \lambda^{\delta})\otimes_{\mathbb{Z}} (\mathbb{Z}\cdot \widehat{\phi}_0),\\
            \mathrm{Hom}(W^{\vee}, W^{\vee})^{\mathscr{S}_{n+1}}\cong (\mathbb{Z}\cdot \mathrm{Id})\otimes_{\mathbb{Z}} (\mathbb{Z}\cdot \mathrm{Id}),
    \end{align*}
    by \cite[Proposition 1.7, Proposition 1.11]{Magni}, and therefore the group is of the claimed structure by definition. Now there exists a group homomorphism
    \[
    \mathrm{Sp}(W)^{\mathscr{S}_{n+1}}\hookrightarrow \mathrm{GL}(2, \mathbb{Z})
    \]
    \[
    f\mapsto \begin{pmatrix} a_1 & a_2\\ \mathrm{exp}(\lambda)(n+1)\cdot a_3& a_4\end{pmatrix},
    \]
    due to the composition properties of $\lambda$ and $\phi_0$, and therefore, the lower left entry must be congruent to $0 \ \mathrm{mod} \ \mathrm{exp}(\lambda)(n+1)$. By construction, the element $\tilde{f}$ is given by 
    \[
    \tilde{f} = \begin{pmatrix} a_4\cdot (\mathrm{Id}\otimes \mathrm{Id}) & -a_2\cdot (\lambda^{\delta}\otimes \widehat{\phi}_0)\\ -a_3\cdot (\lambda\otimes \phi_0) & a_1\cdot (\mathrm{Id}\otimes \mathrm{Id}) \end{pmatrix},
    \]
    so we have the relation that $f^{-1} = J^{-1}f^tJ = \tilde{f}$, where $J = \begin{pmatrix}
        0 & \mathrm{Id}\\ -\mathrm{Id} & 0
    \end{pmatrix}$, which indicates that $f$ is represented by a symplectic matrix. It then follows that the image matrix has determinant $1$. Note that this in turn implies that $\mathrm{gcd}(\mathrm{exp}(\lambda), n+1) = 1$.
\end{proof}
\begin{Rem}
    We note that if $n+1$ is odd, there exists an abelian surface $A$ such that $\End(A)\cong \bZ$ and a polarization $\lambda: A\rightarrow \widehat{A}$ such that $\mathrm{gcd}(\mathrm{exp}(\lambda), n+1) = 1$.
\end{Rem}

\begin{Th}[{\cite[Theorem 2, Theorem 3, Theorem 4]{Magni}}]\label{Magni main theorems} Assume that $n\geq 2$ is even and $\mathrm{gcd}(n+1, [\mathrm{exp}(\lambda)]) = 1$. Let $\mathscr{S}_{n+1}$ be the symmetric group on $n+1$ elements.
\begin{enumerate}
            \item[i)] 
    There exists a short exact sequence

              \begin{equation}
        0\rightarrow \mathbb{Z}\times A[n+1]\rightarrow \mathrm{Aut}(D^b(W))^{\mathscr{S}_{n+1}}\rightarrow \mathrm{Sp}(W)^{\mathscr{S}_{n+1}}\rightarrow 0
           \end{equation}

           \item[ii)] The group of symplectomorphisms
           \[
           \mathrm{Sp}(W, \widehat{W})^{\mathscr{S}_{n+1}}\neq \emptyset
           \]
           and is a right torsor under $\mathrm{Sp}(W)^{\mathscr{S}_{n+1}}$. 
           \item[iii)] The set of invariant derived equivalences
           \[
           \mathrm{Eq}(D^b(W), D^b(\widehat{W}))^{\mathscr{S}_{n+1}}\neq \emptyset
           \]
           and is a right torsor under $\mathrm{Aut}(D^b(W))^{\mathscr{S}_{n+1}}$.
        \end{enumerate}
 
\end{Th}

Assume that $\mathrm{gcd}(n+1, [\mathrm{exp}(\lambda)]) =1$ and pick integers $m_1$ and $m_2$ such that $m_1[\mathrm{exp}(\lambda)]-(n+1)m_2 =1$. Consider the element 
\[
g: =\begin{pmatrix}
\frac{1}{n+1}\lambda \otimes \phi_0 & -\frac{1}{n+1} \mathrm{id}\otimes\phi_0\\
-\frac{1}{n+1} \mathrm{id}\otimes\phi_0 & \frac{m_1}{n+1}\lambda^{\delta}\otimes \phi_0
\end{pmatrix}\in \Hom(W\times \widehat{W}, (W\times \widehat{W})^{\vee}).
\]

The element $g$ corresponds via Theorem \ref{Magni main theorems} to a line bundle $\mathcal{L}\in \mathrm{Pic}(W\times \widehat{W})$ such that $\phi_{\mathcal{L}} = (n+1)g$. The semihomogeneous vector bundle we aim to study is given by 
\begin{equation}
    \mathcal{F}: = [n+1]_*\mathcal{L}^{\otimes n+1}.
\end{equation}

We describe its $\mathscr{S}_{n+1}$-equivariance properties in the next section. 

\section{Semiregular vector bundles on products of Kummer-type IHSMs}\label{section 4 semiregularity bundles}
In this section, we provide the background for understanding derived equivalences of generalized Kummer varieties by means of equivariant derived categories. We start by considering an arbitrary projective variety $X$.
\subsection{Equivariant derived categories}
Let $G$ be a finite group acting on a smooth projective variety $X$. For an object $\mathcal{E}\in D^b(X)$, a \textit{$G$-linearization} on $\mathcal{E}$ is a collection of isomorphisms
\begin{equation}\label{linearization}
    \lambda_g: \mathcal{E}\xrightarrow{\sim} g^*\mathcal{E},
\end{equation}
for each $g\in G$, such that $\lambda_e = \mathrm{Id}_{\mathcal{E}}$ and $\lambda_{gh} = h^*\lambda_g\circ \lambda_h$. We denote a $G$-linearized object by the pair $(\mathcal{E}, \lambda)$.

The triangulated category 
\begin{equation}\label{equivariant derived category}
    D^b_G(X)
\end{equation}
is the derived category consisting of $G$-linearized objects in $D^b(X)$. We refer to $D^b_G(X)$ as the \textit{$G$-equivariant derived category} of $X$.

\label{Schur multiplier condition}
    As a bridge between the derived categories $D^b(X)$ and $D^b_G(X)$, we consider the category of \textit{$G$-invariant} objects in $D^b(X)$:

    \begin{Def} We say that an object $\mathcal{E}\in D^b(X)$ is $G$-invariant if $g^*\mathcal{E}\cong \mathcal{E}$ for each $g\in G$.
    \end{Def}
    Clearly, a necessary condition for an object in $D^b(X)$ to be $G$-equivariant is that it is $G$-invariant. For a simple $G$-invariant object, the obstruction to carrying a $G$-linearization is a class $c\in H^2(G, \mathbb{C}^*)$, the Schur multiplier of $G$, as worked out in \cite{Ploog}.

\subsection{Inflation} Let $G$ be a finite group acting on projective varieties $X$ and $Y$. Denote by $G_{\Delta}\subset G\times G$ the diagonal subgroup (we include ``$\Delta$" to emphasize the diagonal action). We write $\mathrm{Eq}(D^b(X), D^b(Y))^{G_{\Delta}}$ to denote $G_{\Delta}$-invariant equivalences, and $\mathrm{Eq}(D^b_G(X), D^b_G(Y))$ to denote equivalences of triangulated categories \eqref{equivariant derived category}. The elements of the latter consist of Fourier-Mukai kernels that carry a $(G\times G)$-equivariant structure.

\begin{Def}
    Define the intermediary set 
    \begin{equation}
        \mathrm{Eq}_{G_{\Delta}}(D^b(X), D^b(Y)): = \{(\mathcal{G}, \lambda)\in D^b_{G_{\Delta}}(X\times Y): \mathrm{FM}_{(\mathcal{G}, \lambda)}: D^b(X)\xrightarrow{\sim} D^b(Y)\},
    \end{equation}
    which interposes between the sets $\mathrm{Eq}(D^b(X), D^b(Y))^{G_{\Delta}}$ and $\mathrm{Eq}(D^b_G(X), D^b_G(Y))$ in a manner we will explain below. 
\end{Def}
There is a forgetful map 
\[
 \mathrm{Eq}_{G_{\Delta}}(D^b(X), D^b(Y)) \longrightarrow \mathrm{Eq}(D^b(X), D^b(Y))^{G_{\Delta}}
\]
given by $(\cG, \lambda)\mapsto \cG$. By \cite[Theorem 6]{Ploog} (see \cite[Theorem 6.15(i)]{Magni}), its image is the kernel of a map 
\[
\delta_{X, Y}: \mathrm{Eq}(D^b(X), D^b(Y))^{G_{\Delta}}\longrightarrow H^2(G, \bC^{*}),
\]
which sends a simple $G_{\Delta}$-invariant kernel to the class obstructing the existence of a $G_{\Delta}$-linearization, and its fibers are torsors under $\Hom(G, \bC^*)$. In particular, if $H^2(G, \bC^*) = 0$, then every $G_{\Delta}$-invariant kernel admits a $G_{\Delta}$-linearization, unique up to a character of $G$. Furthermore, the sets $\mathrm{Eq}_{G_{\Delta}}(D^b(X), D^b(Y))$ and $\mathrm{Eq}(D^b_G(X), D^b_G(Y))$ are related by an inflation map
\begin{equation}\label{inflation map}
    \mathrm{inf}_{G_{\Delta}}^{G^2}(\mathcal{G}, \lambda): = \bigoplus_{[g]\in G_{\Delta}\backslash G^2} g^*\mathcal{G}.
\end{equation}
For notation, we write 
\begin{equation}
    G^2\cdot \mathcal{G}: =  \mathrm{inf}_{G_{\Delta}}^{G^2}(\mathcal{G}, \lambda).
\end{equation}
The $G^2$-linearization 
\begin{equation}\label{G^2-linearization}
\lambda^G_h: G^2\cdot \mathcal{G}\xrightarrow{\sim} h^*(G^2\cdot \mathcal{G})
\end{equation}
on the right-hand side is given by extending linearly the $\mathcal{G}$-linearization $\lambda_h$, for $h\in G_{\Delta}$, and permutation of the summands for $h\in \Delta G\backslash G^2$. We will find that for a general derived equivalence of generalized Kummer varieties, it will prove sufficient to have control over particular classes in $H^2(G, \mathbb{C}^*)$ known as \textit{Schur multipliers}.
\begin{Prop}[\cite{Ploog}]\label{inflation prop}
    For $G$ a finite group acting faithfully on smooth projective varieties $X$ and $Y$, there is an exact sequence of groups
    $$
    0\rightarrow Z(G)\rightarrow \mathrm{Aut}(D^b(X))^{hG}\rightarrow \mathrm{Aut}(D^b_G(X)),
    $$
    and a $G$-equivariant map of pseudo-torsors
    \begin{equation}\label{inflation correspondence}
    \mathrm{Eq}_{G_{\Delta}}(D^b(X), D^b(Y))\xrightarrow{\mathrm{inf}^{G^2}_{G_{\Delta}}} \mathrm{Eq}(D^b_G(X), D^b_G(Y)).
    \end{equation}
\end{Prop}
\begin{proof}
    See \cite[Theorem 6]{Ploog} and \cite[Theorem 6.15]{Magni}.
\end{proof}
\begin{Rem}
     The sets $\mathrm{Eq}(D^b(X), D^b(Y))^G$ and $\mathrm{Eq}(D^b_G(X), D^b_G(Y))$ are pseudo-torsors of $\mathrm{Aut}(D^b(X))^G$ and $\mathrm{Aut}(D^b_G(X))$ respectively. Let $G = \mathscr{S}_{n+1}$. Then $Z(G) = 0$ for $n\geq 2$, so the inflation map $\mathrm{inf}^{G^2}_{G}$ is injective, as the map
     \[
     \mathrm{Aut}(D^b(X))^{hG}\longrightarrow \mathrm{Aut}(D^b_G(X))
     \]
     is injective.
\end{Rem}
\subsection{The BKR correspondence}
Let $X$ be a projective variety and $G$ a finite group acting on $X$. 
\begin{Def}
    A \textit{$G$-cluster} on $X$ is a $G$-invariant zero-dimensional closed subscheme $Z$ on $X$ such that $H^0(Z, \mathcal{O}_Z)\cong \mathbb{C}[G]$ is an isomorphism of $G$-modules. 
\end{Def}
Note that $G$-clusters generalize $G$-invariant subschemes that are free $G$-orbits; every free $G$-orbit is clearly a $G$-cluster, but there are (necessarily non-reduced) $G$-clusters that are not free $G$-orbits. Let $G = \mathscr{S}_{n+1}$. We let $\mathrm{Hilb}^{\mathscr{S}_{n+1}}(A^{n+1})$ denote the irreducible component of the moduli space of $\mathscr{S}_{n+1}$ clusters on $A^{n+1}$ that contains all points corresponding to free orbits of the permutation action. Closed points of $\mathrm{Hilb}^{\mathscr{S}_{n+1}}(A^{n+1})$ therefore parametrize length-$(n+1)!$, $0$-dimensional $\mathscr{S}_{n+1}$-invariant subschemes $Z$ of $A^{n+1}$ such that $H^0(Z, \mathcal{O}_Z)$ is the $\mathscr{S}_{n+1}$-regular representation. The component $\mathrm{Hilb}^{\mathscr{S}_{n+1}}(A^{n+1})$ is a reduced component of the moduli space of $G$-clusters. 

\begin{Prop}[{\cite[Theorem 5.1]{Haiman01} and \cite[Section 4]{Haiman99}}]\label{Hilbert scheme isomorphic to moduli space of clusters}
The Hilbert scheme $A^{[n+1]}$ is isomorphic to the reduced component $\mathrm{Hilb}^{\mathscr{S}_{n+1}}(A^{n+1})$ of the moduli space of $\mathscr{S}_{n+1}$-clusters.
\end{Prop}

\begin{Rem}
In \cite{Haiman99} and \cite{Haiman01} the above result is stated for $X = \mathbb{C}^2$, but the argument goes through for $X$ a projective surface: For $\mathcal{Z}\subset \mathrm{Hilb}^{\mathscr{S}_{n+1}}(X^{n+1})\times X^{n+1}$ the universal family of $G$-clusters, the projection morphism $\mathrm{Id}\times \mathrm{pr}_1(\mathcal{Z}): \mathrm{Hilb}^{\mathscr{S}_{n+1}}(X^{n+1})\times X^{n+1}\rightarrow \mathrm{Hilb}^{\mathscr{S}_{n+1}}(X^{n+1})\times X$ is flat and hence gives rise to a classifying morphism $\mathrm{Hilb}^{\mathscr{S}_{n+1}}(X^{n+1})\rightarrow X^{[n+1]}$. This classifying morphism is inverse to a morphism $X^{[n+1]}\rightarrow \mathrm{Hilb}^{\mathscr{S}_{n+1}}(X^{n+1})$ which arises by means of the universal property of the Hilbert scheme (see \cite[Theorem 3.1]{Haiman01}, cf. \cite[Corollary 1.3]{BKR}).

\end{Rem}
Recall that we let \begin{equation} K_{n+1}: = \mathrm{ker}(\Sigma: \mathbb{Z}^{n+1}\longrightarrow \mathbb{Z}),\end{equation} where $\Sigma$ is the summation map and let $W$ and $\widehat{W}$ be defined as
\begin{equation}
    W: = A\otimes K_{n+1},\hspace{3mm} \widehat{W}: = \widehat{A}\otimes K_{n+1}.
\end{equation}
Note the slight abuse of notation for $\widehat{W}$; indeed, there is an isomorphism of $\widehat{A}\otimes K_{n+1}\cong \widehat{A}\otimes \widehat{K}_{n+1}$ as abelian varieties, yet they are not isomorphic as $\mathbb{Z}[\mathscr{S}_{n+1}]$-modules. 

Define $\mathrm{Hilb}^{\mathscr{S}_{n+1}}(W)$ analogously to $\mathrm{Hilb}^{\mathscr{S}_{n+1}}(A^{n+1})$, where we naturally identify $W\cong A^n$ via the identification of $K_{n+1}$ with $\mathbb{Z}^n$ given by $(a_1, \dots , a_n)\mapsto (a_1,\dots, a_n, -\sum_i a_i)$. That is, we identify $\mathrm{Hilb}^{\mathscr{S}_{n+1}}(W)$ with the connected component of free orbits of the $\mathscr{S}_{n+1}$-action on the factors of $W$. Note that the isomorphism of Proposition \ref{Hilbert scheme isomorphic to moduli space of clusters} implies the existence of a crepant resolution
\[
\mathrm{Hilb}^{\mathscr{S}_{n+1}}(A^{n+1})\longrightarrow A^{(n+1)},
\] 
the fiber of which over zero is $\mathrm{Hilb}^{\mathscr{S}_{n+1}}(W)\cong \Kum_n(A)$, and which is well known to be semi-small. Therefore we have the following:

\begin{Th}[{\cite[Theorem 1.1]{BKR}}]

    The derived McKay correspondence yields isomorphisms
    \[
    D^b(\Kum_n(A))\cong D^b(\mathrm{Hilb}^{\mathscr{S}_{n+1}}(W))\cong D^b_{\mathscr{S}_{n+1}}(W).
    \]
\end{Th}

Consider the following universal subschemes
\begin{center}
    \begin{tikzcd}
             & \mathcal{Z} \arrow[ld] \arrow[rd] &              &                     & \mathcal{\widehat{Z}} \arrow[ld] \arrow[rd] &                     \\
W &                                   & \Kum_{n}(A) & \widehat{W} &                                          & \Kum_{n}(\widehat{A})
\end{tikzcd}
\end{center}
of $\mathscr{S}_{n+1}$-clusters on $W$ and $\widehat{W}$. We define Haiman's isospectral Hilbert scheme $\mathcal{Z}\times \mathcal{\widehat{Z}}$ as the reduced fiber product

\begin{center}
    \begin{tikzcd}
\mathcal{Z}\times \widehat{\mathcal{Z}} \arrow[d, "p_1\times p_2"] \arrow[rr, "q_1\times q_2"] &  & \mathrm{Kum}_{n}(A)\times \mathrm{Kum}_{n}(\widehat{A}) \arrow[d] \\
W\times \widehat{W} \arrow[rr]                                                                                &  & W/\mathscr{S}_{n+1}\times \widehat{W}/\mathscr{S}_{n+1}       
\end{tikzcd}
\end{center}

When taking products, the derived McKay correspondence yields a derived equivalence by (\cite[Theorem 1.1]{BKR}):
\begin{Prop}\label{BKR equivalence of categories}
There is an equivalence of categories \[D^b_{\mathscr{S}_{n+1}\times \mathscr{S}_{n+1}}(W\times \widehat{W})\xrightarrow{\sim} D^b(\Kum_{n}(A)\times \Kum_{n}(\widehat{A})).\]
\end{Prop}
\begin{proof}
It follows immediately that $ \Kum_n(A)\times \Kum_n(\widehat{A})$ is a crepant resolution $X: = W/\mathscr{S}_{n+1}\times \widehat{W}/\mathscr{S}_{n+1}$. The canonical sheaf $\omega_{W\times \widehat{W}}\cong \omega_{W}\boxtimes\omega_{\widehat{W}}$ is locally trivial as a $G^2$-sheaf, and so the assumptions of \cite[Theorem 1.1]{BKR} are satisfied. 
\end{proof}

\begin{Rem}
Note that we use the convention where the Fourier-Mukai kernel associated to the derived McKay correspondence is $\mathcal{O}_{\mathcal{Z}\times \mathcal{\widehat{Z}}}$, as in \cite{Krug18}. So, for an object $\mathcal{G}\in D^b_{G_{\Delta}}(W\times \widehat{W})$, define 
\begin{equation}\label{BKR Magni}
\mathcal{G}^{K}: = (q_1\times q_2)^{G_{\Delta}}_*\left[(p_1\times p_2)^*(\cG)\otimes \mathcal{O}_{\mathcal{Z}\times \mathcal{\widehat{Z}}}\right],\end{equation}
where $p_1\times p_2: \mathcal{Z}\times \mathcal{\widehat{Z}}\rightarrow X$ and $q_1\times q_2: \mathcal{Z}\times \mathcal{\widehat{Z}}\rightarrow \Kum_{n}(A)\times \Kum_{n}(\widehat{A})$ are projection maps. Here, $(q_1\times q_2)^{G_{\Delta}}_*$ denotes the pushforward invariant under the diagonal action. The fiber of $\mathcal{G}^{K}$ at a point $Z\in \Kum_{n}(A)\times \Kum_{n}(\widehat{A})$ is $H^0(Z, \mathcal{G}|_{Z})^{G_{\Delta}}$, where $Z$ is viewed as a product of clusters $Z\subset W\times \widehat{W}$. In particular, if $\cG$ is a $G_{\Delta}$-equivariant vector bundle, then $\cG^K$ is a vector bundle of rank $|G|\cdot \rk(\cG)$, and it agrees with the BKR image of the inflation $G^2\cdot \cG$ (see Lemma \ref{local freeness}).
\end{Rem}

\subsection{Derived equivalences of Kummer-type IHSMs}
From now on, we will let $G: = \mathscr{S}_{n+1}$ and $G_{\Delta}\subset G\times G$ be 
the diagonal subgroup consisting of elements $(\sigma, \sigma)$, for $\sigma\in G$.
Let us provide an outline of Magni's strategy in \cite{Magni} for constructing derived equivalences of generalized Kummer varieties. The following is the main theorem that will be reached by the end of this subsection. We also remark that the results of \cite{Magni} are part of a more general construction described in \cite{YuxuanPaper}.

\begin{Th}[\cite{Magni}]\label{Magni derived equivalence}
    Let $n+1\geq 3$ be odd. There exist abelian surfaces $A$ and $\widehat{A}$ and a $G$-invariant simple semihomogeneous vector bundle $\mathcal{F}$ on $W\times \widehat{W}$ such that $G^2\cdot \mathcal{F} : = \mathrm{inf}^{G^2}_{G_{\Delta}}(\mathcal{F})\rightarrow W\times \widehat{W}$ is a $G^2$-equivariant Fourier-Mukai kernel of a derived equivalence 
    \[
    \Phi_{G^2\cdot \mathcal{F}}: D^b_{G}(W)\xrightarrow{\sim} D^b_{G}(\widehat{W}).
    \]
    Via BKR conjugation, the kernel yields a derived equivalence

    \[ \Phi_{\mathcal{F}^K}:D^b(\Kum_n(A))\xrightarrow{\sim} D^b(\Kum_n(\widehat{A})),
    \]
    where $\mathcal{F}^K\in D^b(\Kum_n(A)\times \Kum_n(\widehat{A}))$ is the image of $G^2\cdot \mathcal{F}$ under the equivalence of categories given in Proposition \ref{BKR equivalence of categories}.
\end{Th}

\begin{Def}\label{unitary bundle}
    A vector bundle $\mathcal{U}$ is \textit{unipotent} if it admits a filtration \[
    0 = \mathcal{U}_0\subset \mathcal{U}_1\subset \cdots \subset \mathcal{U}_m = \mathcal{U},
    \] such that for each $i$ from $1$ to $m$, $\mathcal{U}_i/\mathcal{U}_{i-1}\cong \mathcal{O}_X$.
\end{Def}
A semihomogeneous  vector bundle $\mathcal{G}$ on an abelian variety admits a decomposition via the Jordan-H\"{o}lder filtration as 
\begin{equation}\label{Jordan-Holder filtration}
    \mathcal{G}\cong \bigoplus_j \mathcal{F}_j\otimes \mathcal{U}_j,
\end{equation}
where each $\mathcal{F}_j$ is a simple semihomogeneous vector bundle and each $\mathcal{U}_j$ is unipotent, by \cite[Proposition 6.18]{Mukai1}. For ease of notation, define 
\begin{equation}
    X:= W\times \widehat{W}.
\end{equation}
We consider the Poincar\'{e} bundle $\mathcal{P}\rightarrow X\times \widehat{X}$.

In order to rewrite \eqref{Jordan-Holder filtration} in a manner that is computationally accessible, we must study the action of $(-)\otimes \mathcal{P}_{\alpha}$ for $\alpha\in \widehat{X}[n]$ on the bundle $\mathcal{G}$. By the property of semihomogeneous vector bundles that any such bundle $\mathcal{F}_j$ can be represented as a bundle $\mathcal{F}_0\otimes \mathcal{P}_{\alpha_j}$ (see \cite[Proposition 7.11]{Mukai1}), the action of $\widehat{X}[n+1]$ on each summand $\mathcal{F}_j$ has stabilizer $\{\alpha\in \widehat{X}[n+1]: \mathcal{F}_0\cong \mathcal{F}_0\otimes \mathcal{P}_{\alpha_j}\}$. This is indeed, the homogeneous subset of $\widehat{X}$, i.e. the subset $\Sigma_{\mu}\subset \widehat{X}[n+1]$ from \eqref{semihomogeneous kernel of projection map}. Magni proves the following: 
\begin{Prop}[{\cite[Proposition 7.14, Situation 7.15, Proposition 7.21]{Magni}}]\label{bundle decomposition}
    The Jordan-H\"{o}lder filtration in \eqref{Jordan-Holder filtration} admits a split unipotent part, i.e. for each summand $\mathcal{U}_j\cong \mathcal{O}_{X}^{r_j}$. Furthermore, the resulting split decomposition 
    \begin{equation}\label{split Jordan-Holder}
    \mathcal{F}\cong \bigoplus_{\alpha_j\in \widehat{X}[n+1]/\Sigma_{\mu}} (\mathcal{F}_0\otimes \mathcal{P}_{\alpha_j})^r
    \end{equation}
    contains at least one $G$-invariant summand. 
\end{Prop}
\begin{Rem}
    For $\mathcal{F}$ the simple semihomogeneous vector bundle in Theorem \ref{Magni derived equivalence}, note that by definition, $$\mathrm{inf}_{G_{\Delta}}^{G^2}(\mathcal{F}) = \bigoplus_{g\in G_{\Delta} \backslash G^2}g^*\mathcal{F}$$ is a direct sum of simple semihomogeneous vector bundles, each of whose slope is equal to that of $\mathcal{G}$.  
\end{Rem}

Magni therefore proves that there exists a simple semihomogeneous vector bundle $\mathcal{F}_0$ on $X$ which provides a derived equivalence between the abelian varieties $W$ and $\widehat{W}$. The next step is to prove this derived equivalence carries a $G_{\Delta}$-linearization and therefore can be enhanced to an equivalence of equivariant derived categories by way of Proposition \ref{inflation prop}. This requires that we have control over the projective representations of $G$. We have the following key proposition: 
\begin{Prop}[{\cite[Proposition 7.25]{Magni}}]\label{Conditions for equivariance}
Let $H$ be a finite group acting on a projective variety $X$ and let $\mathcal{E}\rightarrow X$ be a sheaf. Assume:
\begin{enumerate}
    \item $\mathcal{E}$ is simple;
    \item $\mathcal{E}$ is $H$-invariant;
    \item There exists $r\in \mathbb{N}$ such that $\mathcal{E}^r$ is $H$-equivariant;
    \item Every $r$-dimensional projective representation of $H$ is isomorphic to a linear representation of $H$.
    \end{enumerate}
Then, $\mathcal{E}$ is $H$-equivariant.

\end{Prop}

\begin{Prop}[{\cite[Propositions 7.24, 7.25, 7.26 6.11]{Magni}}]
    Let $n$ be even. The $G$-equivariant Fourier-Mukai kernel $\mathcal{F}^G$ can be inflated to yield a $G^2$-equivariant sheaf \begin{equation}\label{G^2F}G^2\cdot \mathcal{F}\in D^b_{G^2}(W\times \widehat{W})\end{equation} which provides an equivalence of categories 
    \[
    \Phi_{G^2\cdot \mathcal{F}}: D^b_G(W)\longrightarrow D^b_G(\widehat{W}).
    \]
\end{Prop}
\begin{proof}

Proposition \ref{Conditions for equivariance} implies that the sheaf $\mathcal{F}^G\rightarrow W\times \hat{W}$ is $G_{\Delta}\cong G$-equivariant and hence 
\[
\mathcal{F}^G\in D^b_{G_{\Delta}}(W\times \widehat{W}).
\] Every projective representation of $G$ is already linear if $n$ is even. This follows from the surjectivity of the map 
    \[
    \mathrm{Hom}(\mathscr{S}_{n+1}, \GL(r, \mathbb{C}))\rightarrow \mathrm{Hom}(\mathscr{S}_{n+1}, \PGL(r, \mathbb{C})),
    \]
    where $r = (n+1)^{2n-2}$ is odd,
    as worked out in \cite[Proposition 7.26]{Magni}. The class in $H^2(G, \mathbb{C}^*)$ that obstructs the existence of a $G$-linearization of $\mathcal{F}_0$ therefore vanishes and hence we can enhance the sheaf $\mathcal{F}_0$ to the sheaf $\mathcal{F}^G$. We can then apply the inflation map \eqref{inflation map} to acquire a sheaf $\inf(\mathcal{F}^G)_{G_{\Delta}}^{G^2}$ in $Eq(D^b_G(W), D^b_G(\widehat{W}))$. In the special case where $n = 2$, the cohomology group $H^2(\mathscr{S}_3, \mathbb{C}^*) = 0$, hence, the forgetful map $\mathrm{Eq}_{G_{\Delta}}(D^b(W), D^b(\widehat{W})\rightarrow \mathrm{Eq}(D^b(W), D^b(\widehat{W})^{G_{\Delta}}$ is surjective by \cite[Theorem 6.15(i)]{Magni} and therefore the $G_{\Delta}$-invariant kernel $\cF_0$ admits a $G_{\Delta}$-lienarization.  
\end{proof}

\begin{Lem}\label{pairwise non-isomorphic}
The summands of $G^2\cdot \mathcal{F} : = \mathlarger{\bigoplus}_{g\in G_{\Delta} \backslash G^2}g^*\mathcal{F}^G$ are pairwise non-isomorphic.
\end{Lem}
\begin{proof}
Every right coset of $G_\Delta$ in $G^2$ contains exactly one element of the form $(1,\tau)$ with $\tau\in G$, since $(\sigma,\sigma^{\prime})=(\sigma,\sigma)(1,\sigma^{-1}\sigma^{\prime})$. Since $F^G$ is $G_\Delta$-invariant, the summand $g^*F^G$ depends, up to isomorphism, only on the coset of $g$: for $\delta\in G_\Delta$ we have $(\delta g)^*F^G\cong g^*\delta^*F^G\cong g^*F^G$. So the summands are $(1,\tau)^*F^G$ for $\tau\in G$.

Suppose $(1,\tau)^*F^G\cong(1,\tau')^*F^G$ with $\tau\neq\tau'$. Then applying $(1,\tau^{-1})^*$ and using $(gh)^*\cong h^*g^*$, we obtain
\[
(1,\sigma)^*F^G\cong F^G,\hspace{5mm} \sigma:=\tau^{\prime}\tau^{-1}\neq 1.
\]
We show that this is a contradiction. Let $\Gamma_\sigma:=\{(\sigma(y),y):y\in\widehat{W}\}\subset\widehat{W}\times\widehat{W}$, so that $\Phi_{\cO_{\Gamma_\sigma}}\cong\sigma^*$. By base change along the automorphism,
\[
\Phi_{(1,\sigma)^*F^G}\cong\Phi_{\cO_{\Gamma_\sigma}}\circ\Phi_{F^G}.
\]
Hence, $\Phi_{\cO_{\Gamma_\sigma}}\circ\Phi_{F^G}\cong\Phi_{F^G}$. Since $\Phi_{F^G}$ is an equivalence, composing with its inverse gives $\Phi_{\cO_{\Gamma_\sigma}}\cong\mathrm{id}_{D^b(\widehat{W})}=\Phi_{\cO_\Delta}$. By the uniqueness of Fourier--Mukai kernels \cite[Theorem 5.14]{Huybrechts3}, $\cO_{\Gamma_\sigma}\cong\cO_\Delta$, so $\Gamma_\sigma=\Delta$, i.e. $\sigma(y)=y$ for every $y\in\widehat{W}$. However, $\sS_{n+1}$ acts faithfully on $\widehat{W}$. Hence $\sigma x\neq x$, a contradiction.
\end{proof}

\subsection{Properties of the kernel $\cF^K$}
  We next study the BKR image of the derived equivalence constructed in the previous section. Let $\mathcal{F}^G$ denote the simple semihomogeneous vector bundle on a product of abelian varieties $X_1$ and $X_2$ such that $\Phi_{\mathcal{F}^G}: D^b_G(X_1)\rightarrow D^b_G(X_2)$ is an equivalence of triangulated categories. In the sequel we will take $X_1 = W$ and $X_2 = \widehat{W}$, as in the previous section. Let \begin{equation} \mathcal{F}^K: = \mathrm{BKR}(G^2\cdot \mathcal{F}, \lambda)\end{equation} denote the image of $(\mathcal{F}^G, \lambda)$ under the Bridgeland-King-Reid equivalence, where we have chosen $\lambda$ to be the linearization induced by the trivial sign character.

\begin{Def}
    We denote by $K$ and $\widehat{K}$ the generalized Kummer varieties $\Kum_n(A)$ and $\Kum_n(\widehat{A})$ respectively. 
\end{Def}

In the sequel, let 
\begin{equation}
E^{\prime}: = G^2\cdot \cF = \bigoplus_{\tau\in G}(1, \tau)^*\cF^G, \hspace{5mm} R: = \rk(E^{\prime}) = (n+1)!\, \rk(\cF).
\end{equation}

Let 
\begin{equation}
\zeta: W\times \widehat{W}\rightarrow  Q: =(W/G)\times (\widehat{W}/G), \hspace{5mm} h: K\times \widehat{K}\rightarrow Q
\end{equation}
denote the quotient map and the product of the Hilbert-Chow morphisms respectively.

Let $W^{\circ}\subset W$ be the locus with trivial stabilizer. Its complement is codimension $2$ in $W$. Let $K^{\circ}\times \widehat{K}^{\circ}\subset K\times \widehat{K}$ be the open set of pairs of free orbits. By \cite{BKR}
and Lemma \ref{local freeness}(2), the restriction of $h$ given by $h^{\circ}: K^{\circ}\times \widehat{K}^{\circ}\xrightarrow{\sim} \zeta(W^{\circ}\times \widehat{W}^{\circ})$ is an isomorphism. Furthermore, 
\begin{equation}\label{definition of the map pi}
\pi: = (h^{\circ})^{-1}\circ \zeta|: W^{\circ}\times \widehat{W}^{\circ}\rightarrow K^{\circ}\times \widehat{K}^{\circ}
\end{equation}
is an \'etale cover of order $|G^2|$.

\begin{Lem}\label{local freeness}
    \begin{enumerate}
        \item $\cF^K$ is a locally free sheaf of rank $R$ on $K\times \widehat{K}$.
        \item Let $\pi: W^{\circ}\times \widehat{W}^{\circ}\rightarrow K^{\circ}\times \widehat{K}^{\circ}$ be the map as defined above. Then $\pi^*\cF^K\cong E^{\prime}$ on $W^{\circ}\times \widehat{W}^{\circ}$, and the isomorphism is $G^2$-equivariant.
    \end{enumerate}
\end{Lem}
\begin{proof}
Write $q:=q_1\times q_2\colon \cZ\times\widehat{\cZ}\to K\times\widehat{K}$ and $p:=p_1\times p_2\colon \cZ\times\widehat{\cZ}\to W\times\widehat{W}$.

(1): The morphism $q_1\colon\cZ\to K$ is flat, since $\cZ$ is a family of $G$-clusters, and it is proper with finite fibers, hence finite. Each fiber is a $G$-cluster $Z$ with $H^0(Z,\cO_Z)\cong\bC[G]$, so $q_1$ has degree $|G|$. The same holds for $q_2\colon\widehat{\cZ}\to\widehat{K}$, and therefore $q$ is finite and flat of degree $|G|^2$.

The sheaf $V:=p^*E'$ is a $G^2$-equivariant locally free sheaf of rank $R$ on $\cZ\times\widehat{\cZ}$. Since $q$ is finite and flat, $q_*V$ is locally free on $K\times\widehat{K}$. Since $q$ is affine, the fiber of $q_*V$ at a point $x$ is $H^0\big(q^{-1}(x),V|_{q^{-1}(x)}\big)$. By definition, $\cF^K=(q_*V)^{G^2}$ (cf.\eqref{BKR Magni}).

For each $\gamma\in G^2$, we have $q\circ\gamma=q$, so $\gamma$ acts on $q_*V$ by an $\cO_{K\times\widehat{K}}$-linear automorphism. Consider the averaging operator
\[
e:=\frac{1}{|G^2|}\sum_{\gamma\in G^2}\gamma\ \in\ \End_{\cO}(q_*V).
\]
Since $g\mapsto\gamma g$ is a permutation of $G^2$, we have $\gamma\circ e=e$ for every $\gamma$, so $e$ takes values in $(q_*V)^{G^2}$. Conversely, $e(t)=t$ for every invariant section $t$. Hence $e^2=e$ and $\operatorname{Im}(e)=(q_*V)^{G^2}=\cF^K$. Writing $s=e(s)+(s-e(s))$ shows that
\[
q_*V=\cF^K\oplus\ker(e),
\]
so $\cF^K$ is a direct summand of a locally free sheaf, and hence is locally free. Moreover, since $\cF^K$ is the image of the idempotent $e$, its fiber at $x$ is the image of $e\otimes k(x)$, namely
\[
\cF^K\otimes k(x)\cong H^0\big(q^{-1}(x),V|_{q^{-1}(x)}\big)^{G^2}.
\]

At a given point $x\in K^{\circ}\times \widehat{K}^{\circ}$, $q^{-1}(x)$ is a single free $G^2$-orbit $P$ of $|G^2|$ reduced points, so $H^0(q^{-1}(x),V)=\bigoplus_{y\in P}V(y)$, and $G^2$ permutes the summands freely and transitively. An invariant element is therefore determined by its component at one point $y\in P$, and every element of $V(y)$ extends uniquely to an invariant element. Hence $\rk(\cF^K)=\rk (V)=R$.

(2): Write $\pi:=\pi_1\times\pi_2$. The morphism
\[
j: W^\circ\rightarrow K^\circ\times W,\qquad w\mapsto(\pi_1(w),w),
\]
lands in $q_1^{-1}(K^\circ)$, since the cluster $\pi_1(w)=G\cdot w$ contains $w$. It is a closed immersion, being a graph between two finite flat $K^\circ$-schemes of degree $|G|$, and hence it is an isomorphism, with inverse $p_1$ and with $q_1\circ j=\pi_1$. The same holds for $\widehat{W}$. Under these identifications, $q$ becomes $\pi$ and $V$ becomes $E^{\prime}$, so
\[
\cF^K|_{K^\circ\times\widehat{K}^\circ}=(\pi_*E^{\prime})^{G^2}.
\]
Since $\pi$ is finite \'etale with Galois group $G^2$, the natural evaluation map $\pi^*(\pi_*E^{\prime})^{G^2}\to E^{\prime}$ is a $G^2$-equivariant isomorphism.
\end{proof}

\begin{comment}
Further, recall that we define the trace map 
\begin{equation}
\Tr^0: \mathcal{H}om_{\mathcal{O}_X}(\mathcal{F}, \mathcal{F})\longrightarrow \mathcal{O}_X.
\end{equation}

\begin{Lem}\label{image of trace invariant}
    Suppose $\mathcal{F}$ is a simple semihomogeneous vector bundle on a product of abelian varieties $W\times \widehat{W}$ that carries a $G^2$-equivariant structure. Then for any $q\geq 0$, the image of 
    $$
    Tr: \mathrm{Ext}^q_{D^b(W\times \widehat{W})}(\mathcal{F}, \mathcal{F})\longrightarrow H^q(W\times \widehat{W}, \mathcal{O}_{W\times \widehat{W}})
    $$
    lies in $H^q(W\times \widehat{W}, \mathcal{O}_{W\times \widehat{W}})^{G^2}$.
\end{Lem}
\begin{proof}
    Since $\mathcal{F}$ carries a $G^2$-invariant structure, we have $\mathcal{F}\in D^b(X\times \widehat{X})^{G^2}$. By \cite[Theorem 5.9]{Mukai1}, we have an isomorphism 
    \begin{equation}
        \Tr^q: H^q(A^{n+1}\times \widehat{A}^{n+1},\, \mathcal{E}nd(\mathcal{F}))\xrightarrow[]{\sim} H^q(A^{n+1}\times \widehat{A}^{n+1}, \mathcal{O}_{A^{n+1}\times \widehat{A}^{n+1}}),  
    \end{equation}
    which implies that there is an isomorphism 
    \begin{equation}
        \mathrm{Ext}^q_{W\times \widehat{W}}(\mathcal{F}, \mathcal{F})^{G^2}\xrightarrow{\sim} H^q(W\times \widehat{W}, \mathcal{O}_{W\times \widehat{W}})^{G^2}.
    \end{equation}
    Therefore the image of the trace map is isomorphic to $H^q(X\times \widehat{X}, \mathcal{O}_{X\times \widehat{X}})^{G^2}$.
\end{proof}
\end{comment}

\section{Semiregularity of fully faithful functors}
In this section, we briefly investigate the semiregularity of fully faithful functors $\Phi_{\cP}: D^b(X)\rightarrow D^b(Y)$, where $X$ and $Y$ are smooth projective. Some of these results have been covered in \cite{Rienks}. In particular, we recover that any derived equivalence between $K$-trivial varieties has a semiregular kernel (\cite[Corollary 8.5]{Rienks}). 
Hochschild cohomology $\HH^*(X)$ acts naturally on Hochschild homology $\HH_*(X)$ by Yoneda composition: 
\begin{equation}\label{Yoneda composition}
    \HH^j(X)\times \HH_i(X)\rightarrow \HH_{i-j}(X),
\end{equation}
given by $(\mu, \nu)\mapsto \mu[i]\circ \nu$. The main result of this section is the following:

\begin{Th}\label{semiregularity fully faithful functors}
  Let $\cF\in D^b(X\times Y)$ be the kernel of a fully-faithful functor $\Phi_{\cF}: D^b(X)\rightarrow D^b(Y)$. The following are equivalent:
  \begin{enumerate}
      \item $\cF$ is semiregular\label{semiregularity condition 1} 
      
      \item\label{infinitesimal Torelli 1} The natural map 
  \[
  \rho_X^{\HH}:\HH^2(X)\rightarrow \bigoplus_i\Hom(\HH_i(X), \HH_{i-2}(X))
  \] defined by \eqref{Yoneda composition} is injective.

  \item\label{infinitesimal Torelli 2} The natural map $(\rho_1^{\HT}, \rho_2^{HT})$ is injective, for
 \begin{align*}
 \rho_1^{\HT}:H^1(\cT_X)\rightarrow \bigoplus_{p, q}\Hom(H^{p, q}(X), H^{p-1, q+1}(X))\\
 \rho_2^{\HT}: H^0(\bigwedge^2\cT_X)\rightarrow  \bigoplus_{p, q}\Hom(H^{p, q}(X), H^{p-2, q}(X)).
 \end{align*}

    \end{enumerate}
\end{Th}

\begin{Def}
    If $X$ satisfies the equivalent conditions \eqref{infinitesimal Torelli 1} and \eqref{infinitesimal Torelli 2} above, we say $X$ satisfies the \textit{derived infinitesimal Torelli theorem}.
\end{Def}

Before we prove Theorem \ref{semiregularity fully faithful functors}, we provide an elementary proof of the following corollary:

\begin{Cor}[{\cite[Lemma 8.2, Corollary 8.5]{Rienks}}]\label{K-trivial corollary}
    Let $X$ be a compact K\"ahler manifold such that $\omega_X\cong \cO_{X}$ and let $\cF\in D^b(X\times Y)$ be the kernel of a fully-faithful functor $\Phi_{\cF}: D^b(X)\rightarrow D^b(Y)$ (e.g. the kernel of a derived equivalence). Then $\cF$ is semiregular. 
\end{Cor}
\begin{proof}

We show that $X$ satisfies condition \eqref{infinitesimal Torelli 2} in the statement of Theorem \ref{semiregularity fully faithful functors}. Observe that condition \eqref{infinitesimal Torelli 2} holds if and only if the map $\HT^2(X)\rightarrow \End_{\bC}(H^*(X, \bC))$, given by $\xi \mapsto (\alpha \mapsto \xi \lrcorner\,\alpha)$ is injective.
        It therefore suffices to check that for each nonzero $\xi\in \HT^2(X)$, there exists at least one $\alpha\in H^*(X, \bC)$ such that $\xi \lrcorner \, \alpha \neq 0$. Let $\Vol_X \in H^0(\omega_X)$ be a holomorphic volume form. With respect to the decomposition of $\HT^2(X)$, write $\xi = \xi_0+\xi_1+\xi_2$, for $\xi_0\in H^0(\bigwedge^2\cT_X)$, $\xi_1\in H^1(\cT_X)$, and $\xi_2\in H^2(\cO_X)$.

    Contraction with the volume form $\Vol_X$ induces a natural sheaf-theoretic isomorphism \[\bigwedge^k\cT_X\cong \Omega^{n-k}_X,\] and hence isomorphisms $H^q(\bigwedge^k\cT_X)\cong H^q(\Omega^{n-k}_X)$. Since $\Vol_X$ is a global section, $\xi_i\lrcorner\, \Vol_X$ is the image of $\xi_i$ under this isomorphism. Therefore, $\xi_0\lrcorner\, \Vol_X\in H^0(\Omega^{n-2}_X)$, $\xi_1\lrcorner\, \Vol_X\in H^1(\Omega^{n-1}_X)$, and $\xi_2\lrcorner\, \Vol_X\in H^2(\omega_X)$, and, since these classes lie in distinct cohomological degrees, $\xi \lrcorner \, \Vol_X = 0$ if and only if $\xi_i\lrcorner \, \Vol_X = 0$, i.e. if and only if $\xi_i = 0$ for each $i$. 
\end{proof}

It follows in particular that the kernel $\cF^K$ is semiregular.

\subsection{Proof of Theorem \ref{semiregularity fully faithful functors}}

\begin{Lem}\label{conditions 2 and 3}
Conditions \eqref{infinitesimal Torelli 1} and \eqref{infinitesimal Torelli 2} in the statement of Theorem \ref{semiregularity fully faithful functors} are equivalent.  
\end{Lem}
\begin{proof}
    The result follows from the proof of C\u ald\u araru's conjecture (\cite[Theorem 1.4]{CRVdB12}): 

    Let $\theta^{\bullet}: = \HKR\circ \iota_{\sqrt{\td_X}}: \HT^*(X)\xrightarrow{\sim} \HH^*(X)$, and $\theta_{\bullet}: = (\td_{X} \wedge -)\circ \HKR: \HH_*(X)\xrightarrow{\sim }H\Omega_*(X)$. Then for $\xi\in \HT^*(X)$ and $\upsilon\in \HH_*(X)$, \cite[Theorem 1.4]{CRVdB12} gives
    \[
    \theta_{\bullet}(\rho_X^{\HH}(\theta^{\bullet}\xi)(\upsilon)) = \xi\lrcorner\, \theta_{\bullet}(\upsilon),
    \]
    which implies that $\rho_X^{\HH}(\theta^{\bullet}\xi)(\upsilon) = \theta^{-1}_{\bullet}(\xi \lrcorner \, \theta_{\bullet}(\upsilon))$, since $\theta_{\bullet}$ is an isomorphism. Therefore, $\rho_{X}^{\HH}(\theta^{\bullet}\xi) = 0$ if and only if $\xi\lrcorner\, (-) = 0$ as an endomorphism of $H^*(X, \bC)$, since both $\theta^{\bullet}$ and $\theta_{\bullet}$ are injective. Contraction writh $\sqrt{\td_X}$ preserves total degree, so $\theta^{\bullet}$ restricts to an isomorphism $\HT^2(X)\cong \HH^2(X)$. Thus, restricting to degree $2$ yields the result, as condition \eqref{infinitesimal Torelli 2} holds if and only if the natural map $\HT^2(X)\rightarrow \End_{\bC}(H^*(X, \bC))$ is injective. 
\end{proof}

\begin{proof}[Proof of Theorem \ref{semiregularity fully faithful functors}]
    $(1)\iff (2)$ can be deduced from \cite[Section 8]{Rienks}. Then combine this with Lemma \ref{conditions 2 and 3}.
\end{proof}

\section{Derived equivalences of IHSMs and the Lefschetz standard conjecture}\label{section 6 LSC and derived equivalences}

\subsection{The Lefschetz standard conjectures}\label{standard conjecture subsection}

Let us recall some basic facts concerning the Lefschetz standard conjectures. For any compact K\"{a}hler manifold $X$ of dimension $2n$, there exists an $\mathfrak{sl}_2$-representation on the total cohomology $H^*(X, \mathbb{Q})$. For $X$ projective and for a fixed ample class $f\in H^{1, 1}(X, \mathbb{Q})$, the $\mathfrak{sl}_2$-algebra is defined by the triple $(L_f, h, \Lambda_f)$. Here, $L_f$ denotes the Lefschetz operator which acts by $(-)\cup f$, $h\in \End(H^*(X, \mathbb{Q}))$ acts by $(k-2n)\cdot \Id$ on $H^k(X, \mathbb{Q})$, and $\Lambda_f \in \mathrm{End}(H^*(X, \mathbb{Q}))$ is the unique operator such that
\begin{equation}\label{adjoint lefschetz operator}
[L_f, \Lambda_f] = h, \hspace{5mm} [h, \Lambda_f] = -2\Lambda_f.
\end{equation}
What's more, for all $k\leq 2n$ the Hard Lefschetz theorem yields isomorphisms
\[
L^{2n-k}_f: H^{k}(X, \mathbb{Q})\xrightarrow{\sim} H^{4n-k}(X, \mathbb{Q}).
\]
When working in the complex setting, all of the Lefschetz standard conjectures are implied by the following: 

\begin{Conj}[Grothendieck]\label{LSC}
  The dual Lefschetz operator $\Lambda_f\in \End (H^*(X, \bQ))$ is induced by an algebraic correspondence.
\end{Conj}
\begin{Lem}\label{lemma on equality of conjectures 1.1 and 5.1}
    If $\Lambda_f$ is algebraic, then Conjecture \ref{BX} holds for $X$, where $\alpha =f$.
\end{Lem}
\begin{proof}
    $\Lambda_f^{2n-k}$ restricts to an isomorphism $H^{4n-k}\xrightarrow{\sim} H^k$ of $\mathfrak{sl}_2$-weight spaces of weights $\pm(2n-k)$. $\Lambda^{2n-k}_f$ is induced by an algebraic cycle of codimensional $k$, so the claim follows from \cite[Lemma 6]{Charles} (cf. \cite[Theorem 4.1]{Kleiman94}).
\end{proof}
\begin{Rem}
We again remark that for the purposes of this paper, and in light of Lemma \ref{lemma on equality of conjectures 1.1 and 5.1}, we refer to Conjecture \ref{LSC} as \textit{the} Lefschetz standard conjecture (LSC). 
\end{Rem}

\subsection{Derived equivalences and the LSC}
In this section we describe a series of results of Markman, Verbitsky, and Taelman and how they are relevant to our problem. Throughout this section we assume that $X$ and $Y$ are deformation equivalent IHSMs each of dimension $2n$.

\begin{Def}
    A Hodge isometry 
    \[
    f: H^*(X, \mathbb{Q})\longrightarrow H^*(Y, \mathbb{Q})
    \]
    with respect to the Mukai pairings on $H^*(X, \mathbb{Q})$ and $H^*(Y, \mathbb{Q})$ is said to be \textit{degree-reversing} if for each $k$, the image of $H^k(X, \mathbb{Q})$ under $f$ is contained in $H^{4n-k}(Y, \mathbb{Q})$.
\end{Def}

Let $\mathcal{F}$ be an object of positive rank in the bounded derived category of coherent sheaves $D^b(X\times Y)$ and assume that $\mathcal{F}$ is the Fourier-Mukai kernel of an equivalence of triangulated categories
\begin{equation}
\Phi_{\mathcal{F}}: D^b(X)\longrightarrow D^b(Y).
\end{equation}

Recall the definition of the characteristic class $\kappa$ from equation (\ref{twisted kappa}). We denote by 
\begin{equation}\label{phi kappa}
    \phi_{\kappa}: = \left[\kappa(\mathcal{F})\sqrt{\td_{X\times Y}}\right]_*: H^*(X, \mathbb{Q})\longrightarrow H^*(Y, \mathbb{Q})
\end{equation}
the Hodge isometry associated to $\kappa(\mathcal{F})\sqrt{td_{X\times Y}}$. We show that $\phi_{\kappa}$ is a degree-reversing Hodge-isometry.

\begin{Rem}
    Let us first remark on how to compute the isometry $\phi_{\kappa}$: Let $\pi_X$ and $\pi_Y$ denote the projections from the product $X\times Y$. Let $\lambda_X\in H^2(X, \mathbb{Q})$ and $\lambda_Y\in H^2(Y, \mathbb{Q})$ be classes such that $c_1(\mathcal{F}) = \pi^*_X(\lambda_X)+\pi^*_Y(\lambda_Y)$, and let $r>0$ be the rank of $\mathcal{F}$. Then 
    \begin{align*}
    \phi_{\kappa}(-) = \left[\ch(\mathcal{F})\,\mathrm{exp}\left(\frac{-c_1(\mathcal{F})}{r}\right)\sqrt{td_{X\times Y}}\right]_*(-) \\= \mathrm{exp}\left(\frac{-\lambda_Y}{r}\right)\cup \left[\ch(\mathcal{F})\sqrt{td_{X\times Y}}\right]_*\left(\mathrm{exp}\left(\frac{-\lambda_X}{r}\right)\cup (-)\right).
    \end{align*}

    Let us define 
    \begin{equation}
        B_{\lambda}: = \mathrm{exp}(\lambda), 
    \end{equation}
    for $\lambda\in H^2(X, \mathbb{Q})$, so that we may write 
    \begin{equation}\label{description of phi kappa}
        \phi_{\kappa} = B_{-\lambda_{Y}/r}\, [\Phi_{\mathcal{F}}]_*\, B_{-\lambda_{X}/r}.
    \end{equation}
\end{Rem}

\begin{Def}\label{LLV lattice}
    Let $X$ be an IHSM. The \textit{rational LLV lattice} is the graded lattice
    \begin{equation}
    \tilde{H}(X, \mathbb{Q}): = \mathbb{Q}\alpha\, \oplus H^2(X, \mathbb{Q})\, \oplus \mathbb{Q}\beta,
    \end{equation}
    equipped with an intersection pairing $(- , -)$ such that $\alpha$ and $\beta$ are isotropic and orthogonal to $H^2(X, \mathbb{Q})$ and $(\alpha, \beta) = -1$. Furthermore, $(- , -)$ restricts to $H^2(X, \mathbb{Q})$ as the BBF form. The grading is given by assigning $\alpha$ to have degree $-2$, $H^2(X, \mathbb{Q})$ to have degree $0$, and $\beta$ to have degree $2$. Equip $\tilde{H}(X, \bQ)$ with a weight $2$ Hodge structure given by $H^{1, 1}(\tilde{H}(X, \bC)) := \bC\alpha \oplus H^{1, 1}(X)\oplus \bC \beta$, $H^{2, 0}(\tilde{H}(X, \bC)) := H^{2, 0}(X)$, and $H^{0, 2}(\tilde{H}(X, \bC)): = H^{0, 2}(X)$.
\end{Def}

Applying results of Verbitsky and Taelman, Markman associates to $\phi_{\kappa}$ an isometry
\begin{equation}
\tilde{\phi}_{\kappa}: \tilde{H}(X, \mathbb{Q})\longrightarrow \tilde{H}(Y, \mathbb{Q}).
\end{equation}
We devote the next subsection to describing this association.

\subsection{The LLV algebra and isometries of the LLV lattice}

For a smooth complex projective variety $X$, the Hard Lefschetz theorem implies that any ample class $\lambda\in H^2(X, \mathbb{Q})\cap H^{1, 1}(X)$ determines a Lie algebra $\mathfrak{g}_{\lambda}\subset \mathrm{End}(H^*(X, \mathbb{Q}))$, which is isomorphic to $\mathfrak{sl}_2$. The Looijenga-Lunts-Verbitsky (LLV) Lie algebra is generated by all algebras $\mathfrak{g}_{\lambda}$ for which $\lambda\in H^2(X, \mathbb{Q})$ satisfies the hard Lefschetz property. We write $\mathfrak{g}(X)$ for the LLV Lie algebra of $X$. Let us from now on assume that $X$ is an IHSM. 
\begin{Prop}[{\cite{Verbitsky1}}, \cite{LL}]\label{Graded Lie algebra isomorphism}
    For $\lambda\in H^2(X, \bQ)$, let $e_{\lambda}\in \mathfrak{so}(\tilde{H}(X, \bQ))$ be the endomorphism given by $e_{\lambda}(\alpha) = \lambda$, $e_{\lambda}(\mu)= (\lambda, \mu)\beta$ for all $\mu\in H^2(X, \bQ)$, and $e_{\lambda}(\beta) = 0$. There is a unique isomorphism of graded Lie algebras
    $$
    \mathfrak{g}(X)\cong \mathfrak{so}(\tilde{H}(X, \mathbb{Q})),
    $$
    mapping $e_{\lambda}$ to $e_{\lambda}$ for each $\lambda \in H^2(X, \bQ)$.
\end{Prop}

\begin{Def}
    The \textit{Verbitsky component} $SH^*(X, \mathbb{Q})\subset H^*(X, \mathbb{Q})$ is the subalgebra generated by $H^2(X, \mathbb{Q})$. 
\end{Def}

\begin{Prop}[{\cite{Verbitsky1}}]\label{SH(X) irreducible}
    The Verbitsky component $SH^*(X, \mathbb{Q})$ is an irreducible $\mathfrak{g}(X)$-submodule of $H^*(X, \mathbb{Q})$ appearing with multiplicity one. 
\end{Prop}

\begin{Prop}[{\cite[Proposition 3.5]{Taelman}}]\label{Unique module isomorphism}
    There is a unique map of $\mathfrak{so}(\tilde{H}(X, \mathbb{Q}))$-modules 
    \begin{equation}
        \Psi_X: SH^*(X, \mathbb{Q})[2n]\longrightarrow \mathrm{Sym}^n\tilde{H}(X, \mathbb{Q}), 
    \end{equation}
    such that $\Psi_X(1) = \alpha^n/n!$.
\end{Prop}

Let 
\begin{equation}
    \Delta: \mathrm{Sym}^n\tilde{H}(X, \mathbb{Q})\longrightarrow \mathrm{Sym}^{n-2}\tilde{H}(X, \mathbb{Q})
\end{equation}
denote the Laplacian operator which contracts elements of $\mathrm{Sym}^n\tilde{H}(X, \mathbb{Q})$ against the intersection pairing $(-, -)$ on $\tilde{H}(X, \bQ)$ as follows: 
$$
\Delta: x_1 \cdots x_n\mapsto \sum_{i<j} (x_i, x_j) x_1\cdots \widehat{x}_i\cdots \widehat{x}_j
\cdots x_n,$$
where $\widehat{x}_k$ denotes omission of the element. 

\begin{Lem}\label{exact sequence of modules}
    There is an exact sequence of $\mathfrak{g}(X)$-modules
    $$
    0\rightarrow SH^*(X, \mathbb{Q})[2n]\xrightarrow{\Psi_X} \mathrm{Sym}^n\tilde{H}(X, \mathbb{Q})\xrightarrow{\Delta} \mathrm{Sym}^{n-2}\tilde{H}(X, \mathbb{Q})\rightarrow 0
    $$
\end{Lem} 
\begin{proof}
    This follows from Proposition \ref{Unique module isomorphism}, the fact that $\Delta$ is a surjective $\mathfrak{g}(X)$-module morphism, and the fact that $SH^*(X, \mathbb{Q})$ is an irreducible $\mathfrak{g}(X)$-module. 
\end{proof}

\begin{Def}
    Let 
    \begin{equation}
        S_{[n]}\tilde{H}(X, \mathbb{Q}): = \ker\left[\mathrm{Sym}^n\tilde{H}(X, \mathbb{Q})\xrightarrow{\Delta} \mathrm{Sym}^{n-2}\tilde{H}(X, \mathbb{Q})\right].
    \end{equation}
\end{Def}

\begin{Rem}
    Note that elements of $S_{[n]}\tilde{H}(X, \mathbb{Q})$ are spanned by $nth$-powers of isotropic elements of $\tilde{H}(X, \mathbb{Q})$. Note also that as a result of Lemma \ref{exact sequence of modules}, the image of $\Psi_X$ is precisely $S_{[n]}\tilde{H}(X, \mathbb{Q})$, and we shall therefore refer to $\Psi_X$ as the map 
    $$
    \Psi_X: SH^*(X, \mathbb{Q})[2n]\longrightarrow S_{[n]}\tilde{H}(X, \mathbb{Q}).
    $$
\end{Rem}

\begin{Th}[{\cite[Theorem A]{Taelman}}]\label{Taelman A}
    Let $\Phi_{\mathcal{E}}: D^b(X)\rightarrow D^b(Y)$ be a derived equivalence of deformation equivalent IHSMs and let $\phi_{v(\mathcal{E})}$ denote the rational Hodge isometry induced by convolution with the Mukai vector $v(\mathcal{E})$. Then there exists a canonical isomorphism of Lie algebras 
    $$
    \Phi^{\mathfrak{g}}: \mathfrak{g}(X)\longrightarrow \mathfrak{g}(Y)
    $$
    such that $\phi_{v(\mathcal{E})}$ is equivariant with respect to $\Phi^{\mathfrak{g}}$.
\end{Th}

We will express the analogous grading operator to $h_X$ on the LLV lattice as  
\begin{equation}\label{tilde h}
\tilde{h}_X: \tilde{H}(X, \mathbb{Q})\rightarrow \tilde{H}(X, \mathbb{Q})
\end{equation}
which acts on $\alpha_X$ by $-2$ and on $\beta_X$ by $2$, and acts trivially on $H^2(X, \mathbb{Q})$.

\begin{Def}\label{parallel transport operator definition}
    Let $X_1$ and $X_2$ be deformation-equivalent compact K\"ahler manifolds. An isomorphism $f: H^*(X_1, \bZ)\rightarrow H^*(X_2, \bZ)$ is called a \textit{parallel transport operator} if it arises as follows: There exist a smooth and proper family of compact K\"ahler manifolds $\pi: \cX\rightarrow B$, points $b_1, b_2\in B$, isomorphisms $g_i: X_i\rightarrow \cX_{b_i}$ with $\cX_{b_i}$ the fibers of $\pi$ over $b_i$, and a continuous path $c$ from $b_1$ to $b_2$, such that the parallel transport along the path $c$ in the local system $R\pi_*\bZ$ induces $g_{2*}\circ f\circ g_{1}^*: H^*(\cX_{b_1}, \bZ)\rightarrow H^*(\cX_{b_2}, \bZ)$.
\end{Def}

\begin{Def}\label{Definition of groupoid G}
Recall that a \textit{groupoid} is a category where every morphism is an isomorphism. Let $\mathcal{G}$ be the groupoid whose objects are pairs $(X, \epsilon)$ of an IHSM $X$ and an orientation $\epsilon$ of $H^2(X, \mathbb{Q})$, and whose morphisms are rational isometries with respect to the Mukai pairings on $X$ and $Y$ that arise as compositions of morphisms of the following three types:
\begin{enumerate}
        \item Parallel-transport operators $f: H^*(X, \bQ)\rightarrow H^*(Y, \bQ)$, which are isometries;
        \item Hodge isometries of the form
        \[[\Phi_{\cP}]_*: = 
        \left[\ch(\cP)\sqrt{\td_{X\times Y}}\right]_*: H^*(X, \bQ)\rightarrow H^*(Y, \bQ),
        \]
        where $\cP\in D^b(X\times Y)$ is the Fourier-Mukai kernel of a derived equivalence.

        \item An isometry of $H^*(X, \bQ)$ induced by the cup product with $\exp(\lambda)$, for $\lambda\in H^2(X, \bQ)$.
\end{enumerate}
\end{Def}
Such a rational Hodge isometry $\phi\in \Hom_{\cG}(X, Y)$ induces an adjoint isomorphism of Lie algebras 
\begin{equation}\label{adjoint Lie algebra isomorphism}
\mathrm{Ad}_{\phi}: \mathrm{End}(H^*(X, \mathbb{Q}))\longrightarrow \mathrm{End}(H^*(Y, \mathbb{Q})),
\end{equation}
given by $\mathrm{Ad}_{\phi}(\zeta) = \phi\zeta\phi^{-1}$.

Let $\tilde{\mathcal{G}}$ be the groupoid whose objects are IHSMs and whose morphisms are rational Hodge isometries of the LLV lattice. By the results of \cite[Section 4]{Taelman}, there exists a functor \[\tilde{H}: \mathcal{G}\longrightarrow \tilde{\mathcal{G}},\]
which is characterized by the following proposition: 
\begin{Prop}[{\cite[Proposition 4.4, Theorem 4.10, and Theorem 4.11]{Taelman}}]
    Let $V_1$ and $V_2$ be $d$-dimensional $\mathbb{Q}$-vector spaces with nondegenerate bilinear pairings and $\phi: S_{[n]}V_1\rightarrow S_{[n]}V_2$ an isometry such that $\phi(S_{[n]}\mathfrak{so}(V_1))\phi^{-1} = S_{[n]}\mathfrak{so}(V_2)$. Assume that $d$ is odd for $n$ even and choose orientations for $V_1$ and $V_2$. Then there exists a unique isometry $\tilde{H}(\phi):V_1\rightarrow V_2$ such that the restriction $S_{[n]}(\tilde{H}(\phi))$ of its $n$th-symmetric power to $S_{[n]}V_1$ satisfies $S_{[n]}(\tilde{H}(\phi)) = \phi$, if $n$ is odd, and $\nu S_{[n]}(\tilde{H}(\phi)) = \phi$ if $n$ is even, where $\nu = 1$ if $\tilde{H}(\phi)$ is orientation preserving, and $-1$ otherwise.  
\end{Prop}

\begin{Def}\label{definition of groupoid G_an}
    Let $\cG_{an}$ be the subgroupoid of $\cG$ with the same objects. Let the morphisms $\Hom_{\cG_{an}}(X, Y)$ in $\cG_{an}$ be compositions of morphisms of the following two types:
    \begin{enumerate}
        \item Parallel-transport operators $f: H^*(X, \bQ)\rightarrow H^*(Y, \bQ)$, which are Hodge-isometries (see Definition \ref{parallel transport operator definition});
        \item Hodge isometries of the form
        \[
        \left[\kappa(\cP)\sqrt{\td_{X\times Y}}\right]_*: H^*(X, \bQ)\rightarrow H^*(Y, \bQ),
        \]
        such that $\cP$ is a locally free twisted sheaf over $X\times Y$ and such that the Azumaya algebra $\cE nd(\cP)$ is a deformation of the endomorphism algebra of a locally-free Fourier-Mukai kernel of an equivalence of derived categories (see Definition \ref{deformation of Azumaya algebras}). 
    \end{enumerate}
\end{Def}

The functor $\tilde{H}$ sends a morphism $\phi\in \mathrm{Hom}((X, \epsilon_X), (Y, \epsilon_Y))$ to $\tilde{H}(\Psi_Y\circ \phi_{SH} \circ \Psi_X^{-1}): \tilde{H}(X, \mathbb{Q})\rightarrow \tilde{H}(Y, \mathbb{Q})$. Let $\nu_i:= \det(f_i)$. Then $\tilde{H}(S_{[n]}(f_1)\circ \phi \circ S_{[n]}(f_2))$ is given by 
    \[\begin{cases}
    \hspace{3mm}f_1\circ \tilde{H}(\phi)\circ f_2, & \hspace{3mm} n \text{ odd}\\
   \hspace{3mm} \nu_1\nu_2 f_1\circ \tilde{H}(\phi)\circ f_2, & \hspace{3mm} n \text{ even},\\
    \end{cases}\]
    (see \cite[Remark 3.2]{Markman22}).

For a degree-reversing Hodge isometry $\phi$, we denote by 
\begin{equation}\label{H tilde map}
    \tilde{H}(\phi): \tilde{H}(X, \mathbb{Q})\longrightarrow \tilde{H}(Y, \mathbb{Q})
\end{equation}
the resulting map on the LLV lattice. Denote by 
\begin{equation}\label{restriction of H tilde map}
\tilde{H}_0(\phi): H^2(X, \mathbb{Q})\longrightarrow H^2(Y, \mathbb{Q})
\end{equation}
the restriction of (\ref{H tilde map}) to $H^2(X, \mathbb{Q})$. By \cite[Theorem 4.8]{Taelman}, $\tilde{H}(\phi)$ exists for odd $n$, and by \cite[Theorem 4.9]{Taelman} it exists for $n$ even and $b_2$ odd (which holds in the Kummer case we consider here).

\begin{Def}\label{Hhat}
For an IHSM $X$ let $\mathcal C'_X:=\{x\in H^2(X,\mathbb R):(x,x)>0\}$ be the positive cone. It deformation retracts onto the unit sphere of any positive three-dimensional subspace, so $H^2(\mathcal C'_X,\mathbb Z)\cong\mathbb Z$. We orient $\mathcal C'_X$ by the generator determined by an ordered basis $(\omega,\operatorname{Re}\sigma,\operatorname{Im}\sigma)$, where $\omega$ is a K\"ahler class and $\sigma$ spans $H^{2,0}(X)$. For an isometry $g\colon H^2(X,\mathbb Q)\to H^2(Y,\mathbb Q)$ set $\nu(g):=1$, if $g$ maps the orientation of $\mathcal C'_X$ to that of $\mathcal C'_Y$, and $\nu(g):=-1$ otherwise. For $\phi\in \Hom_{\cG_{an}}(X, Y)$, define
\[
\widehat H_0(\phi):=\nu\big(\tilde H_0(\phi)\big)\,\tilde H_0(\phi)\colon H^2(X,\bQ)\longrightarrow H^2(Y,\bQ),
\]
following \cite[(5.8)]{Markman22}. Since $\nu(-g)=-\nu(g)$, $\widehat H_0(\phi)$ is orientation preserving.
\end{Def}

\begin{Lem}\label{orientation}
Let $X$ and $Y$ be IHSMs, and let $\mathcal C^+_X\subset H^{1,1}(X,\mathbb R)$ be the connected component of $\{\lambda\in H^{1,1}(X,\mathbb R):(\lambda,\lambda)>0\}$ containing the K\"ahler cone; define $\mathcal C^+_Y$ similarly. If $g\colon H^2(X,\mathbb Q)\to H^2(Y,\mathbb Q)$ is a Hodge isometry, then $g(\mathcal C^+_X)=\nu(g)\,\mathcal C^+_Y$. In particular, an orientation-preserving Hodge isometry maps $\mathcal C^+_X$ onto $\mathcal C^+_Y$.
\end{Lem}

\begin{proof}
Since $g$ is a real Hodge isometry, it maps $\{\lambda\in H^{1,1}(X,\mathbb R):(\lambda,\lambda)>0\}=\mathcal C^+_X\sqcup(-\mathcal C^+_X)$ onto the corresponding set for $Y$, so $g(\mathcal C^+_X)=\varepsilon\,\mathcal C^+_Y$ for some $\varepsilon\in\{\pm1\}$. Let $\omega\in\mathcal C^+_X$ and let $\sigma_X$ span $H^{2,0}(X)$. The basis $(\omega,\operatorname{Re}\sigma_X,\operatorname{Im}\sigma_X)$ is positively oriented, since the set of such bases with $\omega\in\mathcal C^+_X$ is connected and contains $\omega$ K\"ahler classes. Since $g$ is a Hodge isometry, $g(\sigma_X)=c\,\sigma_Y$ for some $c\in\mathbb C^*$, and multiplication by $c$ is an orientation-preserving automorphism of the real plane $\langle\operatorname{Re}\sigma_Y,\operatorname{Im}\sigma_Y\rangle$. Moreover $\varepsilon g(\omega)\in\mathcal C^+_Y$, so $(\varepsilon g(\omega),\operatorname{Re}\sigma_Y,\operatorname{Im}\sigma_Y)$ is positively oriented. Hence, the image under $g$ of the positively oriented basis $(\omega,\operatorname{Re}\sigma_X,\operatorname{Im}\sigma_X)$ has orientation $\varepsilon$, i.e.\ $\nu(g)=\varepsilon$.
\end{proof}

\begin{Lem}
Let $\phi\in\operatorname{Hom}_{\mathcal G}(X,Y)$. Then $\phi$ restricts to an isometry
$\phi_{SH}\colon SH^*(X,\mathbb Q)\xrightarrow{\sim}SH^*(Y,\mathbb Q)$ with respect to the Mukai pairings, and
$\phi_{SH}\,\mathfrak g(X)\,\phi_{SH}^{-1}=\mathfrak g(Y)$. In particular, this holds for every $\phi\in\operatorname{Hom}_{\mathcal G_{an}}(X,Y)$.
\end{Lem}

\begin{proof}
For morphisms of type (2) in Definition \ref{Definition of groupoid G}, this is \cite[Theorem 4.6]{Taelman}. Morphisms of type (1) are ring isomorphisms mapping $H^2(X,\mathbb Q)$ onto $H^2(Y,\mathbb Q)$, and morphisms of type (3) are of the form $\exp(L_\lambda)$ with $L_\lambda\in\mathfrak g(X)$; in both cases the claim is immediate (cf.\ \cite[proof of Lemma 3.4]{Markman22}). The statement for $\mathcal G_{an}$ follows, since every morphism of $\mathcal G_{an}$ is a composition of morphisms of $\mathcal G$ by \cite[p. 1264, Footnote 3]{Markman22}.
\end{proof}

\begin{comment}
\begin{proof}
    By Theorem \ref{Taelman A}, the adjoint isomorphism (\ref{adjoint Lie algebra isomorphism}) restricts to an isomorphism $\mathfrak{g}(X)\cong \mathfrak{g}(Y)$, if $\phi$ is an isometry associated to a derived equivalence. If $\phi$ is a parallel transport operator, the same holds by \cite[p. 1264, Footnote 3]{Markman22}. The fact that $SH^*(X, \mathbb{Q})$ and $SH^*(Y, \mathbb{Q})$ are irreducible $\mathfrak{g}(X)$- and $\mathfrak{g}(Y)$-representations respectively (Proposition \ref{SH(X) irreducible}), and the fact that the Mukai pairings on $H^*(X, \mathbb{Q})$ and $H^*(Y, \mathbb{Q})$ restrict to pairings on $SH^*(X, \mathbb{Q})$ and $SH^*(Y, \mathbb{Q})$, implies the result. 
\end{proof}
\end{comment}

\begin{Prop}[{\cite[Lemma 4.1(1), (2)]{Markman22}}]\label{degree reversing hodge isometries}
Let $\phi_{\kappa}$ be as in (\ref{phi kappa}) and let $\tilde{\phi}_{\kappa}$ be the image of $\phi_{\kappa}$ under the functor $\tilde{H}$ of (\ref{H tilde map}). Then:
    \begin{enumerate}
        \item The isometry $\tilde{\phi}_{\kappa}$ is degree-reversing.
        \item The isometry $\phi_{\kappa}$ is degree-reversing. 
    \end{enumerate}
\end{Prop}
\begin{proof}
    It is convenient that we first prove part $(1)$ and then demonstrate that $(1)\implies (2)$.

    Recall that $\tilde{\phi}_{\kappa}$ is an isometry of Hodge structures
    $$
    \tilde{\phi}_{\kappa}: \tilde{H}(X, \mathbb{Q})\longrightarrow \tilde{H}(Y, \mathbb{Q}),
    $$
    with respect to the Mukai pairings on $\tilde{H}(X, \mathbb{Q})$ and $\tilde{H}(Y, \mathbb{Q})$. We show first that $\mathrm{Span}(\tilde{\phi}_{\kappa}(\alpha_X)) = \mathrm{Span}(\beta_Y)$ and $\mathrm{Span}(\tilde{\phi}(\beta_X)) = \mathrm{Span}(\alpha_Y)$.
    
    Let $r>0$ be the rank of the Fourier-Mukai kernel $\mathcal{F}$, let $\pi_i$ be the projection maps from $X\times Y$, and let $c_1(\mathcal{F}) = \pi_1^*(\lambda_X)+ \pi_2^*(\lambda_Y)$ for $\lambda_X\in H^2(X, \mathbb{Q})$ and $\lambda_Y\in H^2(Y, \mathbb{Q})$. By (\ref{description of phi kappa}), $\tilde{\phi}_{\kappa}$ is of the form 
    \[
    \tilde{\phi}_{\kappa} = B_{-\lambda_Y/r}\tilde{H}([\Phi_{\mathcal{F}}]_*) B_{-\lambda_X/r}.
    \]
    Now, $B_{-\lambda_X/r} = \mathrm{exp}(e_{-\lambda_X/r})$ and hence $B_{-\lambda_X/r}(\beta_X) = \beta_X$. The equality
    \[
    \tilde{H}([\Phi_{\mathcal{F}}]_*)(\beta_X) = k(r\alpha_Y+ \lambda_Y + \frac{(\lambda_Y, \lambda_Y)}{2r}\beta_{Y})
    \]
    for some $k\in \mathbb{Q}$ follows from \cite[Lemma 4.13(iv)]{Beckmann}. We then note that $$B_{-\lambda_Y/r}(r\alpha_Y+ \lambda_Y + \frac{(\lambda_Y, \lambda_Y)}{2r}\beta_{Y}) = r\alpha_Y,$$ and observe 
    $$
    \tilde{\phi}_{\kappa}^{-1} = B_{\lambda_X/r} \tilde{H}([\Phi_{\mathcal{F}^{\vee}[2n]}]) B_{\lambda_Y/r},
    $$
which implies that $\tilde{\phi}_{\kappa}^{-1}$ maps $\mathrm{Span}(\beta_Y)$ to $\mathrm{Span}(\alpha_X)$ and part $(1)$ follows as a result.

    Let us now show that $(1)\implies (2)$. By (1), $\tilde\phi_\kappa(\alpha_X)\in\mathbb{Q}\beta_Y$ and $\tilde\phi_\kappa(\beta_X)\in\mathbb{Q}\alpha_Y$. Since $\tilde\phi_\kappa$ is an isometry, it maps $H^2(X,\mathbb{Q})=\langle\alpha_X,\beta_X\rangle^\perp$ onto $\langle\alpha_Y,\beta_Y\rangle^\perp=H^2(Y,\mathbb{Q})$. As $\tilde h_X$ acts by $-2$, $0$, $2$ on $\alpha_X$, $H^2(X,\mathbb{Q})$, $\beta_X$, and similarly for $\tilde h_Y$, evaluating on $\beta_Y$, $H^2(Y,\mathbb{Q})$ and $\alpha_Y$ gives
\[
\tilde\phi_\kappa\,\tilde h_X\,\tilde\phi_\kappa^{-1}=-\tilde h_Y .
\]
By \cite[Lemma 3.4]{Markman22}, it follows that $\phi_\kappa\,h_X\,\phi_\kappa^{-1}=-h_Y$. Let $x\in H^k(X,\mathbb{Q})$. Since $h_X$ acts on $H^k(X,\mathbb{Q})$ by $k-2n$,
\[
h_Y\,\phi_\kappa(x)=-\phi_\kappa(h_Xx)=\big((4n-k)-2n\big)\,\phi_\kappa(x),
\]
so $\phi_\kappa(x)$ lies in the $\big((4n-k)-2n\big)$-eigenspace of $h_Y$, which is $H^{4n-k}(Y,\mathbb{Q})$. Hence $\phi_\kappa$ is degree-reversing.

%     \[\begin{tikzcd}
% 	{\mathfrak{so}(\tilde{H}(X, \mathbb{Q}))} && {\mathfrak{so}(\tilde{H}(Y, \mathbb{Q}))} \\
% 	\\
% 	{\mathfrak{g}_X} && {\mathfrak{g}_Y}
% 	\arrow["{\mathrm{Ad}(\tilde{\phi}_{\kappa})}", from=1-1, to=1-3]
% 	\arrow["{\Psi_X}", from=3-1, to=1-1]
% 	\arrow["{\mathrm{Ad}(\phi_{\kappa})}", from=3-1, to=3-3]
% 	\arrow["{\Psi_Y}", from=3-3, to=1-3]
% \end{tikzcd}\]
% See \cite{Taelman}.

\end{proof}

The existence of the degree-reversing Hodge isometries associated to a derived equivalence implies that given an algebraic operator $L_{\omega}\in \mathfrak{g}(X)$ with the hard Lefschetz property, the adjoint Lefschetz operator $\Lambda_{\tilde{\phi}(\omega)}\in \mathfrak{g}(Y)$ is algebraic. In particular, this yields the following: 

\begin{Cor}\label{Derived equivalent LSC}
 Let $X$ and $Y$ be projective IHSMs and let $\phi\in \Hom_{\cG_{an}}(X, Y)$ be degree-reversing, where $\cG_{an}$ is as in Definition \ref{definition of groupoid G_an}. Then the Lefschetz standard conjecture holds for both $X$ and $Y$.
\end{Cor}

\begin{proof}
    Define $\tilde{\phi}: = \tilde{H}_0(\phi)$. Let $\lambda$ be ample on $X$, so that $(\lambda, \lambda)>0$. Let $t\in \bQ$ be defined by $\tilde{\phi}(\beta_X) = t\alpha_Y$. Note that $t\neq 0$, since $\tilde{\phi}$ is degree-reversing. By \cite[Lemma 1.6]{Markman22}, $\Lambda_{\lambda} = \frac{2}{t(\lambda, \lambda)}\phi^{-1}L_{\tilde{\phi}(\lambda)}\phi$. Here, $\tilde{H}_0(\phi)$ is a Hodge isometry by \cite[Lemma 3.4, Lemma 4.1(2)]{Markman22}, and $Y$ is projective, so $\tilde{\phi}(\lambda)\in H^{1, 1}(Y, \bQ)$ is algebraic, and hence so is $L_{\tilde{\phi}(\lambda)}$. The correspondences in Definition \ref{definition of groupoid G_an} are analytic: type (1) by \cite[Theorem 5.2, Lemma 5.3]{Markman22}, and type (2) since the $\kappa$-class is the characteristic class of of an untwisted analytic sheaf. Therefore, the correspondences $\phi^{\pm 1}$ are analytic, and hence algebraic by GAGA. So, $\Lambda_{\lambda}$ is algebraic and the LSC statement Conjecture \ref{LSC} holds for $X$. The same argument with $\phi^{-1}$ supplies Conjecture \ref{LSC} for $Y$. Conjecture \ref{BX} holds for both $X$ and $Y$ by Lemma \ref{lemma on equality of conjectures 1.1 and 5.1}.
\end{proof}

\begin{comment}
\begin{proof}
    Recall that over the complex numbers it suffices to show that for fixed classes $f_X\in H^{1, 1}(X, \mathbb{Q})$ and $f_{Y}\in H^{1, 1}(Y, \mathbb{Q})$ the adjoint Lefschetz operators $\Lambda_{f_X}$ and $\Lambda_{f_Y}$ are algebraic. Fix an ample class $f_X\in H^{1, 1}(X, \mathbb{Q})$ and consider $L_{\tilde{\phi}_{\kappa}(f_X)}$. Since, by Proposition \ref{degree reversing hodge isometries}, $\phi_{\kappa(\mathcal{F})}$ is a degree-reversing Hodge isometry whose adjoint restricts to an isomorphism $\mathfrak{g}(X)\cong \mathfrak{g}(Y)$, and furthermore the action of the LLV algebra preserves the rational LLV lattice, then, up to a scalar multiple, we have $\tilde{\phi}^{-1}_{\kappa}L_{\tilde{\phi}_{\kappa}(f_X)}\tilde{\phi}_{\kappa} = \Lambda_{f_X}$. By Proposition \ref{degree reversing hodge isometries}, this is equivalent to $\phi^{-1}_{\kappa}L_{\tilde{\phi}_{\kappa}(f_X)}\phi_{\kappa} = \Lambda_{f_X}$ up to scalar multiple (see the calculation in \cite[Lemma 1.6]{Markman22}). Therefore, the class $L_{f_X}$ being algebraic implies that $\Lambda_{f_X}$ is as well. The symmetry of the statement follows by replacing $\phi_{\kappa(\mathcal{F})}$ with $\phi_{\kappa(\mathcal{F}^*[2n])}$.
\end{proof}
\end{comment}
We can then immediately deduce the following:

\begin{Cor}
   Let $n\geq 2$ be even and $A$ an abelian surface with a polarization of exponent coprime to $n+1$. Then Conjecture \ref{BX} holds for both $\Kum_n(A)$ and $\Kum_n(\widehat{A})$.
\end{Cor}

\begin{proof}
$\cF^K$ is locally free by Lemma \ref{local freeness}, and therefore the isometry $\phi_{\kappa}$ associated to it is of type (2) in Definition \ref{definition of groupoid G_an}. It is degree reversing by Proposition \ref{degree reversing hodge isometries}. The statement is then a combination of Corollary \ref{Derived equivalent LSC} and \cite[Theorems 6.18, 7.27]{Magni}.    
\end{proof}

\subsection{The Hodge locus of the $\kappa$-class}\label{section 5 characteristic classes}

In this section, we conclude that the parallel transport of the $\kappa$-class of the kernel of a derived equivalence is of Hodge-type at the points at which the induced isometry on second cohomology is a Hodge isometry. 

% Let $X$ and $Y$ be deformation equivalent IHSMs of dimensions $2n$, and let $\Phi_\cF: D^b(X)\rightarrow D^b(Y)$ be a derived equivalence such that $\rk(\cF)>0$. 

\begin{comment}
\begin{Def}
    Let $\mathfrak{g}_0(X): = \{\xi\in \mathfrak{g}(X): [h, \xi] = 0\}$. Define its derived subalgebra as 
    \[
    \overline{\mathfrak{g}}_X: = [\mathfrak{g}_0(X), \mathfrak{g}_0(X)].
    \]
\end{Def}

The centralizer of $h$ in $\mathfrak{g}(X)$ decomposes as $\overline{\mathfrak{g}}_X\oplus \bQ h$. The algebra $\overline{\mathfrak{g}}_X$ acts on $H^*(X, \bQ)$ and its restriction to $H^2$ yields an isomorphism
\[
\overline{\mathfrak{g}}_X\xrightarrow{\sim} \mathfrak{so}(H^2(X, \bQ)), \hspace{5mm} \xi\mapsto \xi|_{H^2},
\]
by \cite[Proposition 4.5(ii)]{LL}. Under the isomorphism of Proposition \ref{Graded Lie algebra isomorphism}, $\overline{\mathfrak{g}}_X$
\end{comment}

Let $\cF^K$ be the kernel of the derived equivalence $\Phi_{\cF^K}: D^b(K)\rightarrow D^b(\widehat{K})$ of Theorem \ref{Magni derived equivalence}.
Define
\begin{equation}
    \gamma: = \kappa(\cF^K)\sqrt{\td_{K\times \widehat{K}}}\in H^*(K\times \widehat{K}, \bQ)
\end{equation}
so that $\gamma_* = \phi_\kappa$.  Let $U_X:=\langle\alpha_X,\beta_X\rangle\subset\tilde H(X,\mathbb Q)$ denote the hyperbolic plane. As a result of Proposition~\ref{degree reversing hodge isometries}(1), $\tilde\phi_\kappa(U_K)=U_{\hat K}$, so $\tilde\phi_\kappa$ restricts to an isometry $H^2(K,\bQ)\to H^2(\widehat{K},\bQ)$. We set
\[
\psi:=\widehat H_0(\phi_\kappa): H^2(K,\bQ)\longrightarrow H^2(\widehat{K},\bQ)
\]
(Definition \ref{Hhat}). This is a Hodge isometry \cite[Lemma 4.1(2)]{Markman22}, and is orientation preserving. The sign normalization makes $\psi$ orientation preserving: The maps $\pm\psi$ each induce the same map on period domains. Since parallel-transport operators are orientation preserving, each $\psi_b$ defined below is an orientation preserving Hodge isometry as well.

We cannot assume that the base point $(K, \eta, \widehat{K}, \widehat{\eta})$ of our deformation lies in a connected component of the moduli space of rational Hodge isometries (see Definition \ref{Mpsi}), hence we construct the base $B$ of the deformation and show that the results of \cite{Markman22} apply to it.

Let $\cK\rightarrow \mathrm{Def}(K)$ and $\cK\rightarrow \mathrm{Def}(\widehat{K})$ be the Kuranishi families of $K$ and $\widehat{K}$. Their bases are smooth of dimension $b_2-2 = 5$, by \cite[Theorem 5]{Beauville}. Let $b_0 = 0\in \mathrm{Def}(K)$.  Let $T_{K, b}$ be the parallel transport from the central fiber to the point $b\in \mathrm{Def}(K)$. Define the parallel transport in $\mathrm{Def}(\widehat{K})$ similarly. Let $\Omega_K$ and $\Omega_{\widehat{K}}$ be period domains for $K$ and $\widehat{K}$ respectively. By the local Torelli theorem (\cite[Theorem 5]{Beauville}), after possibly shrinking the local deformation space $\mathrm{Def}(K)$, the period map 
\[
\mathrm{Per}_K: \mathrm{Def}(K)\rightarrow \Omega_K, \hspace{5mm} b\mapsto T^{-1}_{K, b}(H^{2,0}(\cK_b)), 
\]
is an open embedding. The analogous statement holds for $\mathrm{Per}_{\widehat{K}}$. The map $\psi$ is a Hodge isometry and hence it maps $\Omega_K$ to $\Omega_{\widehat{K}}$ and $\mathrm{Per}_{K}(0)$ to $\mathrm{Per}_{\widehat{K}}(0)$. Now let $B\subseteq \mathrm{Def}(K)$ be an open ball centered at $\mathrm{Per}_K(0)$ such that $\psi(\mathrm{Per}_K(B))\subset \mathrm{Per}_{\widehat{K}}(\mathrm{Def}(\widehat{K}))$. Define 
\[
\beta:= \mathrm{Per}^{-1}_{\widehat{K}}\circ \psi \circ \mathrm{Per}_{K}: B\rightarrow \mathrm{Def}(\widehat{K}) 
\]
and let 
\[
\cX: = \cK|_B\rightarrow B, \hspace{5mm} \cY: = \beta^*\widehat{K}\rightarrow B.
\]
The fibers $X_b\times Y_b$ of $\cX\times_B \cY\rightarrow B$ are products of Kummer-type IHSMs such that $X_{b_0}\times Y_{b_0} = K\times \widehat{K}$. For a point $b\in B$, we have parallel transport operators $T_{X, b}: = T_{K, b}$ and $T_{Y, b}: = T_{\widehat{K}, \beta(b)}$. Define the parallel transport of $\psi$ as 
\[
\psi_b: = T_{Y, b}\circ \psi \circ T_{X, b}^{-1}: H^2(X_b, \bQ)\rightarrow H^2(Y_b, \bQ)
\]
and the parallel transport of $\gamma$ as
\[
\gamma_b:= (T_{X, b}\otimes T_{Y, b})(\gamma)\in H^*(X_b\times Y_b, \bQ).
\]

\begin{Prop}\label{hodge isometry psi of hodge type} The following hold for all $b\in B$.

\begin{enumerate}
    \item $\psi_b$ is a Hodge isometry.
    \item $\gamma_b$ is a rational Hodge class on $X_b\times Y_b$. 
\end{enumerate}

\end{Prop}
\begin{proof}
(1): Since parallel transport preserves the cup product, it preserves the BBF form and therefore $\psi_b$ is an isometry. By the construction of the map $\beta$, $\mathrm{Per}_{\widehat{K}}(\beta(b)) = \psi(\mathrm{Per}_K(b)) $, i.e.
\[
T^{-1}_{Y, b}(H^{2,0}(Y_b)) = \psi(T^{-1}_{X, b}(H^{2, 0}(X_b))).
\]
Therefore, $\psi_b(H^{2,0}(X_b)) = H^{2,0}(Y_b)$.

(2): The statement is proven in \cite[Proof of Proposition 5.22, Step 1]{Markman22}. Note that Step 1 of the proof does not use the assumption that $\psi$ maps a K\"ahler class to a K\"ahler class, nor does it use that $X_b\times Y_b$ is projective. Hence, the argument is the same for any connected family of products of IHSMs along which the parallel transport of $\psi$ is a Hodge isometry.  
\end{proof}
\begin{Rem}\label{Remark on kappa class equivalence}
    The statement of Proposition \ref{hodge isometry psi of hodge type} is equivalent to the following: The parallel transport of $\kappa(\cF^K)$ is of Hodge-type at $b$ if $\psi_b$ is a Hodge isometry. This is due to the fact that the invertible map $\sqrt{\td_{X_b\times Y_b}}$ is the parallel transport of $\sqrt{\td_{X\times Y}}$ and is of Hodge-type at $b$.
\end{Rem}

\begin{Cor}\label{existence twisted deformation corollary}
    After possibly shrinking $B$, there is a $\theta_R$-twisted locally free sheaf $\cE$ of rank $R$ on a $\mu_R$-gerbe over $\cX\times_B\cY$, flat over $B$ such that $\det(\cE)\cong \cO$ and $\cE_{K\times \widehat{K}}\cong p^*\cF^K\otimes M$. The bundle $p^*\cF^K\otimes M$ is the gerby normalization of $\cF^K$ as in Definition \ref{gerby normalization def}. For every $b\in B$, $\ch(\cE_b) = \kappa(\cE_b)$ is the horizontal deformation of $\kappa(\cF^K)$.
\end{Cor}
\begin{proof}
   $\cF^K$ is a vector bundle of rank $R$ (Lemma \ref{local freeness}), semiregular (Corollary \ref{K-trivial corollary}), and $\kappa(\cF^K)$ remains of Hodge-type along $B$ (Remark \ref{Remark on kappa class equivalence}). Therefore, we may apply Theorem \ref{analytic Buchweitz-Flenner conjecture} to the family $\cX\times_B\cY\rightarrow B$, which is a deformation of $K\times \widehat{K}$ over a smooth germ $(B, b_0)$, and to $E_0 = \cF^K$. Since $\cE$ is flat over $B$ and has trivial determinant, $\ch(\cE_b) = \kappa(\cE_b)$ and is the horizontal deformation of $\ch(\cE_{K\times \widehat{K}}) = \kappa(\cF^K)$.
     Note that $\det(\cE_0)\cong \cO$. For each fiber, $H^1(X_b\times Y_b, \cO) = 0$, therefore, $c_1(\det(\cE_b))$ is the parallel transport of $c_1(\cO)$ and hence trivial, so $\det(\cE_b)\cong \cO$. By base change, $\det(\cE)\cong p^*p_*\det(\cE)$, and after possibly shrinking $B$, $p_*\det(\cE)$ is trivial. 
\end{proof}

\begin{comment}
\begin{Cor}
  $\gamma_b$ is a rational Hodge class on $X_b\times Y_b$ for every $b\in B$. Equivalently, $\kappa(\cF^K)$ remains of Hodge-type along $B$ for the family $\cX\times_B\cY\rightarrow B$ connected by any continuous path. The same holds for $c_2(\cE nd(\cF_K))$. 
\end{Cor}
\begin{proof}
    By construction of the $\kappa$-class, $\kappa_2(\mathcal{E}nd(\mathcal{F}^K)) = 2\rk(\mathcal{F}^K)\kappa_2(\mathcal{F}^K)$. We have $\ch_2(\mathcal{E}nd(\mathcal{F}^K)) = -c_2(\mathcal{E}nd(\mathcal{F}^K)) = -\kappa_2(\mathcal{E}nd(\mathcal{F}^K))$ by the vanishing of $c_1(\mathcal{E}nd(\mathcal{F}^K))$, and so \[c_2(\mathcal{E}nd(\mathcal{F}^K)) = -2\rk(\mathcal{F}^K)\kappa_2(\mathcal{F}^K).\]
    
\end{proof}
\end{comment}

\section{Slope-stability of the kernel $\cF^K$ and its deformations}\label{section 6 slope-stability}

In this section, we prove that $\cF^K$ is slope-stable with respect to a big and nef class pulled back from $(W/G)\times (\widehat{W}/G)$ and with respect to all nearby ample classes, following Stapleton's proof of stability of tautological bundles on Hilbert schemes in \cite{Stapleton}. We show that at Picard trivial points, the deformation of the slope-stable kernel (the existence of which is established by Corollary \ref{existence twisted deformation corollary}) has no subsheaves of intermediate rank.

Recall the notation from Section \ref{section 4 semiregularity bundles}. $\cF^K$ is the image of $E^{\prime}$ under the BKR equivalence. It is locally free of rank $R$ by Lemma \ref{local freeness}. 

\begin{Lem}
    There is an ample line bundle $H_Q$ on $Q$ such that $\zeta^*H_Q$ is ample and $G^2$-invariant on $W\times \widehat{W}$.
\end{Lem}
\begin{proof}
    Let $L$ be an ample line bundle on $W\times \widehat{W}$. Then $L^{\prime}: = \bigotimes_{g\in G^2}g^*L$ is ample and carries a natural $G^2$-linearization. We note that the action of the stabilizers of the $G^2$-action on the fibers of $L^{\prime\otimes |G^2|}$ is trivial. By Kempf's descent lemma \cite[Theorem 2.3]{KempfLemma}, $L^{\prime \otimes |G^2|}$ descends to $Q$. The descended line bundle $H_Q$ is ample, since its pullback under the finite surjective map $\zeta$ is ample, by \cite[Corollary 1.2.28]{LazarsfeldI}. 
\end{proof}

Let $H: = \zeta^*H_Q$ and $H_K: = h^*H_Q$. Then $H_K$ is big and nef, since $h$ is birational. 
\begin{Def}
    Let $N$ be a nef class and $S$ a sheaf on a normal projective variety $X$. Then $N^{\dim(X)-1}\in N_1(X)_{\bR}$ is a curve class. Define the slope of $S$ with respect to $N$ as
    \[
    \mu_N(S): = c_1(S)\cdot N^{\dim(X)-1}/\rk(S). 
    \]
    $S$ is \textit{$\mu$-stable} (resp. \textit{$\mu$-semi-stable}) with respect to $N$ if for all torsion-free quotients $S\rightarrow T\rightarrow 0$ of intermediate rank, 
    \[
    \mu_N(S)<\mu_N(T) \hspace{5mm} (\mathrm{resp.}\, \mu_N(S)\leq \mu_N(T)).
    \]
    (See \cite[Definition 4.3]{Stapleton}.)
\end{Def}
\begin{Prop}\label{perturbed class stability}

    \begin{enumerate}
        \item $\cF^K$ is $\mu$-stable with respect to $H_K$. 
        \item For every ample class $\Theta$ on $K\times \widehat{K}$, there is $\epsilon_0>0$ such that $\cF^K$ is $\mu$-stable with respect to $H_K+ \epsilon \Theta$, for all $0< \epsilon < \epsilon_0$.
    \end{enumerate}
\end{Prop}
We prove Proposition \ref{perturbed class stability} at the end of the section.
\begin{Lem}\label{Properties of Etau}
    Let $E_{\tau}: = (1, \tau)^*\cF^G$. 
    \begin{enumerate}
        \item Each $E_{\tau}$ is $\mu$-stable with respect to every ample class.
        \item The $E_{\tau}$ are pairwise-nonisomorphic. 
        \item $G^2$ permutes the summands $E_{\tau}$ transitively via the linearization on $E^{\prime}$.
    \end{enumerate}
\end{Lem}
\begin{proof}
    (1): $1\times \tau$ is an automorphism of the abelian variety $W\times \widehat{W}$. Each $E_{\tau}$ is simple semihomogeneous and so we apply Proposition \ref{simple semihomogeneous bundles are mu stable}. (2): This is Lemma \ref{pairwise non-isomorphic}. (3): This follows from the definition of the inflated linearization \eqref{G^2-linearization}.
\end{proof}

\begin{Lem}\label{intermediate ranks lemma}
    Every non-trivial $G^2$-equivariant coherent subsheaf $S\subset E^{\prime}$ of rank $0<\rk(S)<R$ satisfies $\mu_H(S)< \mu_H(E^{\prime})$. 
\end{Lem}
\begin{proof}
Using Lemma \ref{Properties of Etau} and the $G^2$-invariance of $H$, the result follows from the argument of \cite[Lemma 1.3]{Stapleton}. 
\end{proof}

 Let $\pi$ be the map \eqref{definition of the map pi}. Following \cite[\S 1]{Stapleton}, we let $S$ be a torsion-free sheaf on $K\times \widehat{K}$ and, for $j: W^{\circ}\times \widehat{W}^{\circ}\hookrightarrow W\times \widehat{W}$ the inclusion, define
\[
S_{W\times \widehat{W}}: = j_*\pi^*(S|_{K^{\circ}\times \widehat{K}^{\circ}}).
\]
Then if $S\subset \cF^K$, it follows that $S_{W\times \widehat{W}}\subset j_*(E^{\prime}|_{W^{\circ}\times \widehat{W}^{\circ}})$. The right-hand side is isomorphic to $E^{\prime}$, by Hartogs' theorem. So, $S_{W\times \widehat{W}}$ is a $G^2$-equivariant subsheaf of $E^{\prime}$ with rank equal to that of $S$. It follows that $(\cF^K)_{W\times \widehat{W}} = E^{\prime}$.

\begin{Lem}\label{slope equality}
    For every torsion-free sheaf $S$ on $K\times \widehat{K}$
    \[
    |G^2|\,c_1(S)\cdot H_{K}^{4n-1} = c_1(S_{W\times \widehat{W}})\cdot H^{4n-1}.
    \]
\end{Lem}
\begin{proof}
    The statement is the analogue of \cite[Lemma 1.2]{Stapleton}. We carry out a similar computation to the one that is implicit there: Let $m\gg0$ such that $mH_Q$ is very ample and choose general divisors $D_1, \dots, D_{4n-1}\in |mH_Q|$ such that $C: = \cap_iD_i$ is a curve. Assume that $C$ has trivial intersection with the following two subsets:
    \begin{enumerate}
        \item $Q\backslash Q^{\circ}$ (recall $Q^{\circ}\subset Q$ has codimension $2$),
        \item the image of the locus on which $S$ is not locally free (a locus of codimension at least $2$).

    \end{enumerate}
    Then the inverse image $h^{-1}(C)\subset K^{\circ}\times \widehat{K}^{\circ}$ represents $m^{4n-1}H_K^{4n-1}$. Furthermore, $\zeta^{-1}(C)\subset W^{\circ}\times \widehat{W}^{\circ}$ represents $m^{4n-1}H^{4n-1}$ and $\zeta^{-1}(C)\rightarrow C$ is \'etale of degree $|G^2|$.

    On the locus $W^{\circ}\times \widehat{W}^{\circ}$, $S_{W\times \widehat{W}} = \pi^*S$. Therefore, 
    \begin{align*}
    c_1(S_{W\times \widehat{W}})\cdot H^{4n-1} = m^{-(4n-1)}\deg(\pi^*\det S|_{\zeta^{-1}C}) = \\|G^2|\,m^{-(4n-1)}\deg (\det S|_{h^{-1}(C)}) = |G^2|\,c_1(S)\cdot H^{4n-1}_K.
    \end{align*}
\end{proof}

\begin{proof}[Proof of Proposition \ref{perturbed class stability}(1)]
    The argument is similar to that of \cite[Theorem 1.4]{Stapleton}. Let $S\subset \cF^K$ be a sheaf of rank $k$, for $0< k < R$. By possibly replacing $S$ with $S^{\vee \vee}$, we may assume that $S$ is reflexive. Then $S_{W\times \widehat{W}}\subset E^{\prime}$ is $G^2$-equivariant of rank $k$. By Lemmas \ref{intermediate ranks lemma} and \ref{slope equality}, 
    \[
    \mu_{H_K}(S) = \frac{\mu_H(S_{W\times \widehat{W}})}{|G^2|}< \frac{\mu_H(E^{\prime})}{|G^2|} = \mu_{H_K}(\cF^K).
    \]\end{proof}
\begin{Lem}\label{isomorphism of curve classes}
    The map $(-)\cap H^{4n-2}_{K}: N^1(K\times \widehat{K})_{\bR}\rightarrow N_1(K\times \widehat{K})_{\bR}$ is an isomorphism.
\end{Lem}
\begin{proof}
 The Hilbert--Chow morphisms $K\to W/G$ and $\widehat{K}\to\widehat{W}/G$ are projective symplectic
resolutions, hence semismall. Hence, so is $h$. Since $Q$ is normal and $h$ is birational,
$h_*\cO_M=\cO_Q$. So, for $H_Q$ very ample, $H_K$ is lef in the sense of \cite[Definition ~2.1.3]{dCM02}.

By \cite[Theorem ~2.3.1]{dCM02}, $H_K$ satisfies hard Lefschetz and the Hodge--Riemann bilinear
relations. In particular $q(a,b):=a\cdot b\cdot H_K^{4n-2}$ is negative definite on
$\{a\in H^{1,1}(M,\bR): a\cdot H_K^{4n-1}=0\}$. Since $N_1(M)_\bR$ is dual to $N^1(M)_\bR$, there is a map $D\mapsto q(D,-)$, and it suffices that $q$ is nondegenerate on $N^1(M)_\bR$.
Write $D=cH_K+D_0$ with $D_0\cdot H_K^{4n-1}=0$, which is possible because $q(H_K,H_K)=H_K^{4n}>0$.
If $q(D,-)=0$ on $N^1(M)_\bR$, then $c=0$ and $q(D_0,D_0)=0$, so $c_1(D_0)=0$ in $H^2(M,\bR)$ and
$D=0$ in $N^1(M)_\bR$.
\end{proof}
\begin{proof}[Proof of Proposition \ref{perturbed class stability}(2)]
    We apply Lemma \ref{isomorphism of curve classes} and so the result follows immediately from \cite[Proposition 4.8]{Stapleton}.
\end{proof}

\begin{Prop}[{\cite[Theorem 5.1.1]{LubkeTeleman}}]\label{slope-stability open condition}
Let $M$ be a compact K\"ahler manifold, $T$ a parameter space, and $E\rightarrow M$ a possibly twisted complex vector bundle. Let $(J_t, \omega_t, \overline{\partial}_t)_{t\in T}$ be a family of triples consisting of integrable complex structures $J_t$, K\"ahler forms $\omega_t$ with respect to $J_t$, and holomorphic structures $\overline{\partial}_t$ of $E$ over $(M, J_t)$. Write $E_t : = (E, \overline{\partial}_t)$. If $E_0$ is $\mu_{\omega_0}$-stable, then $E_t$ is $\mu_{\omega_t}$-stable, for all $t$ near zero, and it carries a Hermitian-Einstein metric $h_t$ depending continuously on $t$.
    
\end{Prop}

\begin{Def}\label{twisted slope-stability definition}
Let $X$ be a compact K\"ahler manifold of dimension $d$ and $E$ a torsion-free $\theta$-twisted sheaf of rank $R$ on $X$. For a coherent twisted sheaf $S$ of rank $k$ set $\det S:=(\Lambda^kS)^{\vee\vee}$, a $\theta^k$-twisted line bundle. For a subsheaf $S\subset E$ of rank $0<k<R$ define
\[
L_S:=(\det S)^{\otimes R}\otimes(\det E)^{\otimes -k}.
\]
For a K\"ahler class $\omega$, we say $E$ is \emph{$\mu_\omega$-stable} (resp.\ \emph{semistable}) if $c_1(L_S)\cdot\omega^{d-1}<0$ (resp.\ $\le 0$) for every such $S$. If $\theta$ is trivial, then $c_1(L_S)\cdot\omega^{d-1}=kR\big(\mu_\omega(S)-\mu_\omega(E)\big)$. If $\det E\cong\mathcal O_X$, we write $\mu_\omega(S):=c_1\big((\det S)^{\otimes R}\big)\cdot\omega^{d-1}/(kR)$ and $\mu_\omega(E):=0$.
\end{Def}

\begin{Rem}
    The conclusion of Proposition \ref{slope-stability open condition} holds for $\theta$-twisted bundles \cite[Theorem 3.4.7]{WrayThesis}. 
\end{Rem}

\begin{Lem}\label{slope stable with respect to every kahler class}
Let $X$ be a compact K\"ahler manifold with $\mathrm{NS}(X)=0$, and let $E$ be a torsion-free $\theta$-twisted sheaf of rank $R$ on $X$. Then $c_1(L_S)=0$ in $H^2(X,\mathbb R)$ for every subsheaf $S\subset E$ of rank $0<k<R$; in particular $\mu_\omega(S)=\mu_\omega(E)$ for every K\"ahler class $\omega$. Consequently, $E$ is $\mu_\omega$-stable for some K\"ahler class $\omega$ if and only if $E$ has no subsheaf of rank $0<k<R$, in which case $E$ is $\mu_{\omega'}$-stable for every K\"ahler class $\omega'$.
\end{Lem}

\begin{proof}
Since $\det S$ is $\theta^k$-twisted and $\det E$ is $\theta^R$-twisted, $L_S$ is twisted by $\theta^{kR}\theta^{-kR}=1$. Equivalently, $\mu_R$ acts on $L_S$ trivially, so $L_S$ is an untwisted holomorphic line bundle on $X$. Hence $c_1(L_S)\in\mathrm{NS}(X)=0$, and $c_1(L_S)\cdot\omega^{d-1}=0$ for every K\"ahler class $\omega$, i.e.\ $\mu_\omega(S)=\mu_\omega(E)$. If $E$ has a subsheaf of rank $0<k<R$, the strict inequality does not hold for every $\omega$, so $E$ is not $\mu_\omega$-stable for any class $\omega$. Therefore, $E$ is $\mu_{\omega^{\prime}}$-stable for every $\omega^{\prime}$.
\end{proof}

Let $f: \cX\times_B \cY\rightarrow B$ be the smooth proper family over connected complex manifold $B\ni 0$, such that $\cX_0 = K\times \widehat{K}$, as introduced in Corollary \ref{existence twisted deformation corollary}. Let $\cE$ be a locally free $\theta$-twisted sheaf on $\cX\times_B\cY$ such that $\cE|_{K\times \widehat{K}}\cong p^*\cF^K\otimes M$. 

\begin{Cor}\label{m omega t stability}
    Let $\omega_0$ be a K\"ahler form on $K\times \widehat{K}$ such that $\cF^K$ is $\mu_{\omega_0}$-stable. Then, after possibly shrinking $B$, there are K\"ahler forms $\omega_t$ on $\cX_t\times \cY_t$ depending continuously on $t$ and extending $\omega_0$ such that, $\cE|_{\cX_t\times \cY_t}$ is $\mu_{\omega_t}$-stable for each $t\in B$.
\end{Cor}

\begin{Cor}\label{missing-link}
Let $\psi=\widehat H_0(\phi_\kappa)$ and $\psi_t$ be as in Definition \ref{Hhat}. There is $t\in B$ such that:
\begin{enumerate}
\item $\mathrm{NS}(X_t)=\mathrm{NS}(Y_t)=0$;
\item $\psi_t$ maps every K\"ahler class of $X_t$ to a K\"ahler class of $Y_t$;
\item $\cE_t$ is slope-stable with respect to every K\"ahler class on $X_t\times Y_t$.
\end{enumerate}
\end{Cor}

\begin{proof}
After possibly shrinking $B$ as in Corollary \ref{m omega t stability}, we may assume that $\cE_t$ is $\mu_{\omega_t}$-stable for every $t\in B$.

Recall that points $t\in B$ such that $\NS(X_t)=\NS(Y_t)=0$.
 Let $\omega\in\cC^+_{X_t}$. The isometry $\psi$ of Definition \ref{Hhat} is orientation preserving by the construction in \S \ref{section 5 characteristic classes}, and parallel-transport operators are orientation preserving \cite[\S1.1]{Markman11}, so $\psi_t$ is an orientation-preserving Hodge isometry. Lemma \ref{orientation} therefore gives $\psi_t(\omega)\in\cC^+_{Y_t}$. Since $H^{1,1}(Y_t,\bZ)=\NS(Y_t)=0$, the K\"ahler cone of $Y_t$ coincides with $\cC^+_{Y_t}$ \cite[Corollaries~5.7 and~7.2]{Huybrechts2}, see also \cite[Proposition~5.14(3)]{Markman11}. Hence $\psi_t(\omega)$ is a K\"ahler class on $Y_t$, which proves~(2).

For~(3), note that $H^1(\cO)$ vanishes on both factors, so the K\"unneth formula gives $H^2(X_t\times Y_t,\bZ)=H^2(X_t,\bZ)\oplus H^2(Y_t,\bZ)$, where the direct sum is compatible with Hodge structures. Consequently $\NS(X_t\times Y_t)=\NS(X_t)\oplus\NS(Y_t)=0$ by~(1). Since $\cE_t$ is $\mu_{\omega_t}$-stable, Lemma \ref{slope stable with respect to every kahler class} shows that $\cE_t$ is $\mu_{\omega'}$-stable for every K\"ahler class $\omega'$ on $X_t\times Y_t$.

Finally, by~(2) we have $\mathcal K_{X_t}\cap\psi_t^{-1}(\mathcal K_{Y_t})=\mathcal K_{X_t}$, which is a non-empty open cone. For every $\omega$ in this cone, $\pi_1^*\omega+\pi_2^*\psi_t(\omega)$ is a K\"ahler class on $X_t\times Y_t$, and $\mathcal E_t$ is stable with respect to it by (3). 
\end{proof}

\begin{comment}
 We get the following corollary immediately:

\begin{Cor}
Suppose $\NS(X_t\times Y_t)=0$ and $E_t$ is $\mu_{\omega}$-stable for one K\"ahler class $\omega$. Then $E_t$ has no subsheaf of rank $0<k<R$ and hence it is $\mu_{\omega^{\prime}}$-stable for every K\"ahler class $\omega^{\prime}$.
\end{Cor}
\begin{proof}
    Such a subsheaf $S$ would have $\mu_{\omega}(S) = \mu_{\omega}(E_t)$ by Lemma \ref{lem:irreducible}, which contradicts stability.
\end{proof}
\end{comment}

\section{Hyperholomorphic sheaves and moduli spaces of rational Hodge isometries}\label{deforming derived equivalences section}

Our next objective is to utilize the theory of Verbitsky and Markman to deform the Fourier-Mukai kernel over diagonal twistor lines in the moduli space of products of Hodge-isometric marked IHSMs. It is not immediately obvious what it means for a sheaf to be hyperholomorphic over a product of IHSMs. To that end, we deform over a connected component of the moduli space that parametrizes products of IHSMs whose second cohomology rings are rational Hodge-isometric, following Markman in \cite{Markman22}. 

\begin{Def}
    A \textit{marking} of an IHSM $Y$ is a fixed isometry 
    $$
    \eta: (H^2(Y, \bZ), \langle -, -\rangle )\longrightarrow \Lambda,
    $$
    where $\Lambda$ is a fixed lattice equipped with the BBF pairing.
    \end{Def}
    
    We denote by $\mathfrak{M}_{\Lambda}$ the moduli space of $\Lambda$-marked IHSMs. Its dimension is $\mathrm{rank}(\Lambda)-2$; see \cite{Huybrechts2}.  For $q_Y$ the BBF pairing (\ref{BBF form}), we denote by \begin{equation}
    \Omega_{\Lambda}: = \{x\in \mathbb{P}(\Lambda\otimes \mathbb{C}): q_Y(x, x)=0,\hspace{2mm} q_Y(x,\overline{x})>0\}
    \end{equation}
    the five-dimensional period domain of $\Lambda$-marked IHSMs. The period map
    \begin{equation}
        \mathrm{Per}: \mathfrak{M}_{\Lambda}\longrightarrow \Omega_{\Lambda},
    \end{equation}
    defined by $\mathrm{Per}((Y, \eta)) = \eta(H^{2, 0}(Y))$, is surjective and generically injective when restricted to every connected component of $\mathfrak{M}_{\Lambda}$, by \cite{Huybrechts1}.

\begin{Rem}
    Let $\mathfrak{M}^0_{\Lambda}$ denote a connected component of the moduli space $\mathfrak{M}_{\Lambda}$. There exists a universal family 
  \begin{equation}\label{universal family}
      f: \mathcal{Y}\longrightarrow \mathfrak{M}_{\Lambda}^0
  \end{equation}
  of IHSMs by the main result of \cite{Markman2}. There is a trivialization of the local system $R^2f_*\mathbb{Z}\rightarrow \underline{\Lambda}$ by loc.cit. and every local system over $\mathfrak{M}^0_{\Lambda}$ is trivial by \cite[Lemma 2.1]{Markman2}.
\end{Rem}
\begin{Def}
    A \textit{sheaf of Azumaya algebras} of rank $r$ on a variety $X$ is a sheaf $\mathcal{A}$ of $\mathcal{O}_X$-modules with a distinguished global section $1_{\mathcal{A}}$ and an associative multiplication map \[\mathcal{A}\otimes_{\mathcal{O}_X}\mathcal{A}\rightarrow \mathcal{A},\] which is locally (in the analytic or \'{e}tale topology) isomorphic to the matrix algebra $\cE nd(\cO_X^{\oplus r})$.
\end{Def}

\begin{Def}\label{deformation of Azumaya algebras}
    Let $X_1$ and $X_2$ be compact K\"{a}hler manifolds and let $\mathcal{A}_i$ be Azumaya algebras over the $X_i$. Then $\mathcal{A}_2$ is \textit{deformation equivalent} to $\mathcal{A}_1$ if there exists a proper family $\mathcal{X}\rightarrow B$ of compact K\"{a}hler manifolds over a connected analytic base $B$ such that there exist points $b_i\in B$ and isomorphisms $f_i: X_i\rightarrow X_{b_i}$, and an Azumaya algebra $\mathcal{A}$ over $\mathcal{X}$ such that $f^*_i\mathcal{A}\cong \mathcal{A}_i$ for each $i$.
\end{Def}

\begin{Def}\label{Mpsi}
Let $K$, $\widehat{K}$, $\cF^K$ and $\phi_\kappa=[\kappa(\cF^K)\sqrt{\operatorname{td}_{K\times\widehat{K}}}]_*$ be as in \S \ref{section 5 characteristic classes}. Fix a marking $\eta_K\colon H^2(K,\bZ)\to\Lambda$. Choose a parallel-transport operator $P\colon H^2(K,\bZ)\to H^2(\widehat{K},\bZ)$ and set $\eta_{\widehat{K}}:=\eta_K\circ P^{-1}$, so that $(K,\eta_K)$ and $(\widehat{K},\eta_{\widehat{K}})$ lie in the same connected component $\mathfrak M^0_\Lambda$ of $\mathfrak{M}_\Lambda$. Define
\begin{equation}\label{eq:psiLambda}
\psi_\Lambda:=\eta_{\widehat{K}}\circ\widehat H_0(\phi_\kappa)\circ\eta_K^{-1}\in O(\Lambda_{\bQ}),
\end{equation}
where $\widehat{H}_0$ is as in Definition~\ref{Hhat}. Let $\mathfrak{M}_\psi$ be the moduli space of isomorphism classes of quadruples $q=(X_1,\eta_1,X_2,\eta_2)$ of $\Lambda$-marked IHSMs such that
\[
\psi_q:=\eta_2^{-1}\circ\psi_\Lambda\circ\eta_1\colon H^2(X_1,\mathbb Q)\longrightarrow H^2(X_2,\mathbb Q)
\]
is a Hodge isometry which maps some K\"ahler class on $X_1$ to a K\"ahler class on $X_2$. Two quadruples $(X_1,\eta_1,X_2,\eta_2)$ and $(X_1',\eta_1',X_2',\eta_2')$ are isomorphic if the marked pairs $(X_i,\eta_i)$ and $(X_i',\eta_i')$ are isomorphic for $i=1,2$. The space $\mathfrak M_\psi$ is a non-Hausdorff complex manifold \cite[\S5.2]{Markman22}. We denote by $\mathfrak M^0_\psi$ a connected component of $\mathfrak M_\psi$.
\end{Def}

% \begin{Def}\label{moduli space of hodge isometries}
%  Let $E$ be the Fourier-Mukai kernel of nonzero rank of an equivalence $\Phi_E: D^b(X)\rightarrow D^b(Y)$ of derived equivalent IHSMs $X$ and $Y$. Let $\eta_X$ and $\eta_Y$ be respective markings such that $(X, \eta_X)$ and $(Y, \eta_Y)$ lie in the same connected component of the moduli space $\mathfrak{M}^0_{\Lambda}$. Let the moduli space 
% \begin{equation}
% \mathfrak{M}_{\psi}
% \end{equation}
% consist of isomorphism classes of quadruples $(X, \eta_X, Y, \eta_Y)$ such that the following holds:
% \begin{enumerate}

%     \item The map 
%     \begin{equation}\label{psi isometry}
%     \psi_{\Lambda}: = \tilde{H}_0([\kappa(E)\sqrt{td_{X\times Y}}]_*): \Lambda\longrightarrow \Lambda. 
%     \end{equation}
%     is a Hodge isometry, where $\tilde{H}_0$ denotes the map from (\ref{restriction of H tilde map}).

%     \item The Hodge isometry 
%     $$
%     \eta_2^{-1}\circ \psi_{\Lambda}\circ \eta_1: H^2(X, \mathbb{Q})\rightarrow H^2(Y, \mathbb{Q}) 
%     $$
%     maps a subcone of the K\"{a}hler cone of $X$ to a subcone of the K\"{a}hler cone of $Y$. 
% \end{enumerate}

% We denote by $\mathfrak{M}_{\psi}^0\subset \mathfrak{M}_{\psi}$ a connected component of the moduli space $\mathfrak{M}_{\psi}$.

% \end{Def}

Let
     \begin{equation}
         \Pi_i: \mathfrak{M}^0_{\psi}\longrightarrow \mathfrak{M}^0_{\Lambda}
     \end{equation}
     denote the morphism sending $(X, \eta_X, Y, \eta_Y)$ to $(X, \eta_X)$, respectively to $(Y, \eta_Y)$, where $\mathfrak{M}_{\Lambda}^0$ denotes the relevant connected component.

      \begin{Prop}[\cite{Markman22} Lemma 5.7]\label{projection maps surjective}
     For $i = 1, 2$, the morphisms $\Pi_i$ are surjective. 
 \end{Prop}
 
Markman's universal family (\ref{universal family}) yields universal familes $\Pi_i^*(f): \Pi_i^*\mathcal{Y}\rightarrow \mathfrak{M}^0_{\psi}$, $i = 1, 2$ and a universal fiber product
\begin{equation}\label{universal fiber product}
\Xi: \Pi_1^*\mathcal{Y}\times_{\mathfrak{M}^0_{\psi}}\Pi^*_2\mathcal{Y}\longrightarrow \mathfrak{M}^0_{\psi}.
\end{equation}
The triviality of all local systems over $\mathfrak{M}^0_{\Lambda}$ implies the triviality of the local systems $R\Xi_*\mathbb{Z}$ and $R\Xi_*\mathbb{Q}$ by the K\"unneth formula. 

\begin{Def}
    Let $\mathfrak{M}_{\psi}^{\prime}$ be the moduli space of isomorphism classes of quadruples $(X, \eta_X, Y, \eta_Y)$ of marked IHSMs such that 
    \[
    \eta^{-1}_Y\circ \psi_{\Lambda}\circ \eta_X: H^2(X, \bQ)\rightarrow H^2(Y, \bQ)
    \]
    is a Hodge isometry.
\end{Def}
Note that $\mathfrak{M}_{\psi}\subset \mathfrak{M}_{\psi}^{\prime}$ is an open subset where, in addition, $ \eta^{-1}_Y\circ \psi_{\Lambda}\circ \eta_X$ maps some K\"ahler class to some K\"ahler class.

\begin{Def}
    Let $W$ be a positive-definite three-dimensional subspace of $\Lambda_{\mathbb{R}}$, and let $W_{\mathbb{C}}: = W\otimes_{\mathbb{R}}\mathbb{C}$. We call $\mathbb{P}(W_{\mathbb{C}})\cap \Omega_{\Lambda}$ a \textit{twistor line}. Now let $(X, \eta)\in \mathfrak{M}^0_{\Lambda}$ be a marked pair and let $\omega\in H^{1, 1}(X, \mathbb{R})$ be a K\"ahler class. Let $\sigma$ be a holomorphic symplectic form on $X$. The subspace $W: = \langle \omega, \mathrm{Re}(\sigma), \mathrm{Im}(\sigma)\rangle \subset H^2(X, \bR)$ is three-dimensional and positive-definite. Define the twistor line 
    \begin{equation}\label{twistor line}
        Q_{(X,\, \eta,\, \omega)}: = \mathbb{P}(\eta(W_{\mathbb{C}}))\cap \Omega_{\Lambda}.
    \end{equation}

    Now let $X_1$ and $X_2$ be rationally Hodge-isometric IHSMs (not necessarily projective). Let $\omega_1$ be a K\"ahler class on $X_1$,  and consider $\omega_2: = \psi(\omega_1)$ the corresponding K\"ahler class on $X_2$. Let $W_i$ denote the three-dimensional subspace of  The isometry $\phi$ restricts to an isometry from $\eta(W_1)$ to $\eta(W_2)$, which is an isomorphism from $Q(X_1, \eta_1, \omega_1)$ to $Q(X_2, \eta_2, \omega_2)$. Define
    \begin{equation}\label{diagonal twistor line}
    Q_{(X_1, \, \eta_1, \, X_2, \, \eta_2, \, \omega_1)}\subset Q(X_1, \eta_1, \omega_1) \times Q(X_2, \eta_2, \omega_2)
    \end{equation}
    to be the graph of the restriction of $\psi$. We refer to $Q_{(X_1, \, \eta_1, \, X_2, \, \eta_2, \, \omega_1)}$ as the \textit{diagonal twistor line} associated to $(X_1, \eta_1, X_2, \eta_2, \omega_1)$. 
\end{Def}

\begin{Def}
    \begin{enumerate}
        \item Let $(X,\eta)$ and $(X^{\prime}, \eta^{\prime})$ be $\Lambda$-marked pairs in a connected component $\mathfrak{M}^0_{\Lambda}$ of the moduli space $\mathfrak{M}_{\Lambda}$. A \textit{twistor path} in $\mathfrak{M}^0_{\Lambda}$ from $(X,\, \eta)$ to $(X^{\prime},\, \eta^{\prime})$ consists of $\Lambda$-marked pairs $(X_i,\, \eta_i)$ for $0\leq i\leq k+1$ such that $(X_0, \eta_0) = (X, \, \eta)$ and $(X_{k+1},\, \eta_{k+1}) = (X^{\prime}, \eta^{\prime})$, and a twistor line $\tilde{Q}_{(X_i,\, \eta_i,\, \omega_i )}$ through $(X_i, \, \eta_i)$ and $(X_{i+1}, \, \eta_{i+1})$ for each $0\leq i \leq k$.

        \item Let $(X,\,\eta_X,\, Y,\, \eta_Y )$ and $(X^{\prime}, \eta_X^{\prime}, \, Y^{\prime},\, \eta^{\prime}_Y)$ be quadruples in a connected component $\mathfrak{M}^0_{\psi}$ of the moduli space $\mathfrak{M}_{\psi}$. A \textit{twistor path} in $\mathfrak{M}^0_{\psi}$ from $(X,\,\eta_X,\, Y,\, \eta_Y )$ to $(X^{\prime}, \eta_X^{\prime}, \, Y^{\prime},\, \eta^{\prime}_Y)$ consists of quadruples $(X_i, \eta_{X, i}, \, Y_i,\, \eta_{Y, i})$  for $0\leq i\leq k+1$ such that $(X_0, \eta_{X,0}, \, Y_0, \, \eta_{Y, 0}) = (X, \, \eta_{X}, \, Y, \eta_Y)$ and $(X_{k+1},\, \eta_{X, \,k+1}, \, Y_{k+1},\, \eta_{Y,\, k+1}) = (X^{\prime}, \eta_{X}^{\prime}, \, Y^{\prime}, \,  \eta^{\prime}_Y)$, and a twistor line $\tilde{Q}_{(X_i,\, \eta_{X,i},\, Y_i,\, \eta_{Y,i}\, \omega_i )}$ through $((X_{i},\, \eta_{X, \,i}, \, Y_{i},\, \eta_{Y,\, i})$ and $(X_{i+1},\, \eta_{X, \,i+1}, \, Y_{i+1},\, \eta_{Y,\, i+1})$ for each $0\leq i \leq k$.
    \end{enumerate}
\end{Def}

\begin{Def}
    A vector bundle on $X$ is called \textit{hyperholomorphic} if it carries a Hermitian connection whose curvature is of type $(1, 1)$ with respect to every induced complex structure on $X$.
\end{Def}
\begin{Def}
    We say a twistor path in $\mathfrak{M}^0_{\Lambda}$ (resp. $\mathfrak{M}^0_{\psi}$) is \textit{generic} if $\mathrm{Pic}(X_i)$ (resp. $\mathrm{Pic}(X_{1, i})$) is trivial for each $1\leq i\leq k$.
\end{Def}

Consider the pairs $(\Kum_{n}(A), \eta)$ and $(\Kum_{n}(\widehat{A}), \widehat{\eta})$, where $\eta$ and $\widehat{\eta}$ denote markings of $\Kum_{n}(A)$ and $\Kum_{n}(\widehat{A})$ respectively. 

\begin{Prop}[{\cite[Lemma 5.8]{Markman22}}]
    Any two points in $\mathfrak{M}^0_{\psi}$ are connected by a generic twistor path in $\mathfrak{M}_{\psi}^0$.
\end{Prop}

\begin{Prop}[{\cite[Theorem 5.1, Assumption 5.19, Definition 5.21, Proposition 5.22]{Markman22}}]\label{supports compatible vector bundle}

Let $\cE$, $(X, \eta_X, Y, \eta_Y)$, and $\psi_{\Lambda}$ be as in \eqref{eq:psiLambda}. Assume the following:
\begin{enumerate}
    \item $\cE$ is locally free
    \item There are smooth proper families $\cX\rightarrow S$ and $\cY\rightarrow S$ of compact K\"ahler manifolds over a simply connected base $S$ with fibers $X_0$ and $Y_0$ over a point $b\in S$. There is a possibly twisted locally free sheaf $\cE_S$ on $\cX\times_S \cY$ with $\cE nd(\cE_{S}|_{X_0\times Y_0}) \cong \cE nd(\cE)$. 

    \item There is a point $t\in S$ with the following properties: Let $\eta_{X_t}$ and $\eta_{Y_t}$ be markings obtained from $\eta_{X_0}$ and $\eta_{Y_0}$ by parallel transport. Then $\psi_t: = \eta_{Y_t}^{-1}\psi_{\Lambda} \eta_{X_t}$ maps some K\"ahler class of $X_t$ to some K\"ahler class of $Y_t$. Moreover, $E_t: = \cE_S|_{X_t\times Y_t}$ is slope stable with respect to $\pi^*_1\omega +\pi_2^*\psi_t(\omega)$ for all $\omega$ in a non-empty subcone of $\cK_{X_t}\cap \psi^{-1}_t(\cK_{Y_t})$.
\end{enumerate}
Let $\mathfrak{M}^0_{\psi}$ be a connected component of $\mathfrak{M}_{\psi}$ containing $(X_t, \eta_{X_t}, Y_{t}, \eta_{Y_t})$. Then every quadruple $q = (X, \eta_X, Y, \eta_Y)\in \mathfrak{M}^0_{\psi}$ supports a compatible vector bundle, i.e. a possibly twisted locally free sheaf $F$ on $X\times Y$ such that $\phi_F: = [\kappa(F)\sqrt{\td_{X\times Y}}]_*$ is a morphism in $\Hom_{\cG_{an}}(X, Y)$, and $H_0(\phi_F) = \eta^{-1}_Y\psi_{\Lambda}\eta_X$.

\end{Prop}

\begin{Rem}\label{remark that F^K satisfies assumptions}
    Note that the kernel $\cE = \cF^K$ satisfies the assumptions of Proposition \ref{supports compatible vector bundle}.
\end{Rem}

We can now prove our main theorem:

\begin{Th}\label{maintheorem}
Let $Z$ be a projective IHSM of generalized Kummer type of complex dimension $4d$, $d\geq1$. Then the Lefschetz Standard Conjecture \ref{BX} holds for $Z$.
\end{Th}

\begin{proof}
Let $n:=2d$, so that $\dim Z=2n$ and $n$ is even. Let $A$ an abelian surface with a polarization of exponent coprime to $n+1$, so by \cite[Theorems~6.18 and~7.27]{Magni} there exists a derived equivalence $\Phi_{\cF^K}\colon D^b(K)\to D^b(\widehat K)$, where $K=\Kum_n(A)$ and $\widehat K=\Kum_n(\widehat A)$. Its kernel $\cF^K$ is locally free of rank $R$ by Lemma \ref{local freeness}. Let $\eta_K$, $\eta_{\hat K}$, $\psi_\Lambda$ and $\mathfrak M^0_\Lambda$ be as in Definition \ref{Mpsi}, and let $B$, $\cX$, $\cY$ and $\cE$ be as in \S \ref{section 5 characteristic classes} and Corollary~\ref{existence twisted deformation corollary}.

Let $t\in B$ be as in Corollary \ref{missing-link}, and put $\eta_{X_t}:=\eta_K\circ T_{X,t}^{-1}$ and $\eta_{Y_t}:=\eta_{\hat K}\circ T_{Y,t}^{-1}$. Then
\[
\eta_{Y_t}^{-1}\circ\psi_\Lambda\circ\eta_{X_t}=T_{Y,t}\circ\widehat H_0(\phi_\kappa)\circ T_{X,t}^{-1}=\psi_t ,
\]
and $\psi_t$ is a Hodge isometry by Proposition \ref{hodge isometry psi of hodge type}(1) and maps K\"ahler classes of $X_t$ to K\"ahler classes of $Y_t$ by Corollary \ref{missing-link}(2). Hence $q_t:=(X_t,\eta_{X_t},Y_t,\eta_{Y_t})\in \mathfrak M_\psi$. Let $\mathfrak M^0_\psi$ be the connected component of $\mathfrak M_\psi$ containing $q_t$. By Remark \ref{remark that F^K satisfies assumptions}, the hypotheses of Proposition \ref{supports compatible vector bundle} hold for $E=\cF^K$, with $S=B$ and $\cE_S=\cE$, so every quadruple in $\mathfrak M^0_\psi$ supports a compatible vector bundle. Since $(X_t,\eta_{X_t})$ is joined to $(K,\eta_K)$ through $B$, the morphism $\Pi_1$ maps $\mathfrak M^0_\psi$ onto $\mathfrak M^0_\Lambda$, by Proposition \ref{projection maps surjective}.

Since $Z$ is of generalized Kummer type and of dimension $2n$, it is deformation equivalent to $K$. Parellel transport of $\eta_K$ along a family joining $K$ to $Z$ gives a marking $\eta_Z$ of $Z$ with $(Z,\eta_Z)\in\mathfrak M^0_\Lambda$. By Proposition \ref{projection maps surjective}, there is a quadruple $q=(Z,\eta_Z,Z',\eta')$ in $\mathfrak M^0_\psi$, and $\psi_q=\eta^{\prime -1}\circ\psi_\Lambda\circ\eta_Z$ is a rational Hodge isometry. The IHSM $Z'$ is projective. Indeed, since $Z$ is projective there is a class $\lambda\in H^{1,1}(Z,\bZ)$ with $q_Z(\lambda)>0$ \cite[Theorem 3.11]{Huybrechts2}. A positive multiple of $\psi_q(\lambda)$ is then an integral class of type $(1,1)$ on $Z'$ with positive square, so $Z'$ is projective.

By Proposition \ref{supports compatible vector bundle}, the quadruple $q$ supports a compatible vector bundle. Such a compatible vector bundle is a possibly twisted locally free sheaf $F$ on $Z\times Z'$ such that
\[
\phi_F:=\left[\kappa(F)\sqrt{\operatorname{td}_{Z\times Z'}}\right]_*\in\operatorname{Hom}_{\mathcal G_{an}}(Z,Z')
\qquad\text{and}\qquad
\widehat H_0(\phi_F)=\psi_q .
\]
Moreover, $\kappa(F)\sqrt{\operatorname{td}_{Z\times Z'}}$ is the flat deformation of $\kappa(\mathcal E_t)\sqrt{\operatorname{td}_{X_t\times Y_t}}$ along a path in $\mathfrak M^0_\psi$ \cite[proof of Proposition 5.22, Step 2]{Markman22}. By Corollary~\ref{existence twisted deformation corollary}, the latter is the flat deformation of $\kappa(F^K)\sqrt{\operatorname{td}_{K\times\hat K}}$ along $B$. Hence $\phi_F=T_2\circ\phi_\kappa\circ T_1$ for parallel-transport operators $T_1\colon H^*(Z,\mathbb Q)\to H^*(K,\mathbb Q)$ and $T_2\colon H^*(\widehat K,\mathbb Q)\to H^*(Z',\mathbb Q)$. Parallel-transport operators preserve the cohomological degree, and $\phi_\kappa$ is degree-reversing by Proposition \ref{degree reversing hodge isometries}, so $\phi_F$ is degree-reversing. Since $Z$ and $Z'$ are projective, Corollary~\ref{Derived equivalent LSC} applied to $\phi_F$ shows that Conjecture~\ref{BX} holds for $Z$.
\end{proof}

\section{Remarks on further progress on the standard conjectures for Kummer-type IHSMs}\label{section final remarks}
\begin{Rem}\label{putting it all together}
   Combining Theorem \ref{maintheorem} with the main theorem of \cite{Foster} yields the Lefschetz standard conjecture for IHSMs of generalized Kummer type of dimension $2n\geq 4$ in the following degrees:
    \[\begin{cases}
    \hspace{3mm}\text{all,} & n \text{ even}\\
   \hspace{3mm} <n+1, & n \text{ odd}.\\
    \end{cases}\]

\end{Rem}
 The integer $n$ being even appears to be a necessary condition for the existence of derived equivalences of generalized Kummer varieties $D^b(\Kum_n(A))\rightarrow D^b(\Kum_n(\widehat{A}))$ that arise from equivariant derived equivalences $D^b(A)\rightarrow D^b(\widehat{A})$ (see \cite{YuxuanPaper}). As a result, we believe that new ideas will be necessary to construct derived equivalences between generalized Kummer varieties. 
\begin{comment}
We provide a reinterpretation of the result in terms of morphisms of generalized Kummer varieties, and demonstrate that this interpretation shows Magni's conjecture. Let 
\[
f: A\longrightarrow \hat{A}
\]
denote an isogeny of abelian surfaces. The isogeny $f$ defines a rational map of Kummer varieties
\[
\tilde{f}: \Kum_n(A)\longrightarrow \Kum_n(\widehat{A}).
\]
\end{comment}

One possible future avenue would be to discover new examples of derived autoequivalences of generalized Kummer varieties whose induced Hodge isometry on total cohomology is degree reversing. Such examples follow from work in the $K3^{[2]}$ deformation type case in \cite{Addington16} and in the Kummer fourfold case in \cite{Meachan}. However, the result is based on the fact that a particular twisted relative extension sheaf constructed in \cite{Markman20} provides a derived equivalence between four-dimensional Kummer-type IHSMs. It is currently unknown if such examples exist for higher dimensional $K3^{[n]}$-type hyper-K\"{a}hler varieties. 

We sum up the results of this paper in the following theorem, which we hope will be helpful in proving more instances of the LSC:

\begin{Th}[{Theorem \ref{summing up paper results introduction}}]\label{Th: summing up paper results}
    Let $X$ and $Y$ be projective $2n$-dimensional deformation-equivalent IHSMs. Assume that if $n$ is even, then $b_2$ is odd. Assume that there exists a Fourier-Mukai equivalence $\Phi_\cF: D^b(X)\rightarrow D^b(Y)$ such that $\cF$ satisfies the following:
    \begin{enumerate}
        \item $\cF$ is locally free;
        \item $\cF$ is $\mu_{\omega}$-stable for some K\"ahler class $\omega$.
    \end{enumerate}
    Then the LSC holds for any projective variety deformation equivalent to $X$.
\end{Th}
\begin{proof}
    Since $\cF$ is locally free on $X\times Y$, $\rk(\cF)>0$. Set $\phi_{\kappa}: = \left[\kappa(\cF)\sqrt{\td_{X\times Y}}\right]_*: H^*(X, \bQ)\rightarrow H^*(Y, \bQ)$. Since $n$ is odd or $b_2$ is odd, the isometry $\tilde{H}(\phi_{\kappa})$ of \eqref{H tilde map} exists by \cite[Theorems 4.10 and 4.11]{Taelman}. Additionally, $\phi_{\kappa}$ is degree-reversing by Proposition \ref{degree reversing hodge isometries}.
    
    Furthermore, $\cF$ is semiregular by Corollary \ref{K-trivial corollary}, and therefore Theorem \ref{analytic Buchweitz-Flenner conjecture} holds for $\cF$. Hence, there exists a possibly twisted deformation $\tilde{\cF}$ of $\cF$ to a product $\tilde{X}\times \tilde{Y}$ such that $\NS(\tilde{X}\times \tilde{Y}) = 0$. Since $\cF$ is $\mu_{\omega}$-stable, $\tilde{\cF}$ does not contain a proper saturated subsheaf by Lemma \ref{slope stable with respect to every kahler class}.

    Fix markings $\eta_X$ and $\eta_Y$ and the isometry $\psi_{\Lambda}$, as in \eqref{eq:psiLambda}, and construct families $\cX\rightarrow B$ and $\cY\rightarrow B$ as in \S \ref{section 5 characteristic classes}. By Proposition \ref{hodge isometry psi of hodge type} and Remark \ref{Remark on kappa class equivalence}, $\kappa(\cF)$ remains of Hodge-type on $B$. By Corollary \ref{existence twisted deformation corollary}, $\cF$ deforms to a possibly twisted locally free sheaf $\cE$ on $\cX\times_B \cY$. By Corollaries \ref{m omega t stability} and \ref{missing-link}, there is a point $b\in B$ such that condition (3) of Proposition \ref{supports compatible vector bundle} holds. Therefore, as in Remark \ref{remark that F^K satisfies assumptions}, $\cF$ satisfies the hypotheses of Proposition \ref{supports compatible vector bundle}.

    Let $Z$ be a projective IHSM deformation equivalent to $X$. The triple $(X, Y, \cF)$ satisfies the assumptions of Theorem \ref{maintheorem}. Hence, it produces, by Proposition \ref{projection maps surjective}, a quadruple $(Z, \eta_Z, Z^{\prime}, \eta^{\prime})\in \mathfrak{M}_{\psi}^0$ with $Z^{\prime}$ projective. By Proposition \ref{supports compatible vector bundle}, this quadruple supports a compatible vector bundle $\cF^{\prime}$ such that $[\phi_{\kappa(\cF^{\prime})}]_*$ is degree-reversing. Corollary \ref{Derived equivalent LSC} therefore implies Conjecture \ref{BX} for $Z$. 
\end{proof}

\begin{Rem}
    The assumptions of Theorem \ref{Th: summing up paper results} imply that the sheaf $\cF$ is projectively hyperholomorphic over the connected component $\mathfrak{M}^0_{\psi}$ in the sense of \cite[Section 5]{Markman22}. 
\end{Rem}

\bibliographystyle{amsalpha}
\bibliography{bibliography}

@article{Addington16,
    author = {Addington, N.},
    title = {New derived symmetries of some hyperk\"{a}hler varieties},
    journal = {J. Alg. Geom. },
    year ={2016},
    number = {2},
    volume = {3},
    pages ={223--260}
    
}

@misc{ACLS,
    author = {Ancona, G. and  Cavicchi, M. and Laterveer, R. and Sacc\`{a}, G.},
    title = {Relative and absolute
{L}efschetz standard conjectures for some {L}agrangian fibrations},
     year={2023},
      eprint={arXiv:2304.00978v1 },
      archivePrefix={arXiv},
      primaryClass={math.AG},
howpublished = {\url{https://arxiv.org/abs/2304.00978}},
}

@book{Andre,
    author = {Andr\'{e}, Y.},
    title = {Une introduction aux motifs (motifs purs, motifs mixtes, p\'{e}riodes)},
    publisher = {Soci\'{e}t\'{e} Math\'{e}matique de France},
    volume = {17 in {P}anoramas et {S}ynth\`{e}thes [{P}anoramas and {S}yntheses]},
    year = {2004},
    city = {Paris}
}

@article{Arapura,
    author = {Arapura, D.},
    title = {Motivation for {H}odge cycles},
    journal = {Adv. Math.},
    year = {2006},
    volume = {207},
    number = {2}, 
    pages = {762--781}
}

@article{Artamkin,
    author = {Artamkin, V. I.},
    title = {On deformations of sheaves},
    journal = {Math. USSR Izvestiya},
    year = {1989}, 
    volume = {32}, 
    number = {3}, 
    pages = {663--668}
}

@article{Artin,
  author  = {Artin, M.},
  title   = {Algebraization of formal moduli. {II}. {E}xistence of modifications},
  journal = {Annals of Mathematics (2)},
  volume  = {91},
  year    = {1970},
  pages   = {88--135}
}

@article{Beauville,
    author = {Beauville, A.},
    title = {Vari\'{e}t\'{e}s {K}\"{a}hleriennes dont la priemi\`{e}re classe de {C}hern est nulle},
    journal = {J. Diff. Geom.},
    volume = {18},
    year = {1983}, 
    pages = {755--782}
}

@article{Beckmann,
    author = {Beckmann, T.},
    title = {Derived categories of hyper-{K}\"{a}hler manifolds: {E}xtended {M}ukai vector and integral structure},
    journal = {Compos. Math.},
    year = {2023},
    volume = {159}, 
    number = {1}, 
    pages = {109--152}
}

@article{BKR,
    title ={The {M}c{K}ay correspondence as an equivalence of derived categories},
    author ={Bridgeland, T. and  King, A. and Reid, M.} ,
    journal = {J. Amer. Math. Soc.},
    volume = {14}, 
    number = {},
    pages = {535--554},
    year = {2001},
}

@article{Buchweitz-Flenner, 
TITLE = {A semiregularity map for modules and applications to deformations},
author = {Buchweitz, R.-O. and Flenner, H.},
journal= {Compos. Math.},
volume = {137},
number = {},
pages = {135--210},
year = {2003},

}

@article{CRVdB12,
    author = {Calaque, D. and Rossi, C. and Van den {B}ergh, M.},
    title = {C\u{a}ld\u{a}raru's conjecture and {T}sygan's formality},
    journal = {Math. Ann.},
    year = {2012},
    volume = {176}, 
    pages = {865--923}
}

@phdthesis{Caldararu2,
    author = {C\u{a}ld\u{a}raru, A.},
    title = {Derived categories of twisted sheaves on {C}alabi-{Y}au manifolds},
    school = {Cornell University},
    year = {2000}
}

@article{Charles,
  author  = {Charles, F.},
  title   = {Remarks on the {L}efschetz standard conjecture and hyperk{\"a}hler varieties},
  journal = {Comment. Math. Helv.},
  volume  = {88},
  number  = {2},
  year    = {2013},
  pages   = {449--468}
}

@article{CharlesMarkman,
    author = {Charles, F. and Markman, E.},
    title = {The standard conjectures for holomorphic symplectic varieties deformation equivalent to {H}ilbert schemes of ${K}3$ surfaces},
    journal = {Compos. Math.},
    year = {2013},
    volume = {149},
    number = {3}, 
    pages = {481--494}    
}

@article{dCM02,
    author = {de {C}ataldo, M. A. and Migliorini, L.},
    title = {The hard {L}efschetz theorem and topology of semismall maps},
    journal = {Annales Scientifiques de l'École Normale Supérieure},
    year = {2002},
    volume = {35},
    number ={5},
    pages = {759--772}
}

@article{Deligne,
    author = {Deligne, P.},
    title = {Th\'eor\`em de {L}efschetz et crit\`eres de d\'eg\'en\'erescene de suites spetrales},
    journal = {Publications math\'ematiques de l'{I}.{H}.\'{E}.{S}.},
    volume = {35},
    year = {1968},
    pages = {107--126}
}

@misc{dJP,
      title={The period-index problem and {H}odge theory}, 
      author={de Jong, A. J. and Perry, A.},
      year={2022},
      eprint={2212.12971},
      archivePrefix={arXiv},
      primaryClass={math.AG},
howpublished = {\url{https://arxiv.org/abs/2212.12971}},
}

@article{EGAIV4,
  author  = {Grothendieck, A.},
  title   = {\'El\'ements de g\'eom\'etrie alg\'ebrique. {IV}. \'Etude locale des sch\'emas et des morphismes de sch\'emas, {Quatri\`eme} partie},
  note    = {R\'edig\'e avec la collaboration de Jean Dieudonn\'e},
  journal = {Inst. Hautes \'Etudes Sci. Publ. Math.},
  volume  = {32},
  year    = {1967},
  pages   = {5--361},
}

@article{KempfLemma,
    author = {Drezet, J.-M. and Narasimhan, M.S.},
    title = {Group de {P}icard des vari\'et\'es de modules de fibr\'es semi-stables sur les courbes alg\'ebriques},
    journal = {Invent. Math.},
    year = {1989},
    volume = {97},
    number = {1}, 
    pages = {53-94}
}

@article{Foster,
  title={The {L}efschetz standard conjectures for {IHSM}s of generalized {K}ummer deformation type in certain degrees},
  author={Foster, J.},
  journal={Eur. J. Math.},
  year={2024},
    volume = {10},
    pages = {},
    number = {34}
}

@book{Gr,
    author = {Grothendieck, A.},
    title = {Standard conjectures on algebraic cycles},
    publisher = {Oxford {U}niversity {P}ress, {London}},
    pages = {193--199},
    series = {Algebraic Geometry (Internat. Colloq.,
Tata Inst. Fund. Res., Bombay, 1968)},
    year = {1969}
}

@article{Haiman99,
    title = {Mac{D}onald positivity and geometry},
    author = {Haiman, M.},
    journal = {Math. Sci. Res. Inst. Publ.},
    year = {1999},
    pages = {207--254},
    volume = {38}
}

@article{Haiman01,
    author = {Haiman, M.},
    title = {Hilbert schemes, polygraphs, and the {M}ac{D}onald positivity conjecture},
    journal = {J. Amer. Math. Soc.},
    volume = {14},
    number = {4},
    year = {2001},
    pages = {941--1006}
}

@article{HT,
    author = {Hassett, B and Tschinkel, Y.},
    title = {Hodge theory and Lagrangian planes in generalized Kummer fourfolds},
    journal = {Mosc. Math. J.},
    number = {1},
    volume = {13},
    year = {2013},
    pages = {33--56},
}

@article{Huang,
    author = {Huang, S.},
    title = {A note on a question of {M}arkman},
    journal = {J. Pure Appl. Algebra},
    year = {2021},
    volume = {225},
    number = {9}
    
    
}

@article{Huybrechts1,
    title ={A global {T}orelli theorem for hyperk\"{a}hler manifolds [after {M}. {V}erbitsky]},
    author ={Huybrechts, D.} ,
    journal = {S\'{e}minaire Bourbaki},
    volume = {2010/2011}, 
    number = {348},
    pages = {375--403},
    year = {2012},
}

@article{Huybrechts2,
    author = {Huybrechts, D.},
    title ={Compact hyperk\"{a}hler manifolds: {B}asic results} ,
    journal = {Invent. Math.},
    volume ={135},
    number = {1},
    pages = {63--113},
    year = {1999},
}

@book{Huybrechts3,
    author = {Huybrechts, D.},
    title = {Fourier-{M}ukai {T}ransforms in {A}lgebraic {G}eometry},
    publisher = {Oxford {U}niversity {P}ress},
    year = {2006} 
}

@inproceedings{Kleiman68,
booktitle = {Dix expos\'{e}s sur la cohomology des sch\'{e}mas},
  title={Algebraic cycles and the {W}eil conjectures},
  author={Kleiman, S.},
  year={1968},
    publisher = {Adv. Stud. Pure Math.},
    volume = {3},
    city = {Amsterdam, North Holland},
    pages = {359--386}
    
}

@article{Kleiman94,
    author = {Kleiman, S.},
    title = {The standard conjectures},
    journal = {Proc. Sym. in Pure Math.},
    volume = {55},
    pages = {3--20},
    year = {1994}
}

@article{Krug18,
    author = {Krug, A.},
    title = {Remarks on the derived {M}c{K}ay correspondence for {H}ilbert schemes of points and tautological bundles},
    journal = {Math. Ann.},
    volume = {371},
    number = {1-2},
    year = {2018},
    pages = {461--486}
}

@book{LazarsfeldI,
    author = {Lazarsfeld, R.},
    title = {Positivity in {A}lgebraic {G}eometry {I}},
    publisher = {Springer},
    year = {2004}
}

@article{Lieberman,
    author = {Lieberman, D. I.},
    title = {Higher {P}icard varieties},
    journal = {Amer. J. Math.},
    year = {1968}, 
    volume = {90}, 
    number = {4}, 
    pages = {1165--1199}
}

@article{Lie07,
  author  = {Lieblich, Max},
  title   = {Moduli of twisted sheaves},
  journal = {Duke Math. J.},
  volume  = {138},
  number  = {1},
  year    = {2007},
  pages   = {23--118}
}

@article{LL,
    author = {Looijenga, E. and Lunts, V.},
    title = {A {L}ie algebra attached to a projective variety},
    journal = {Invent. Math.},
    year = {1997},
    volume = {129}, 
    pages = {361--412},
}

@book{LubkeTeleman,
    author = {L\"ubke, M. and Teleman, A.},
    title = {The {K}obayashi {H}itchin {C}orrespondence},
    publisher = {World {Scientific} {Publishing} {Co.}},
    year = {1995}
}

@misc{Magni,
      title={Derived equivalences of generalized {K}ummer varieties}, 
      author={Magni, P.},
      year={2022},
      eprint={},
      archivePrefix={arXiv},
      primaryClass={math.AG},
howpublished = {\url{https://arxiv.org/abs/2208.11183}},
}

@article{Markman11,
    author = {Markman, E.},
    title = {A survey of {T}orelli and monodromy results for holomorphic
symplectic varieties},
    journal = {Complex and Differential Geometry, Springer Proc.
Math.},
    year = {2011},
    pages = {257--322}
}

@article{Markman20,
    author = {Markman, E.},
    title = {The {B}eauville-{B}ogomolov class as a characteristic class},
    journal = {J. Alg. Geom.},
    year ={2020},
    volume = {29},
    pages = {199--245}
}

@misc{Markman1,
      title={Stable vector bundles on a hyper-k\"{a}hler manifold with a rank 1 obstruction map are modular}, 
      author={Markman, E.},
      year={2021},
      eprint={2107.13991},
      archivePrefix={arXiv},
      primaryClass={math.AG},
howpublished = {\url{https://arxiv.org/abs/2107.13991}},
}

@article{Markman2,
  TITLE={On the existence of universal families of marked irreducible holomorphic symplectic manifolds},
  author={Markman, E.},
  journal={Kyoto J. Math.},
  volume={61},
  number={1},
  pages={207--223},
  year={2021},
}

@article{Markman22,
    author = {Markman, E.},
    title = {Rational {H}odge isometries of hyper-{K}\"{a}hler varieties of ${K}3^{[n]}$-type are algebraic},
    journal = {Compos. Math.},
    year = {2024}, 
    pages = {1261-1303},
    volume = {160},
    number = {6}
}

@misc{Markman2n,
    title = {Cycles on abelian $2n$-folds of {W}eil type from secant sheaves on Abelian $n$-folds},
    author = {Markman, E.}, 
    year = {2025},
    archivePrefix={arXiv},
    primaryClass={math.AG},
    eprint = {2502.03415},
    howpublished = {https://arxiv.org/abs/2502.03415} 
}

@article{Meachan,
    author = {Meachan, C.},
    title = {Derived autoequivalences of generalised {K}ummer varieties},
    journal = {Math. Res. Let.},
    year = {2015},
    volume = {22},
    number = {4}
}

@article{Mukai1,
    title ={Semi-homogeneous vector bundles on an abelian variety},
    author ={Mukai, S.} ,
    journal = {J. Math. Kyoto Univ.},
    volume = {18}, 
    number = {2},
    pages = {239--272},
    year = {1978},
}

@article{OG10,
    author = {O'Grady, K.},
    title = {Desingularized moduli spaces of sheaves on a {K}3},
    journal = {J. Reine Angew. Math.},
    year = {1999}, 
    volume = {519},
    pages = {49--117}
    
}

@article{OG6,
    author = {O'Grady, K.},
    title = {A new six-dimensional irreducible symplectic variety},
    journal = {J. Alg. Geom.},
    volume = {12},
    number = {3},
    year = {2003},
    pages = {435--505}
}

@article{OGModularK32,
    author = {O'Grady, K.},
    title = {Modular sheaves on hyperk\"{a}hler varieties},
    journal = {Algebr. Geom.},
    year = {2022},
    number ={1}, 
    pages ={1--38}
}

@article{OGKum,
    author = {O'Grady, K.},
    title = {Rigid stable rank 4 vector bundles on {HK} fourfolds of {K}ummer type},
    journal = {\'{E}pijournal de {G}\'{e}ometri\'{e} {A}lg\'{e}brique},
    year = {2024},
    number = {19},
    pages ={1--51}
}

@article{Orlov,
    author ={Orlov, D.} ,
    title = {Derived categories of coherent sheaves on Abelian varieties and equivalences between
them},
    journal = {Izv. Math.},
    year = {2002},
    volume = {66},
    number = {3}, 
    pages = {569--594}
}

@misc{Perry,
    author = {Perry, A.},
    title = {The semiregularity theorem for equivariant noncommutative varieties},
    primaryclass = {math.AG},
    howpublished = {\url{https://arxiv.org/pdf/2604.00511}},
    year = {2026}
 
}

@article{Ploog,
    author = {Ploog, D.},
    title = {Equivariant autoequivalences for finite group actions},
    journal = {Adv. Math.},
    volume = {216},
    number = {1},
    year = {2007},
    pages = {62-74}
}

@misc{Pridham,
    author = {Pridham, J.},
    title = {Semiregularity as a consequence of {G}oodwillie's theorem},
    primaryclass = {math.AG},
    howpublished = {\url{https://arxiv.org/abs/1208.3111v4}},
    year = {2024}
}

@misc{Rienks,

    author = {Rienks, W.},
    title = {On the deformation theory of {F}ourier-{M}ukai transforms between {C}alabi-{Y}au varieties},
    primaryclass = {math.AG},
    howpublished = {\url{https://arxiv.org/pdf/2408.06188}},
    year = {2024} 
}

@book{SGA1,
  author    = {Grothendieck, A.},
  title     = {Rev\^etements \'etales et groupe fondamental ({SGA} 1)},
  series    = {Lecture Notes in Mathematics},
  volume    = {224},
  publisher = {Springer-Verlag},
  address   = {Berlin--New York},
  year      = {1971},
  note      = {S\'eminaire de G\'eom\'etrie Alg\'ebrique du Bois-Marie 1960--1961. Augment\'e de deux expos\'es de M. Raynaud. Expos\'e XII: \emph{G\'eom\'etrie alg\'ebrique et g\'eom\'etrie analytique}, par M. Raynaud},
}

@misc{Stacks,
  author       = {{{S}tacks {P}roject {A}uthors}},
  title        = {{S}tacks {P}roject},
  howpublished = {\url{https://stacks.math.columbia.edu}},
  year         = {2026}
}

@article{Stapleton,
    author = {Stapleton, D.},
    title = {Geometry and stability of tautological bundles on {H}ilbert schemes of points},
    journal = {Alg. and No. Theory},
    year = {2016},
    volume = {10},
    number ={6},
    pages = {1173--1190}
}

@article{Taelman,
    author = {Taelman, L.},
    title = {Derived equivalences of hyperk\"{a}hler varieties},
    journal = {Geom. Top.},
    volume = {27},
pages = {2649--2693},
    year = {2023}
}

@article{TodaFM,
  TITLE={Deformations and {F}ourier-{M}ukai transforms},
  author={Toda, Y.},
  journal={J. Diff. Geom.},
  volume={81},
  number={1},
  pages={197--224},
  year={2009},
}

@article{Verbitsky1,
    author = {Verbitsky, M.},
    title = {Cohomology of compact hyper-{K}\"{a}hler manifolds and its applications},
    journal = {Geom. Funct. Analysis},
    year = {1996},
    volume = {6},
    number = {4},
    pages = {601--611},
    
}

@inbook{VerbitskyHH,
    author = {Verbitsky, M.},
    title = {Hyper-K\"{a}hler manifolds},
    publisher = {International Press, Somerville MA},
    year = {1999},
    chapter = {Hyperholomorphic sheaves and new examples of hyperkaehler manifolds}
}

@article{VerbitskyTorelli,
    author = {Verbitsky, M.},
    title = {Mapping class group and a global {T}orelli theorem for hyper-{K}\"{a}hler manifolds},
    journal = {Duke Math. J.},
    year = {2013},
    volume = {162},
    pages = {2929--2986}
}

@book{VoisinHT,
    author = {Voisin, C.},
    title = {Hodge {T}heory and {C}omplex {A}lgebraic {G}eometry {II}},
    publisher = {Cambridge University Press},
    year = {2003}
}

@inbook{VoisinHodge,
    author = {Voisin, C.},
    title = {Open {P}roblems in {M}athematics},
    editors = {Nash, J.F., Rassias, M.},
    publisher = {Springer},
    year = {2016},
    pages = {521-543}
}

@phdthesis{WrayThesis,
    author = {Wray, A.},
    title = {Moduli spaces of twisted {H}ermite-{E}instein connections over {K}3 surfaces},
    school = {University of {O}regon},
    year = {2020}
}

@misc{YuxuanPaper,
    title = {Lifting derived equivalences of abelian surfaces to generalized {Kummer} varieties},
    author = {Yang, Y.}, 
    year = {2025},
    archivePrefix={arXiv},
    primaryClass={math.AG},
    eprint = {2507.11358v2},
    howpublished = {https://arxiv.org/abs/2507.11358} 
}

@article{YoshiokaAbelian,
    author = {Yoshioka, K.},
    title = {Moduli spaces of stable sheaves on abelian surfaces},
    journal = {Math. Ann.},
    volume = {321},
    number = {4},
    pages = {817--884},
    year = {2001},
}
\vspace{3mm}

\textsc{Department of mathematics, University of Oregon, Eugene OR, 97403, USA}\par\nopagebreak

\textit{Email address:} \href{email}{jrfos@uoregon.edu}

\end{document}